\documentclass[reqno]{compositio}
\usepackage[english]{babel}
\usepackage{pifont}
\usepackage[latin1]{inputenc}
\usepackage[notref,notcite]{}
\usepackage[T1]{fontenc}
\usepackage{MnSymbol}
\usepackage{amsmath}
\usepackage{amsfonts}
\usepackage{amsthm}
\usepackage{amscd}
\usepackage{upgreek}
\usepackage{amssymb}
\usepackage[all]{xy}
\usepackage{graphicx}
\usepackage[usenames,dvipsnames]{color}
\usepackage{mathrsfs}
\usepackage{caption}
\usepackage{tikz-cd} 
\usepackage{braket}
\usepackage{mathrsfs}
\usepackage[
    colorlinks = true,
    linkcolor = black,
    urlcolor  = black,
    citecolor = blue,
    bookmarksopen = true,
    backref = page
]{hyperref}

\usepackage{comment}
\usepackage{todonotes}
\newcommand{\Mitya}[1]{\todo[size=\tiny,inline,color=lightgray]{#1 \\ \hfill --- Mitya}}
\theoremstyle{plain}
\newtheorem{thm}{Theorem}[section]
\newtheorem{cor}[thm]{Corollary}
\newtheorem{lem}[thm]{Lemma}
\newtheorem{prop}[thm]{Proposition}

\theoremstyle{definition}

\newtheorem{eg}[thm]{Example}
\newtheorem{rmk}[thm]{Remark}

\renewcommand{\ge}{\geqslant}
\renewcommand{\le}{\leslant}

\newcommand{\field}[1]{\mathbb{#1}}
\newcommand{\Q}{\field{Q}}
\newcommand{\C}{\field{C}}

\newcommand{\R}{\field{R}}
\newcommand{\M}{\text{M}}
\newcommand{\N}{\field{N}}
\newcommand{\Z}{\field{Z}}
\newcommand{\A}{\field{A}}
\newcommand{\F}{\field{F}}
\newcommand{\p}{\field{P}}
\newcommand{\G}{\field{G}}

\newcommand{\proofstep}[1]{%
  \medskip
  \noindent\emph{#1.}
}

\DeclareMathOperator{\End}{End}
\DeclareMathOperator{\Aut}{Aut}
\DeclareMathOperator{\rank}{rank}
\DeclareMathOperator{\Lie}{Lie}
\DeclareMathOperator{\ord}{ord}

\DeclareMathOperator{\Gal}{Gal}

\DeclareMathOperator{\Hom}{Hom}
\DeclareMathOperator{\id}{id}
\DeclareMathOperator{\rk}{rk}

\DeclareMathOperator{\Spec}{Spec}
\DeclareMathOperator{\Sp}{Sp}

\DeclareMathOperator{\Frob}{Frob}

\DeclareMathOperator{\pgl}{PGL}
\DeclareMathOperator{\PSL}{PSL}
\DeclareMathOperator{\pic}{Pic}

\DeclareMathOperator{\gl}{GL}

\DeclareMathOperator{\Stab}{Stab}
\DeclareMathOperator{\pr}{pr}
\DeclareMathOperator{\Fix}{Fix}

\newcommand{\acal}{\mathcal A}
\newcommand{\cala}{\mathscr A}

\newcommand{\calb}{\mathcal B}

\newcommand{\ccal}{\mathscr C}
\newcommand{\cald}{\mathscr D}

\newcommand{\calf}{\mathcal F}

\newcommand{\call}{\mathscr L}

\newcommand{\calm}{\mathscr M}

\newcommand{\ncal}{\mathscr N}
\newcommand{\calo}{\mathscr O}

\newcommand{\calr}{\mathcal R}

\newcommand{\caly}{\mathscr Y}
\newcommand{\calx}{\mathscr X}
\newcommand{\calz}{\mathcal Z}
\newcommand{\gota}{\mathfrak a}

\newcommand{\gotd}{\mathfrak d}

\newcommand{\gotm}{\mathfrak m}
\newcommand{\gotn}{\mathfrak n}

\renewcommand{\ge}{\geqslant}
\renewcommand{\le}{\leqslant}

\newcommand{\interior}[1]{%
	{\kern0pt#1}^{\mathrm{o}}%
}
\newtheorem{innercustomgeneric}{\customgenericname}
\providecommand{\customgenericname}{}
\newcommand{\newcustomtheorem}[2]{%
	\newenvironment{#1}[1]
	{%
		\renewcommand\customgenericname{#2}%
		\renewcommand\theinnercustomgeneric{##1}%
		\innercustomgeneric
	}
	{\endinnercustomgeneric}
}

\newcustomtheorem{customthm}{Theorem}

\numberwithin{equation}{section}

\title{The Pila--Zannier strategy for Drinfeld modules and Drinfeld modular curves}
\author{Gal Binyamini}
\email{gal.binyamini@weizmann.ac.il}
\address{Department of Mathematics, Weizmann Institute of Science, Israel}

\author{Dmitry Novikov}
\email{dmitry.novikov@weizmann.ac.il}
\address{Department of Mathematics, Weizmann Institute of Science, Israel}

\author{Francesco Maria Saettone}
\email{francesco.saettone@weizmann.ac.il}
\address{Department of Mathematics, Weizmann Institute of Science, Israel}

\subjclass{11G09; 11G93, 14G05}
\makeindex
\begin{document}

\begin{abstract}
We extend the Pila--Zannier strategy to Drinfeld modules: we prove analogues of the Manin--Mumford theorem
for a product of two Drinfeld modules of equal rank, and of the Andr\'e--Oort theorem  for a product of two Drinfeld modular curves. In characteristic zero, several steps of this strategy rest on $o$-minimality, which has no counterpart over a function field; we replace the counting step by the rigid analytic Pila--Wilkie theorem of Binyamini--Kato, and this appears to be its first arithmetic application. The functional transcendence input, namely an analogue of the Ax--Lindemann theorem in both settings, is established here by an independent point counting argument.
\end{abstract}
  \maketitle
  \setcounter{tocdepth}{2}
	\tableofcontents
    
\section{Introduction}

\medskip
\noindent\textbf{The Pila--Zannier strategy over function fields.}\quad
The celebrated strategy introduced by Pila and Zannier in \cite{pz}, originally to
reprove the Manin--Mumford conjecture, has become the standard approach to
statements of Andr\'e--Oort and, more broadly, Zilber--Pink type. In outline
it runs as follows. The special points one wishes to understand are the
images, under a transcendental uniformisation, of points that are rational,
or of small height, in the uniformizing space. A counting theorem of
Pila--Wilkie  \cite{pilawilkie} bounds the number of such rational
points on the transcendental part of a definable set by $c(\varepsilon)
H^\varepsilon$; a functional transcendence theorem of Ax--Lindemann type
identifies the algebraic part as the preimage of a special subvariety; and
an arithmetic estimate bounds the Galois orbit of a special point from
below, by a fixed power of its complexity. Were a subvariety to carry a Zariski dense set of special points without itself being special, the
counting bound and the Galois lower bound would stand in direct conflict,
and it is this contradiction that forces the subvariety to be special.

The counting theorem is the step that, in characteristic zero, stems from $o$-minimality. Over a function field of positive
characteristic no such theory is available.
The rigid analytic counting theorem of the first author and Kato \cite{rigidpilawilkie} supplies a missing ingredient: a Pila--Wilkie type bound for rigid-analytic sets, proved by the methods of rigid geometry in place of $o$-minimality, together with a
refinement in blocks that furnishes the family version one needs in
applications. One of our purposes here is to show that this theorem  perfectly matches its classical counterpart  \cite{pilawilkie} in the Pila--Zannier strategy for Drinfeld modules, and thus provides its first application.
	
\medskip
\noindent\textbf{Main results.}\quad
We phrase our results uniformly, after Pila \cite{pilaAO}, across two
parallel situations.

Throughout, let $q$ be a power of a prime $p$. Let $A$ be the ring of elements of $F=\F_q(T)$
which are regular outside the place $\infty$ of $F$ at which $T$ has a pole. Let $F_\infty=\F_q(\!(1/T)\!)$ be the completion of $F$ at $\infty$ , and let $\C_\infty$ be the completion of an algebraic closure of $F_\infty$. 
By a \emph{uniformized variety} we mean a pair $(S,\pi)$, consisting of a
$\C_\infty$-variety $S$ and a transcendental uniformisation $\pi$. In both of
the situations the underlying variety we consider is the same, namely
\[
S\;\simeq\;\A^2_{\C_\infty}
\]
and the two cases are distinguished by the analytic source of $\pi$, by the
uniformising map itself, and by the structure of an Anderson $A$-module in case~(D), and of a coarse moduli space in case~(H):
\begin{itemize}
\item[\textnormal{(D)}] \emph{(Drinfeld module)} $\cald$ and $\cald'$ are
Drinfeld $A$-modules over $\C_\infty$ of the same rank, and $S=\cald\times\cald'$
is their product, regarded as an Anderson $A$-module with product action $\Phi=\varphi\oplus\varphi'$; its underlying variety is
$\A^2_{\C_\infty}$. The uniformisation
$\pi=\exp_\Phi\colon\G_{a,\C_\infty}^2\rightarrow S(\C_\infty)$ is the Drinfeld
exponential.
\item[\textnormal{(H)}] \emph{(hyperbolic)} $S=Y(1)\times Y(1)$ is a product of
two Drinfeld modular curves, its underlying variety being identified with
$\A^2_{\C_\infty}$ through the $j$-function. The uniformisation
$\pi=\boldsymbol{j}=j\times j\colon\Omega^2\rightarrow\A^2_{\C_\infty}$ is the
product of two copies of the $j$-function on the Drinfeld upper half plane
$\Omega$.
\end{itemize}
Each situation carries a notion of \emph{special point} of $S$, and of
\emph{weakly special} and \emph{special} subvariety. The special points are
the torsion points in case~(D) and the CM points in case~(H). In case~(D), weakly special subvarieties are translates of
sub-$\Phi$-modules, where both $\{0\}$ and $S$ itself are
allowed as sub-$\Phi$-modules. They are \emph{special},
also called \emph{torsion subvarieties}, when the translating
point can be chosen torsion.
In case~(H), weakly special subvarieties are all points of $S$,
$S$ itself, and the embedded curves
$\A^1_{\C_\infty}\times\{x\}$,
$\{y\}\times\A^1_{\C_\infty}$, and $Y_0'(N)$,
where $x,y\in\C_\infty$ and $N\in A-\{0\}$.
Here $Y_0'(N)\subset S$ is the image of $Y_0(N)$ under
the source and target maps.
The special subvarieties in case~(H) are $S$, the curves
$Y_0'(N)$, the coordinate curves with CM fixed coordinate,
and the points whose two coordinates are CM.

We now state our Andr\'e--Oort--Manin--Mumford result, uniform across the two situations.

\begin{customthm}{A}
Let $(S,\pi)$ be a uniformized variety as in {\rm(D)} or {\rm(H)}, and in case {\rm(H)} assume that $q$ is odd. An irreducible subvariety $V\subseteq S$ contains a Zariski dense set of special points if and only if $V$ is special.
\end{customthm}

In case~(D) this is a Manin--Mumford statement, in case~(H) an Andr\'e--Oort statement, and neither is new: case~(D) is within reach of the work of Denis
\cite{denismaninmumford} and Scanlon \cite{scanlon} (with a minor difference: we consider a product of {\em distinct} Drinfeld modules), and case~(H), for $q$ odd, is a theorem of Breuer \cite{breuerAO}. We underline that the restriction to odd $q$ in case~{\rm(H)} is a restriction of the present
method, as explained in  Remark~\ref{r:odd-restriction}: in characteristic $2$, the (ramified at $\infty$) Artin--Schreier case introduces a non-uniform factor in the height estimate for representatives in the fundamental domain. What is new is that both follow from one and the same method, and that they are the first arithmetic consequences we are aware of drawn from the rigid-analytic counting theorem of \cite{rigidpilawilkie}. We regard the uniform proof, and the functional transcendence it rests on, as the main point of the paper.
\\

One main input for the proof is an Ax--Lindemann--Weierstrass theorem which takes one and the same form in the
linear and the hyperbolic situation.

\begin{customthm}{B}
Let $(S,\pi)$ be a uniformized variety as in {\rm(D)} or {\rm(H)}. For an
algebraic subvariety $V\subseteq S$, let $W$ be a maximal irreducible
algebraic subvariety of $\pi^{-1}(V)$. Then $\pi(W)$ is weakly special.
\end{customthm}

In case~(D) this is proved below as Theorem~\ref{alw}, in case~(H) as Theorem~\ref{p:hyperbolicAL}. We prove both from scratch, using lattice point counting, the structure of $p$-polynomials, and the classification of algebraic correspondences on $\pgl_2$. In particular, the functional transcendence input is obtained directly, without using a Gauss--Manin connection or differential-algebraic machinery.

Two features of these proofs have no counterpart in characteristic zero. The
first is that the equal rank hypothesis in case~(D) is a genuine necessity
rather than a convenience, as shown in Example~\ref{eg:different-ranks}. This phenomenon is a reminder that the
Frobenius map is never far away.

The second is the role played by the Frobenius map $\Frob_p\colon x\mapsto x^p$ in the hyperbolic
case. There one must understand the algebraic correspondences on
$Y(1)\times Y(1)$. In the proof, such a correspondence gives rise to a
Zariski closed subgroup $\Sigma\subseteq \pgl_{2}(\C_\infty)^2$,
which is the graph of an automorphism of the abstract group
$\pgl_2(\C_\infty)$. Over the complex numbers the relevant automorphisms are
inner, and the correspondence is governed by an ordinary conjugation. Over
$\C_\infty$ a new possibility arises: an abstract automorphism may be twisted
by a field automorphism, and the algebraicity of $\Sigma$ forces such a twist, in
the situation at hand, to be an integer  power $\Frob_p^m$ of the Frobenius. After switching the two coordinates if necessary,
we may write this as $\Frob_p^m$ with $m\ge0$.

The correspondence is thus a priori governed by
$\operatorname{Inn}(\delta)\circ\uptau^m$ rather than by an honest
conjugation, and ruling out the case $m>0$ is the content of
Proposition~\ref{p:nofrob}. The proof uses a counting argument. If $m>0$, such a twist would force a
finite-index subgroup of $\pgl_2(\F_q[T])$ into a conjugate of
$\pgl_2(\F_q[T^{p^m}])$. This is impossible: after restricting to the unipotent subgroup, one would have to fit a finite-index additive subgroup of $\F_q[T]$ inside the much sparser subring $\F_q[T^{p^m}]$.

\medskip
\noindent\textbf{Organization.}\quad
Section~2 develops the analytic and arithmetic machinery on which the proof
rests. Although it starts from standard foundations, much of it consists of
estimates established here for the purpose. Its first part recalls definitions and basic properties of Drinfeld modules and their rigid analytic uniformisation, and then proves the inputs the counting
argument requires: that the exponential preimages of torsion points lie in a
single affinoid ball (Lemma~\ref{l:torsionball}), a lower bound
for the Galois action on torsion (Lemma~\ref{masserboundtorsion}, actually based on \cite{breuertorsionbound}), and a
comparison between the height of such a preimage and the order of the
corresponding torsion point (Lemma~\ref{l:heightorder}).

Its second part, on Drinfeld modular curves, recalls the notions of the Drinfeld upper half
plane, the $j$ function, the CM theory of rank-$2$ Drinfeld modules, and the
Taguchi height---the Drinfeld analogue of the Faltings height---and establishes the geometric and height-theoretic estimates we use: a covering of the truncated fundamental domain by closed affinoid balls of radius $q^{-1}$(Lemma~\ref{l:cover}),
height bounds for (separable) quadratic points in the fundamental domain (Lemma~\ref{l:heightfundamentaldomain}),
and, for $q$ odd, a Colmez-type bound for the Taguchi height of a CM Drinfeld module (Lemma~\ref{colmezbound}). The principal result of this part is Proposition~\ref{p:taguchi-weil}, a Drinfeldian analogue of Silverman's comparison \cite{silv} between the Faltings height of an elliptic curve and the Weil height of its $j$-invariant: it relates the Taguchi height of a rank-$2$ Drinfeld module to the Weil height of its $j$-invariant, with an error term logarithmic in the Weil height. We use it here to bound the height of a CM point, and the estimate is of independent interest and we expect to apply it beyond the scope of the present paper.

Section~3 carries out the Pila--Zannier strategy and is divided into three
parts. The first proves the Ax--Lindemann--Weierstrass theorem of Theorem~B,
which we already discussed. The second states the rigid analytic Pila--Wilkie theorem of the first author and Kato \cite{rigidpilawilkie} in the form we use, namely its
overconvergent version and the refinement in blocks
(Theorem~\ref{t:blockcounting}). The third and final part combines these with the height and
Galois orbit estimates of Section~2 to prove Theorem~A: the Manin--Mumford
case~(D) first, and then the Andr\'e--Oort case~(H).

%\medskip
%\noindent\textbf{Synopsis.}\quad

\medskip
\noindent\textbf{Comparison with the classical strategy.}\quad
Beyond the characteristic-$p$ features of the Ax--Lindemann theorem noted
above, the way the strategy is assembled departs from the classical arguments
in two further, essential respects; for the latter we refer to Scanlon's
survey \cite{scanlonsurvey}. Both departures trace back to the absence of the
real (``Betti'') structure on which $o$-minimality rests.

The first is already visible for the Manin--Mumford theorem, and it concerns the counting
space itself. In the classical setting the real uniformisation $\R^{2g}/\Z^{2g}\to A(\C)$ of an abelian variety furnishes Betti coordinates:
a torsion point becomes a rational point inside the \emph{bounded} fundamental
box $[0,1)^{2g}$, the covering restricted to that box has finite fibers, and
the plain Pila--Wilkie theorem already suffices as in \cite{pz}. We have no such real structure and no Betti
coordinates. A torsion point of $\cald\times\cald'$ lifts under $\exp_\Phi^{-1}$ to
a vector $\big(\sum_i a_i\xi_i,\ \sum_j b_j\xi'_j\big)$ whose coefficients
$(a_1,\dots,a_r,b_1,\dots,b_r)$ are recorded in the \emph{higher-dimensional}
parameter space $\A^{2r}_{\C_\infty}$, with the rank $r$ of the Drinfeld modules
entering through the $2r$ lattice coordinates. The linear map
$\ell\colon\A^{2r}_{\C_\infty}\to\G_{a,\C_\infty}^2$ that reconstitutes the
point from its coefficients has fibers of dimension $2r-2$, hence positive dimensional as soon as $r\ge 2$; these fibers are genuine algebraic
blocks inside the affinoid we count in. We are therefore forced to invoke the refinement in blocks
(Theorem~\ref{t:blockcounting}) already in the Manin--Mumford case, where
classically it is not needed.

The second concerns the Andr\'e--Oort theorem and the very applicability of the counting
theorem. Classically the fundamental domain of $\mathrm{SL}_2(\Z)$ on the
upper half plane is semialgebraic, and although it runs out to the cusp, the
$q$-expansion expresses the restriction of $j$ to it through the real
exponential, so that $o$-minimality tames the cusp and Pila--Wilkie applies to
the whole domain at once \cite{pilaAO,scanlonsurvey}. We have neither
$o$-minimality nor a bounded ambient space: the Drinfeld fundamental domain
$\calf$ runs out to the cusp, and $\calf$ is not affinoid, so the
rigid analytic counting theorem cannot be applied to $\calf$ directly. We must
first manufacture a bounded domain, and this is where the height enters.
Combining the Colmez-type bound for the Taguchi height (Lemma~\ref{colmezbound},
a Drinfeld analogue of an estimate that Tsimerman established, to a different end, in \cite{tsimao}) with the comparison
between the Taguchi and Weil heights (Proposition~\ref{p:taguchi-weil}), we
bound $h(j(z))\ll_\varepsilon|D|_\infty^\varepsilon$, hence the imaginary part
$|z|_i\le |D|_\infty^\varepsilon$, of a CM point of discriminant $D$. This
confines the relevant conjugates to a bounded region, which
Lemma~\ref{l:cover} then covers by $|D|_\infty^{O(\varepsilon)}$ closed affinoid
balls of radius $q^{-1}$, each a bounded affinoid to which rigid Pila--Wilkie does apply. In this sense, the Taguchi height plays here the role that
$o$-minimality plays in the classical proof at the cusp, as it supplies the boundedness needed before the counting theorem can be applied. This mechanism is not specific to positive
characteristic. In characteristic zero one could similarly combine the
Faltings height of CM elliptic curves with Silverman's comparison between
Faltings and Weil heights to truncate the classical fundamental domain before
counting; the usual $o$-minimal proof simply makes this truncation
unnecessary. By the
pigeonhole principle one ball captures a definite proportion of the Galois
conjugates, and we count there. The final point is one of uniformity: the
number of balls grows with $D$, so the counting constant must not depend on
which ball we land in. It is the family form of the rigid-analytic
counting theorem (Theorem~\ref{t:blockcounting}) that yields a constant uniform
across the cover, and this uniformity is what carries the argument.

\medskip
\noindent\textbf{Earlier works.}\quad
The Andr\'e--Oort conjecture for products of Drinfeld modular curves was first proved by Breuer \cite{breuerAO}, and extended in \cite{breuer2}, adapting
Edixhoven's method to the Drinfeld setting; unlike its archimedean model the
proof is unconditional, the Riemann hypothesis for function fields being a
theorem of Deligne. It was carried to Drinfeld modular varieties in
\cite{breuer3}, for special points lying in a single Hecke orbit, and in \cite{drinfeldmodvarAO} under further technical hypotheses; the
latter result has no characteristic-zero counterpart. More recently,
\cite[Remark~3.28]{breut} proposes a first notion of special subvariety in a
moduli space of Shtukas, with a view towards an Andr\'e--Oort statement in that
setting. In a different direction, 
\cite{angles} prove an effective Andr\'e--Oort statement for the family of hyperbolas $XY=\gamma$, for $\gamma\in\F_q[T]^\times$, by means of height estimates. On the Manin--Mumford side, the conjecture of Denis \cite{denismaninmumford} on the torsion of a Drinfeld module was established by Scanlon \cite{scanlon} for the self-product of a Drinfeld module through the model theory of difference fields.

\medskip
\noindent\textbf{Future directions.}\quad The Manin--Mumford and Andr\'e--Oort statements
proved here are the first instances of the Zilber--Pink conjecture. The
Pila--Zannier strategy behind them, however, is not bound to this layer: supplied with the
appropriate functional transcendence, the same machinery is expected to reach higher strata of Zilber--Pink. That transcendence is no longer Ax--Lindemann but various incarnations of the Ax--Schanuel
theorem, which we establish for the Drinfeld $j$-function in the forthcoming work \cite{bns};
combined with the rigid Pila--Wilkie counting, and with the new height estimates that such a
generalization demands, we expect it to yield Zilber--Pink results in the Drinfeld modular setting. On the linear side, the first inroads into this conjecture for Drinfeld modules are due to Brownawell and
Masser, who in \cite{bm} formulate a Zilber--Pink statement for curves in a product of Carlitz
modules and prove it in dimension three, by a height bound rather than by point counting.

\medskip
\noindent\textbf{Notation and conventions.}\quad Let $q$ be a power of a prime $p$. Set
$C:=\mathbb P^1_{\mathbb F_q}$ and let $ \infty:=[1:0]\in C$.
Consider $C-\{\infty\}\simeq \mathbb A^1_{\mathbb F_q}$ with coordinate $T$.
Then the ring of functions on $C$ regular away from $\infty$ is
\[
A:=H^0(C-\{\infty\},\calo_C)=\mathbb F_q[T]
\]
and its fraction field is $F:=\mathbb F_q(T)$.

Let $v_\infty$ be the place of $F$ corresponding to $\infty$, normalized by
\[
v_\infty(f/g)=\deg_T(g)-\deg_T(f)\qquad \text{for}\;\;f,g\in \mathbb F_q[T],\ g\neq 0,
\]
and set the associated absolute value
\[
|f/g|_\infty:=q^{-v_\infty(f/g)}=q^{\deg_T(f)-\deg_T(g)}.
\]
We use the conventions $\deg(0)=-\infty$, $v_\infty(0)=+\infty$,
and $|0|_\infty=0$.
It is non-archimedean; that is,
$|x+y|_\infty\le
\max\{|x|_\infty,|y|_\infty\}$ for all $x,y\in F$.
The completion of $F$ with respect to $|\cdot|_\infty$ consists of formal Laurent series
\[
F_\infty= \mathbb F_q(\!(1/T)\!)
\]
and its ring of integers is $\calo_{F_\infty}=\mathbb F_q[\![1/T]\!]$.
For $x\in F$ written uniquely as a reduced fraction $x=a/b$ with $a,b\in A$ coprime and $b$ monic, we set
\[
H(x):=\max\{|a|_\infty,|b|_\infty\}
\]
and its logarithmic version
\[
h(x):=\log_q H(x)=\max\{\deg_T(a),\deg_T(b)\}.
\]
This agrees with the Weil height on $\p^1$ coming from Section \ref{s:weil}. For $u\in\R_{\ge0}$, we write $\log_q^+u:=\log_q\max\{1,u\}$.

Let $L$ be a field. For a fixed $\F_q$-morphism $\iota\colon A\rightarrow L$,  the pair $(L,\iota)$ is called an $A$-field. If $\iota$ is injective we say that $L$
has \emph{generic $A$-characteristic}; otherwise $L$ has \emph{finite} characteristic. In what follows we work exclusively in the generic characteristic case.

Let $\mathbb G_{a,L}=\Spec L[X]$ be the additive group over $L$. Denote by
$\uptau\in\End_L(\mathbb G_{a,L})$ the $q$-power endomorphism given by $\uptau^*(X)=X^q$.
We write $L\{\uptau\}$ for the skew-polynomial ring in $\uptau$ with the relation
\[
\uptau b=b^q\uptau,
\]
for every $b\in L$.
In the definitions of Drinfeld and Anderson modules below,
$\End_L$ denotes the ring of $\F_q$-linear
$L$-group-scheme endomorphisms.

Let $\mathbb C_\infty$ be the completion of a fixed algebraic closure of $F_\infty$; by Krasner's lemma it is algebraically closed, but not locally compact.

For a finite place $v\neq\infty$ of $F$, there is a unique monic irreducible polynomial
$P\in A$ such that $v=\ord_P$. We normalize $\ord_P$ by $\ord_P(P)=1$ and define, for $f\in F-\{0\}$,
\[
|f|_v:=q^{-\deg_T(P)\,\ord_P(f)},
\]
and $|0|_v:=0$.

If $L/F$ is a finite extension and $w$ is a place of $L$ lying above $v$, we normalize the
absolute value on $L$  by
\[
|l|_w \ :=\ \bigl|N_{L_w/F_v}(l)\bigr|_v^{\,1/[L_w:F_v]}
\]
for every $l\in L$, where $F_v$ and $L_w$ denote the completions at $v$ and $w$.
%Equivalently, we normalize the discrete valuation $\ord_w$ by $\ord_w(L^\times)=\mathbb Z$ and set
%$|l|_w=q^{-\deg(w)\,\ord_w(l)}$, where $\deg(w)=[k(w):\mathbb F_q]$.
We denote by $\calo_L$ the integral closure of $A$ in $L$.

Lastly, by variety we shall mean  an integral, separated scheme of finite type.

\subsubsection*{Acknowledgments}
This work was partially done while the authors were at the Institute for Advanced Study in Princeton and they would like to thank the institute for its hospitality and for providing excellent working conditions. 
G.B.\ was supported by the Marvin V.\ and Beverly J.\ Mielke Endowed Fund and the Infosys Member Fund and D.N.\ was supported by Kovner Member Fund.
G.B.\ and F.M.S.\ were also supported by the European Union (ERC, SharpOS, 101087910) and by the Israel Science Foundation (grant No.\ 2067/23).
D.N.\ was also supported by the Israel Science Foundation grant 1167/17 and by Minerva grant 714141. We also thank Federico Pellarin for his comments on a first draft of this paper.

\subsubsection*{AI statement}
All of the new ideas of this work are due to the authors and were developed from November 2024 to April 2026. The execution of these ideas is also due to the authors with very few exceptions, where AI was exploited for minor technical simplifications and corrections.
{\em GPT 5.6  pro} was extensively used to proofread this work.

	\section{Drinfeld modules}

	For an introductory, detailed treatment of the background material presented here, we invite the reader to go through \cite{pap}.

	\subsection{Powers of a Drinfeld module}

		Let us recall that, for an $L$-group scheme $G$ with unit section $e\colon \Spec L\rightarrow G$, the tangent space of $G$ along $e$ is $\Lie(G):=\Hom_L(e^*\Omega^1_{G/L},L)$. In particular, for $G$ smooth over $\Spec L$, then $\Lie G$ is a locally free $L$-module of rank equal to the $L$-dimension of $G$. Hence, for $G^n$ the $n$th fiber product of $G$ over $\Spec L$, we have that $\Lie G^n=(\Lie G)^n$.
	
	\subsubsection{Drinfeld modules}
    
    Let $(L,\iota)$ be an $A$-field. A {\em Drinfeld $A$-module} $\cald$ over $L$ consists of a pair
    \[
    (G,\varphi)
    \]
    where $G$ is an $L$-group scheme isomorphic to $\G_{a,L}$ and $\varphi$ is an $\F_q$-algebra morphism 
    \[
    \varphi\colon A\rightarrow \End_L(G)\simeq L\{\uptau\},\;\;a\mapsto \varphi_a
    \]
	such that the induced action on $\Lie G$ is given by the structure map $\iota$ and such that $\varphi$ is not the composition of $\iota$ with $L\subset L\{\uptau\}$. In symbols, $\Lie(\varphi_a)=\iota(a)$ and $\deg_\uptau(\varphi_a)>0$ for every nonconstant $a\in A$. Note that, since $A$ is generated by $T$ over $\F_q$, then $\varphi_T$ uniquely determines $\varphi$. The {\em rank} of $\cald$ is the natural $r=\deg_\uptau(\varphi_T)$ such that $\deg_\uptau(\varphi_a)=r\deg(a)$ for every $a\in A-\{0\}$. We can thus write $$\varphi_T=\iota(T)+g_1(T)\uptau+\dots+g_r(T)\uptau^r,$$ 
    for $g_i(T)\in L$ and $g_r(T)\in L^\times$.
	
	A morphism between Drinfeld $A$-modules $\cald$ and $\cald'$ over $L$ is a morphism $f\colon G\rightarrow G'$ of $L$-group schemes  such that $\varphi_T'\circ f=f\circ\varphi_T$.

	\subsubsection{Products of Drinfeld modules}
	%As in \cite{anderson}, 
    We briefly introduce, following \cite[Section 5]{goss}, the higher dimensional analog of  Drinfeld modules, which allows us to consider a power of the latter.
	\\
	
	An {\em Anderson $A$-module} $\cala$ of dimension $d$ over $L$ consists of a pair
	\[
	(G_d,\Phi)
	\]
	where $G_d$ is an $L$-group scheme isomorphic to $\G_{a,L}^d$ and $\Phi$ is an $\F_q$-algebra morphism 
	\[
	\Phi\colon A\rightarrow \End_L(G_d)\simeq \M_d(L)\{\uptau\}
	\]
	such that $\Lie(\Phi_T)=\iota(T)I_d+N$, where $N\in M_d(L)$ is a nilpotent matrix, and such that $\ker \Phi_T$ is a non-trivial finite group scheme. The {\em rank} of $\cala$ is given by $\log_q\#\ker(\Phi_T)$.
	
	A morphism of Anderson $A$-modules $\cala$ and $\cala'$ over $L$ is a morphism $G_d\rightarrow G_{d'}$ of $L$-group schemes commuting with the action of $A$. 
	
	In particular, given two Drinfeld $A$-modules $\cald$ and $\cald'$ over $L$, from now on we will focus on the Anderson $A$-module, which we denote 
    $$\cald\times\cald'$$
    obtained by taking the Cartesian product of $\cald$ and $\cald'$, the diagonal action $\Phi_T=\text{diag}(\varphi_T,\varphi'_T)$ and $N=0$.
	
	Lastly, a {\em sub-$\Phi$-module} of $\cald\times\cald'$ is defined to be a connected $L^{\text{sep}}$-algebraic subgroup $H$ of $\G_{a,L^{\text{sep}}}^2$ which is  stable under the action of $A$ via $\Phi$, i.e., $\Phi_a(H)\subseteq H$ for all
$a\in A$.
		
	\subsubsection{Rigid analytic uniformisation}
	
	Recall that by an {\em $A$-lattice} $\Lambda$ in $\C_\infty$, we mean a discrete, finitely generated projective $A$-submodule of $\C_\infty$.
	Since $A$ is a principal ideal domain, then we can write $\Lambda=\bigoplus_{i=1}^rAz_i$ for $\{z_1,\dots,z_r\}$ an $F_\infty$-basis for an $F_\infty$-subspace of $\C_\infty$; hence the rank of $\Lambda$ is its rank as a free $A$-module. We highlight that, since  $\C_\infty$ has infinite degree over $F_\infty$, there exist lattices of any rank $r\ge1$.
	
	To such a lattice, one can attach the {\em Drinfeld exponential}
	\[
\exp_\Lambda(z):=z\prod_{0\neq\lambda\in\Lambda}\left( 1-\dfrac{z}{\lambda}\right) 
\]
	for $z\in \C_\infty$. Since the intersection of $\Lambda$ with any ball in $\C_\infty$ is finite, the infinite product converges and the Drinfeld exponential is an entire rigid-analytic function. It admits an expansion $\exp_\Lambda(z)=z+\sum_{n>0} a_n z^{q^n}$ and it defines a surjective $\F_q$-linear map whose kernel $\Lambda$ consists of simple zeroes only. Remarkably, via this expansion one can show  the additivity of the Drinfeld exponential: for every $z,w$ one has 
    \[
    \exp_\Lambda(z+w)=\exp_\Lambda(z)+\exp_\Lambda(w)
    \qquad
\text{and}
\qquad
\exp_\Lambda(\lambda z)=\lambda\exp_\Lambda(z)
	\] 
     for $\lambda\in\F_q$, so that $\exp_\Lambda$ is $\Lambda$-periodic. In particular, as $a_0=1$, the Drinfeld exponential admits a composition inverse
    $\log_\Lambda(z)=z+\sum_{n>0} b_nz^{q^n}$
	which has positive radius of convergence.
    
	In \cite[Proposition 3.1]{drinfellipt}, Drinfeld showed an equivalence of categories between Drinfeld $A$-modules over $\C_\infty$ and $A$-lattices in $\C_\infty$.  For each $a\in A-\{0\}$, there exists  $\varphi_{\Lambda,a}\in\C_\infty\{\uptau\}$ of  polynomial degree $q^{r\deg_T(a)}$, equivalently $\uptau$-degree $r\deg_T(a)$, such that the Drinfeld exponential satisfies the functional equation
	\begin{equation}\label{drinfeldexp}
    \exp_\Lambda(\iota(a)z)=\varphi_{\Lambda,a}(\exp_\Lambda(z)).
	\end{equation}
	Through the $\F_q$-algebra morphism $\varphi_\Lambda$ one has a Drinfeld $A$-module structure on $\G_{a,\C_\infty}$, considered as $\C_\infty/\Lambda$ via $\exp_\Lambda$. 
	
	To summarize, the Drinfeld exponential induces the following sequence
	
	\begin{equation}\label{e:sesexp}
	\begin{tikzcd}
		&0\arrow{r} & \Lambda\arrow{r} & \Lie\G_{a,\C_\infty}\arrow{r}{\exp_\Lambda}& \cald(\C_\infty)\arrow{r}&0
	\end{tikzcd}
	\end{equation}
	which is exact.
	\\
In particular, consider 
\[
\widetilde{\pi}:=\big(T-T^q\big)^{\frac{1}{q-1}}\prod_{i=1}^\infty\left(1-\dfrac{T^{q^i}-T}{T^{q^{i+1}}-T} \right)\;,
\]
which is known to be transcendental over $F$.

The $A$-lattice $\widetilde{\pi}A$ corresponds to the Drinfeld module $\ccal$ of rank-$1$  known as {\em Carlitz module}, whose exponential is given by
\begin{equation}\label{carlitzexp}
    \exp_\ccal(z)=\sum_{n\ge 0}\dfrac{z^{q^n}}{D_n}
\end{equation}
where $D_n=\big( T^{q^n}-T\big)D^q_{n-1}$ and $D_0=1$. Note that $\big(T^q-T\big)^\frac{1}{q-1}$ can be thought of as {\em an}\footnote{Another analogue could simply be $(-T)^{\frac{1}{q-1}}$.} analogue of the complex imaginary unit. Thus, the function field counterpart of the classical complex $\exp(2\pi iz)$ is given by 
\[
Q(z):=\frac{1}{\exp_\ccal\big(\widetilde{\pi}z\big)}
\]
so that we have Laurent expansions of modular functions in terms of $Q(z)$, which we will just denote by $Q$.
\\
    
The higher dimensional situation is somehow subtler. Given an Anderson $A$-module $\cala$ of dimension $d$ over $\C_\infty$, there exists a unique exponential function 
$$\exp_\Phi\colon \Lie\G_{a,\C_\infty}^d\rightarrow \cala(\C_\infty)$$ 
defined via the power series $\exp_\Phi(z)=z+\sum_{i>0} A_iz^{(q^i)}$,
for the column vector $z=(z_1,\dots,z_d)^{\mathsf T}$ and
$A_i\in M_d(\C_\infty)$, where powers are taken coordinatewise, such that
$\exp_\Phi(\Lie(\Phi_a)(z))=\Phi_a(\exp_\Phi(z))$. As before, $\exp_\Phi$ is entire on $\C_\infty^d$, but in general it fails to be surjective. Whenever $\exp_\Phi$ is surjective, the Anderson $A$-module is said to be {\em uniformizable}, and $\cald\times\cald'$ is clearly so.

\subsubsection{Torsion points}

Let $L'$ be an $L$-algebra and let $\cald$ be a Drinfeld $A$-module of rank $r$ over $L$.
For $a\in A-\{0\}$, the {\em $a$-torsion submodule} $\cald[a]$ of  $\cald$ is the closed subscheme whose functor of points sends $L'$ to
	 \[
	 \cald[a](L')=\{x\in\cald(L'): \varphi_a(x)=0\}.
	 \]
	 Since every non-zero morphism between Drinfeld $A$-modules is an isogeny (see \cite[Definition~3.3.1]{pap}), the $a$-torsion submodule is a finite locally free $L$-group scheme. 
     { 
     
     }
	 
     Moreover, $\cald[a]$ has also an $A$-module structure via $f.x=\varphi_f(x)$ for $f\in A$ and $x\in\cald[a](L')$. In fact, since $\varphi_a\circ\varphi_f=\varphi_{af}=\varphi_{fa}=\varphi_f\circ\varphi_a$, we have that $\varphi_a(f.x)=0$.  By a structure theorem as \cite[Theorem 1.2.3]{pap}, it is easy to see that $\#\cald[a]=|a|_\infty^r$. Since we are in generic characteristic, $\varphi_a$ is separable; hence its splitting field $L\big(\cald[a]\big)$ is a Galois extension of $L$.
	 
	 We define the {\em torsion submodule} to be the (ind-)subscheme $\cald_{\text{tor}}$ whose functor of points is
	 \[
	 \cald_{\text{tor}}(L')=\bigcup_{a\in A-\{0\}}\cald[a](L')
	 \]
	 for an $L$-algebra $L'$. Since $A$ is a principal ideal domain, for each $x\in\cald_{\text{tor}}(L')$ there exists a unique monic polynomial $a\in A$, called the {\em order} of $x$, such that $\varphi_a(x)=0$ and minimal with respect to polynomial division. Note that, since every $\varphi_a$ is separable for $a\neq 0$, then $\cald_{\text{tor}}(L^{\text{alg}})\subset L^{\text{sep}}$. In addition, via the inverse of the Drinfeld exponential, one sees that
	 \begin{equation}\label{e:torsion}
	 \cald_{\text{tor}}(\C_\infty)\simeq \dfrac{F\otimes_A\Lambda}{\Lambda}.
	 \end{equation}
	 As $\Lambda$ has $A$-rank $r$, we have $(F\otimes_A\Lambda)/\Lambda\simeq (F/A)^r$. All these considerations hold for $\cald\times\cald'$ simply by keeping track of its dimension and of the ranks. In particular, note that $(\cald\times\cald')_{\text{tor}}(L^{\text{alg}})$ is Zariski dense in $\G_{a,L^{\text{sep}}}^2$. 
	 \\
	 
Henceforward, let $\cald=(G,\varphi)$ and $\cald'=(G',\varphi')$ denote two Drinfeld $A$-modules of ranks $r$ and $r'$ respectively, with associated lattices $\Lambda$ and $\Lambda'$, and fix two $A$-bases $(\xi_i)_{i=1}^r$ of $\Lambda$ and $(\xi'_j)_{j=1}^{r'}$ of $\Lambda'$.

\begin{lem}\label{l:torsionball}
Let $x\in\cald_{\text{tor}}(\C_\infty)$ and let $z\in\exp_\cald^{-1}(x)$. Writing
$z=\sum_{i=1}^r a_i\xi_i$ with $a_i\in F$,
there exist  unique polynomials $P_i\in A$ such that, setting
$\tilde a_i:=a_i-P_i$, one has
 $\tilde z:=\sum_{i=1}^r \tilde a_i\xi_i \equiv z \bmod \Lambda$, 
 %$\tilde a_i\in \frac{1}{T}\F_q[\![1/T]\!]$
and 
\[
|\tilde a_i|_\infty \le \dfrac{1}{q}
\]
for all $i$.
\end{lem}

\begin{proof}
For each $i$,  the rational function $a_i\in F$ can be written uniquely as
$a_i=P_i+\tilde a_i$ for
$P_i\in A$ and $\tilde a_i\in \frac{1}{T}\F_q[\![1/T]\!]$. Then $\sum_i P_i\xi_i\in\Lambda$, so $\tilde z\equiv z\pmod{\Lambda}$.
The uniqueness is straightforward: if $a_i-P_i\in \frac{1}{T}\F_q[\![1/T]\!]$ and $a_i-Q_i\in \frac{1}{T}\F_q[\![1/T]\!]$ with $P_i,Q_i\in A$, then
$P_i-Q_i\in A\cap \frac{1}{T}\F_q[\![1/T]\!]=\{0\}$.  
\\Finally, the $\infty$-adic valuation of $\tilde a_i$ is positive, that is
$|\tilde a_i|_\infty\le q^{-1}$.
\end{proof}

The following lemma states the analog of a result of Masser
\cite[Corollary~(ii), p.156]{massertorsionbound} on the Galois action on
torsion points.

\begin{lem}\label{masserboundtorsion}
Assume that $L$ is finitely generated over $\F_q$. There exist positive constants
$C'=C'(\cald\times\cald',L)$ and $\eta=\eta(\cald\times\cald',L)$
such that, for any torsion point
$x\in(\cald\times\cald')_{\mathrm{tor}}(L^{\mathrm{alg}})$
of order $a\in A-\{0\}$, one has
\[
[L(x):L]\ge C' |a|_\infty^\eta.
\]
\end{lem}

\begin{proof}
Recall that, since $x$ has order $a$, the $A$-submodule generated
by $x$ is isomorphic to $A/a$, and therefore has cardinality $\#(A/a)=|a|_\infty$.
All of these points are $L(x)$-rational torsion points of
$\cald\times\cald'$. Hence
\[
|a|_\infty
\le
\#(\cald\times\cald')_{\mathrm{tor}}(L(x)).
\]
%By \cite[Theorem~1.1.II]{breuertorsionbound},
%applied to $\cald\times\cald'$, there are constants $C_0,\gamma>0$,
%depending only on $\cald\times\cald'$, such that
{Applying \cite[Theorem~1.1(b),(II)]{breuertorsionbound}
separately to the two factors and multiplying the bounds, we obtain two positive constants $C_0$ and $\gamma$, depending only on the two Drinfeld modules and $L$, such that}
\[
\#(\cald\times\cald')_{\mathrm{tor}}(L(x))
\le
C_0\bigl([L(x):L]\log\log\bigl(e^e[L(x):L]\bigr)\bigr)^\gamma.
\]
This implies, after absorbing
the harmless $\log\log$-factor into a small power, that there are constants
$C'>0$ and $\eta>0$, depending only on $\cald\times\cald'$ and $L$, such
that
$[L(x):L]\ge C' |a|_\infty^\eta$.
\end{proof}
\begin{comment}
The following lemma states the analog of a result of Masser \cite[Corollary (ii) p.156]{massertorsionbound} on the Galois action on torsion points.
	 
	 \begin{lem}\label{masserboundtorsion}
	 	There exist two positive constants $C'=C'(\cald\times\cald',L)$ and $\eta=\eta(\cald\times\cald',L)$ such that, for any $x\in(\cald\times\cald')_{\text{tor}}(L^{\text{alg}})$ of order $a$, the inequality
	 	\[
	 	[L(x):L]\ge C'/|a|_\infty^\eta
	 	\]
	 	holds.
	 \end{lem}
 \begin{proof}
 	
Let $\gamma=\rk_A\End_{L^{\text{sep}}}(\cald)/r$. By \cite[Theorem 1.1.II]{breuertorsionbound}, we have the following bound
\begin{equation}\label{brauerthmbound}
\#\cald_{\text{tor}}(L)\le C'([L:F]\log\log[L:F])^\gamma,
\end{equation}
that immediately extends to a bound on the torsion points of $(\cald\times\cald')(L)$.
Given $x$ a torsion point of order $a$ over $L$, the action of the absolute Galois group of $L$ gives $|a|_\infty$ distinct torsion points over $L$. Thus by \eqref{brauerthmbound} we get the desired bound on the order of the torsion point $x$.
\end{proof}
\end{comment}
	
\subsubsection{Weil height}\label{s:weil}

\begin{comment}
	{ \\

	Following \cite[Th\'eor\`eme 1]{canonicalheight}, one has the analog of the Neron--Tate height on a Drinfeld $A$-module $\cald$ of rank $r$. The {\em canonical height} on $\cald(F^{\text{alg}})$ is the function 
	$\widehat{h}_\cald\colon\cald(F^{\text{alg}})\rightarrow \R$ defined by
	\[
	\widehat{h}_\cald(x):=\lim_{n\to\infty}\dfrac{h(\varphi_{T^n}(x))}{q^{rn}}.
	\]
	In particular, this is the unique function such that $\widehat{h}_\cald(\varphi_a(x))=|a|_\infty^r\widehat{h}_\cald(x)$, for $a\in A$ of positive degree, and such that $\widehat{h}_\cald -h$ is a bounded function on $\cald(F^{\text{alg}})$. As in the classical case, the canonical height vanishes only on torsion points. Thus, by combining these last two properties, we have
	\begin{equation}\label{canonicalheight}
		|h(x)|\le O_\cald(1)
	\end{equation}
    for $x\in\cald_{\text{tor}}(F^{\text{alg}})$. Moreover, an explicit bound for the difference between the canonical and the Weil height is obtained in \cite[Theorem 1.6]{ingram}.
    
    Naturally, both the previous definition and the bound \eqref{canonicalheight} extend for the Anderson $A$-module $\cald\times\cald'$.
}
\end{comment}

As in the number field case, this construction is independent of $L$ and of the coordinates. Fixing an embedding $\A^2\hookrightarrow \p^2$, the Weil height of points in affine spaces is obtained as the Weil height relative to this embedding. More generally, we use the same notation for points of affine space of any dimension, and write $H(a):=q^{h(a)}$ for the corresponding multiplicative Weil height.
    \\
    
    Let $L$ be a finite extension of $F$ of degree $d$. For a point
$x=[x_0:x_1:x_2]\in \p^2(L)$ we define its (logarithmic) Weil height by
\[
h(x)\ :=\ \frac{1}{d}\sum_{w} d_w\,\max_{0\le i\le 2}\bigl\{-\mathrm{val}_w(x_i)\bigr\},
\]
where $w$ runs over the places of $L$,
$\operatorname{val}_w$ is normalized by
$\operatorname{val}_w(L^\times)=\mathbb Z$,
and $d_w=[\mathbb F_w:\mathbb F_q]$, with $\mathbb F_w$ the residue field at $w$. As in the number field case, this height is
independent of the choice of homogeneous coordinates for $x$, and is functorial in
the extension $L/F$.
For $a\in \A^2_{F^{\text{alg}}}$ we define $h(a)$ to be the height of its image in projective space under this embedding.
For a torsion point
$x\in(\cald\times\cald')_{\mathrm{tor}}(\C_\infty)$,
let $z_x=(u_x,v_x)\in\exp_\Phi^{-1}(x)$
be its reduced\footnote{By this we mean the unique element of
$\exp_\Phi^{-1}(x)$ whose coefficients with respect to the fixed
$A$-bases of $\Lambda$ and $\Lambda'$ belong to
$\frac{1}{T}\mathbb F_q[\![1/T]\!]$, or equivalently have
$\infty$-adic absolute value at most $q^{-1}$; see
Lemma~\ref{l:torsionball}.} logarithmic representative with respect to the
fixed bases $(\xi_i)_{i=1}^r$ of $\Lambda$ and
$(\xi_j')_{j=1}^{r'}$ of $\Lambda'$. Thus
$u_x=\sum_{i=1}^r t_i\xi_i$ and $v_x=\sum_{j=1}^{r'}t_j'\xi_j'$,
where $t_i,t_j'\in F$ and $|t_i|_\infty, |t_j'|_\infty\le q^{-1}$. We set
\[
H(z_x):=
H(t_1,\dots,t_r,t_1',\dots,t_{r'}'),
\]
where the right-hand side is the multiplicative Weil height on
$\A^{r+r'}(F)$.

\begin{lem}\label{l:heightorder}
For every
$x\in(\cald\times\cald')_{\mathrm{tor}}(\C_\infty)$,
one has
\[
H(z_x)
\le
|\ord(x)|_\infty
\le
H(z_x)^{r+r'}.
\]
\end{lem}

\begin{proof}
Write
$t_i=\frac{a_i}{b_i}$ and $t_j'=\frac{a_j'}{b_j'}$
in lowest terms, with monic denominators, and denote
$\ell=
\operatorname{lcm}
(b_1,\dots,b_r,b_1',\dots,b_{r'}')$.
Since $\ell t_i,\ell t_j'\in A$ for every $i,j$, we have
$\ell z_x\in\Lambda\times\Lambda'$. Hence, by
\eqref{drinfeldexp}, applied to each factor,
$\Phi_\ell(x)=\exp_\Phi(\ell z_x)=0$.
Therefore $\ord(x)$ divides $\ell$, and so
\[
|\ord(x)|_\infty\le|\ell|_\infty.
\]

Since $z_x$ is reduced, every non-zero coefficient satisfiesn $\deg_T(a_i)<\deg_T(b_i)$ and $\deg_T(a_j')<\deg_T(b_j')$. Thus $H(t_i)=q^{\deg_T(b_i)}$ and
$H(t_j')=q^{\deg_T(b_j')}$. It follows that
\[
\deg_T(\ell)
\le
\sum_{i=1}^r\deg_T(b_i)
+
\sum_{j=1}^{r'}\deg_T(b_j') \
\le
(r+r')
\max\left\{
\max_i\deg_T(b_i),
\max_j\deg_T(b_j')
\right\}.
\]
Since
$H(z_x)\ge
\max\left\{
\max_i H(t_i),
\max_j H(t_j')
\right\}$,
we obtain $|\ell|_\infty\le H(z_x)^{r+r'}$.
This proves the upper bound.

For the lower bound, set $n=\ord(x)$. By \eqref{drinfeldexp}, $\exp_\Phi(nz_x)=\Phi_n(x)=0$,
and hence $nz_x\in\Lambda\times\Lambda'$. Thus
$c_i:=nt_i\in A$ and $c_j':=nt_j'\in A$ for every $i,j$. Since $z_x$ is reduced, every non-zero $c_i,c_j'$ satisfies
\[
\deg_T(c_i)<\deg_T(n)
\qquad
\text{and}
\qquad
\deg_T(c_j')<\deg_T(n).
\]
The affine point $(t_1,\dots,t_r,t_1',\dots,t_{r'}')$
therefore has homogeneous coordinates
$[n:c_1:\cdots:c_r:c_1':\cdots:c_{r'}']$,
and these give
$H(z_x)\le q^{\deg_T(n)}=|n|_\infty$.
Since $n=\ord(x)$, this proves the lower bound.
\end{proof}

	\subsection{Drinfeld modular curves}

	For an introductory, detailed treatment of the background material here presented we invite the reader to go through \cite{pap}.

\subsubsection{Drinfeld half plane}
As a reference for the equicharacteristic case, we invite the reader to see \cite{genestier}.

Consider the formal scheme $\widehat{\Omega}$ over $\calo_{F_\infty}$ obtained by successive blow-ups of rational points on the special fiber of $\p^1$, whose generic fiber (in the sense of Raynaud) is the rigid-analytic Drinfeld plane $\Omega$ over $F_\infty$ defined as
\[
\Omega=\p^1(\C_\infty)-\p^1(F_\infty)\;.
\]
 This is a connected (but not simply connected) admissible open subspace of the rigid-analytic $\p^1(\C_\infty)$. This implies that $\Omega$ has a covering by open sets $U_n$ such that the $U_n$'s are disks with $q$ smaller disks removed.
 Given a point $z\in \Omega$, we define, as in \cite[V.1]{gekeler}, its {\em imaginary part} as
\[
|z|_i:=\inf_{x\in F_\infty}|z-x|_\infty=\min_{x\in F_\infty}|z-x|_\infty\;.
\]
As we have a M\"obius action of $\gl_2(A)$ on $\Omega$, the following admissible open (and closed) subset of $\Omega$
\[
\calf=\big\{z\in\Omega : |z|_\infty=|z|_i\ge 1 \big\}
\]
is the (closest analog to a classical) fundamental domain for such an action as every  $z\in\Omega$ is $\gl_2(A)$-equivalent to some element of $\calf$. Moreover, if $z,z'\in\calf$ are $\gl_2(A)$-equivalent, then $|z|=|z|_i=|z'|=|z'|_i$ (see \cite[Proposition 6.5]{gekeler2}). For a quadratic extension $K$ of $F$, we shall denote by $\calf_K:=\calf\cap K$.

\begin{lem}\label{l:cover}
Assume $q$ is odd and let $K=F(\sqrt D)$, with $D\in A$ squarefree and suppose that $K$ has no embeddings into $F_\infty$.
Fix $R\ge 1$ and put $d:=\lceil \log_q R\rceil$ and
$c:=\big\lceil \deg_T(D)/2\big\rceil$. Let $\rho\in\{0,1\}$ be defined by $\rho=0$ in the unramified case and $\rho=1$ in the ramified case.
Define the finite set of centers
\[
C_{d}:=\Bigl\{\,P_2+\frac{P_1}{T^{c}}\sqrt D:\ P_1,P_2\in A,\ \deg P_i\le d\Bigr\}\subset K,
\]
and for each $P\in C_{d}$ the subset
\[
U_P^\rho:=\big\{z\in\Omega:\ |z-P|_\infty\le q^{-1}\big\}.
\]
Then
\[
\bigl\{z\in\calf_K:\ |z|_i\le R\bigr\}\subset \bigcup_{P\in C_{d}} U_P^\rho.
\]
In particular,
{   $\#C_d=q^{2d+2}\le q^4R^2$}. {  After discarding the  $U_P^\rho$'s that do not meet
$\{z\in\calf_K:|z|_i\le R\}$, the remaining sets form a covering by closed affinoid balls contained in $\Omega$.}
\end{lem}

\begin{proof}
Let $z\in \calf_K$ with $|z|_i\le R$. Write uniquely $z=a+b\sqrt D$ with
$a,b\in F_\infty$. Since $z\in\calf_K$ we have $|z|_\infty=|z|_i\le R$. Moreover, as $K$ has no embeddings into $F_\infty$, then $K\otimes_FF_\infty\simeq K_\infty$ is a quadratic field extension of $F_\infty$. Let $\sigma\in\Aut_{F_\infty}(K_\infty)$ be such that $\sigma\colon\sqrt{D}\mapsto-\sqrt{D}$. %$|\sigma(z)|_\infty=|z|_\infty$. 
For $z=a+b\sqrt{D}$, write $a=\frac{z+\sigma(z)}{2}$ and $b\sqrt{D}=\frac{z-\sigma(z)}{2}$. Since $q$ is odd, we have $|2|_\infty=1$. As automorphisms preserve the norm, by the ultrametric inequality we have $|a|_\infty\le \max\{|z|_\infty,|\sigma(z)|_\infty\}=|z|_\infty$. Similarly we get $|b|_\infty|\sqrt{D}|_\infty\le |z|_\infty$. Therefore
$|a|_\infty\le |z|_\infty\le R$ and $|b|_\infty\le |z|_\infty\,|\sqrt D|_\infty^{-1}\le R\,q^{-\deg(D)/2}$.
%In particular, the polynomial part of $a$ has degree $\le d$, and the polynomial part of $b$ has degree $\le d-c$.

Choose $P_2\in A$ to be the polynomial part  of the Laurent expansion of $a$. Then $a-P_2\in \frac1T\calo_\infty$,
hence $|a-P_2|_\infty\le q^{-1}$.

%Write $b=\sum_{n} b_nT^n$ with $b_n\in\F_q$ and define $P_1:=\sum_{n=-c}^{d} b_n\,T^{n+c}\in A$.
{  Let $P_1\in A$ be the polynomial part of $bT^c$.
Since
\[
|bT^c|_\infty
\le q^{d+c-\deg_T(D)/2}\le q^{d+1/2},
\]
the integrality of the Laurent exponents gives $\deg P_1\le d$.
Likewise, $|a|_\infty\le q^d$ gives $\deg P_2\le d$.}
Then
\[
b-\frac{P_1}{T^c}\in T^{-(c+1)}\calo_\infty
\qquad\text{so}\qquad
\Bigl|b-\frac{P_1}{T^c}\Bigr|_\infty\le q^{-(c+1)}.
\]
Set $P:=P_2+\frac{P_1}{T^c}\sqrt D\in C_{d}$. By the ultrametric inequality,
\[
|z-P|_\infty
\le \max\Bigl\{|a-P_2|_\infty,\ \Bigl|b-\frac{P_1}{T^c}\Bigr|_\infty\,|\sqrt D|_\infty\Bigr\}
\le \max\{q^{-1},\ q^{-(c+1)}|\sqrt D|_\infty\}.
\]
Since $|\sqrt D|_\infty=q^{\deg(D)/2}$ and $c=\lceil \deg(D)/2\rceil$, we have
$q^{-(c+1)}|\sqrt D|_\infty\le q^{-1}$, hence $|z-P|_\infty\le q^{-1}$.
Therefore $z\in U_P^\rho$, proving the covering.

In particular,
$\#\{P_i\in A:\deg P_i\le d\}=q^{d+1}$, so
$\#C_{d}=(q^{d+1})^2=q^{2d+2}$. 

{   If $U_P^\rho$ meets the truncated set at $z$, then
$|P-z|_\infty\le q^{-1}<1\le |z|_i$.
Thus, for every $t\in F_\infty$, ultrametricity gives
\[
|P-t|_\infty=|z-t|_\infty\ge1.
\]
Consequently $\{w\in\C_\infty:|w-P|_\infty\le q^{-1}\}$
lies in $\Omega$ and equals $U_P^\rho$, which is therefore a closed affinoid ball.}
\end{proof}

\begin{lem}\label{l:carlitz-cusp}
Let $\exp_{\ccal}$ be the Carlitz exponential
normalized so that
$\ker(\exp_{\ccal})=\widetilde\pi A$.
Then, for every $z\in\Omega$, we have
\[
\left|
\widetilde\pi\,
\exp_{\ccal}(\widetilde\pi z)^{-1}
\right|_\infty
\le
|z|_i^{-1}.
\]
In particular,
\[
\left|
\widetilde\pi\,
\exp_{\ccal}(\widetilde\pi z)^{-1}
\right|_\infty
\rightarrow 0
\qquad
\text{as}\quad
|z|_i\rightarrow\infty.
\]
\end{lem}

\begin{proof}
By \cite[(2.2.v)]{gekelerj}, we have
\[
\frac{1}{\exp_A(w)}
=
\sum_{a\in A}\frac{1}{w-a}
\qquad
\text{for }
w\notin A.
\]
Since $\exp_{\ccal}$ is the exponential associated with the lattice
$\widetilde\pi A$, the scaling relation for lattice exponentials gives
$\exp_{\ccal}(\widetilde\pi z)=
\widetilde\pi\,\exp_A(z)$.
Thus
\[
\widetilde\pi\,
\exp_{\ccal}(\widetilde\pi z)^{-1}
=
\frac{1}{\exp_A(z)}
=
\sum_{a\in A}\frac{1}{z-a}.
\]
For every $a\in A\subset F_\infty$, we have
$|z-a|_\infty\ge |z|_i$,
and hence
$\left|\frac{1}{z-a}\right|_\infty
\le |z|_i^{-1}$.
The ultrametric inequality therefore gives
\[
\left|
\sum_{a\in A}\frac{1}{z-a}
\right|_\infty
\le
\sup_{a\in A}
\left|\frac{1}{z-a}\right|_\infty
\le
|z|_i^{-1}.
\]
\end{proof}

\subsubsection{Drinfeld modular curves}

Let $N$ be a non-zero ideal of $A$ and let $N^{-1}$ denote its inverse fractional ideal in $F$. A {\em level $N$ structure} on a Drinfeld module $\cald$ of rank $2$ over $L$ is an $A/N$-module morphism 
\[
\alpha\colon \left(N^{-1}/A\right)^2\rightarrow \cald[N](L)
\]
such that, as divisors over $\G_{a}$, we have $\cald[N]=\sum_n[\alpha(n)]$, for $n\in\left(N^{-1}/A\right)^2$.

As in \cite{drinfellipt}, consider the moduli problem associating, to an $A$-scheme $S$, the isomorphism classes of the pairs $(\cald,\alpha)$. If $N$ is divisible by at least $2$ primes, then this functor is representable by an affine $A$-scheme $Y(N)$ of finite type and smooth of relative dimension $1$ away from the primes dividing $N$. Without any restriction on $N$, we obtain a coarse moduli space. In particular, for $N=(1)$, i.e., without level structure, we denote by $Y(1)$ the coarse moduli space whose points parametrize isomorphism classes of Drinfeld modules of rank $2$. 
\\

Let us recall that, for any scheme $X$ of locally finite type over a complete, non-archimedean field $L$ of characteristic $p$, there is a rigid analytic space whose $L$-points coincide with the $L$-points of $X$.  In particular, each connected component of rigid analytification of $Y(N)$ is isomorphic to $\Omega/\Gamma$ for $\Gamma$ a congruence subgroup of $\gl_2(A)$. 
These curves can be compactified by adding a finite set of cusps given by $\p^1(F)/\Gamma$, and we denote their compactification by $X(N)$.
Intuitively, as noticed in \cite[Remark 2.16]{bpr}, for $z\in\calf$, approaching the cusp through $j(z)$ corresponds to $|z|_i\to\infty$.

As usual, we denote by $\Gamma_0(N)$ the subgroup of $\gl_2(A)$ of matrices whose first entry in the bottom row is a multiple of $N$.

\subsubsection{The $j$-invariant}\label{subsec:j-invariant}

Following \cite[3, p.47]{gekeler}, we recall that a {\em Drinfeld modular form of weight $k\ge 1$ and type $0$} for $\gl_2(A)$ is a rigid analytic function 
$f\colon\Omega\rightarrow \C_\infty$
such that:
\begin{itemize}
    \item $f(\gamma z)=(cz+d)^kf(z)$,  where $(c,d)$ is the bottom row of $\gamma\in\gl_2(A)$;
    
    \item $f(z)$ is holomorphic at the cusp of $\gl_2(A)$.
\end{itemize}

More generally, a form of type $m\in\Z/(q-1)\Z$ satisfies
$f(\gamma z)=\det(\gamma)^{-m}(cz+d)^kf(z)$, with the same
holomorphy condition at the cusp.

Let $z\in\Omega$ and let $\Lambda_z=A+zA$. Prototypically, for any positive integer $k$ divisible by $q-1$, we define the Eisenstein series $E_k(z)=\sum_{\lambda\in \Lambda_z -\{0\}}\lambda^{-k}$, which is a Drinfeld modular form of weight $k$.
Let {  $\cald_z$} also {  denote} the Drinfeld module of rank $2$ associated to $\Lambda_z$, which is uniquely determined by
\[
\varphi_T=\iota(T) + g(z)\uptau + \Delta(z)\uptau^2,
\]
for $\Delta(z)\in\C_\infty^\times$. In particular, $g$ and $\Delta$ are Drinfeld modular forms of weight $q-1$ and $q^2-1$, respectively. Moreover, $\Delta$ is nowhere vanishing on $\Omega$. We then define the {\em $j$-invariant} of $\cald$ as
\[
j(z)=\dfrac{g(z)^{q+1}}{\Delta(z)}.
\]
Let $h:=P_{q+1,1}$ be Gekeler's Poincar\'e series of weight
$q+1$ and type $1$. Gekeler proved that the graded
$\C_\infty$-algebra of Drinfeld modular forms of all types for
$\gl_2(A)$ is
$M(\gl_2(A))=\C_\infty[g,h]$, see
\cite[Theorem~5.13]{gekelerj}. With Gekeler's normalized discriminant
$\Delta_{\mathrm{Gek}}:=\widetilde{\pi}^{\,1-q^2}\Delta$
one has
\[
\Delta_{\mathrm{Gek}}=-h^{q-1}
\]
by \cite[Theorem~9.1]{gekelerj}. Equivalently, with our convention that
$\cald_z$ is attached to the lattice $\Lambda_z=A+zA$, so that
$\varphi_T=T+g(z)\uptau+\Delta(z)\uptau^2$,
the discriminant coefficient satisfies
\[
\Delta(z)=-\widetilde{\pi}^{\,q^2-1}h(z)^{q-1}
\]
see the discussion immediately after \cite[Corollary~6.3]{gekelerquasi}. In particular, $h$ is nowhere vanishing on $\Omega$.

Note that $j$ is a modular function on $\Omega$ of weight $0$ but not a modular form, as it is not holomorphic at the cusp. However, since $g$ and $\Delta$ are analytic and $\Delta$ is nowhere vanishing on $\Omega$, it follows that $j$ is analytic as well. In particular, the $j$-function has the following $Q$-expansion at infinity
\begin{equation}\label{jexpansion}
j(z)=-\dfrac{1}{Q^{q-1}}+r(Q)
\end{equation}
where $r(Q)\in A[\![Q]\!]$ (see \cite[Section 6]{gekelerj}).

Two elements of $\Omega$ determine homothetic lattices if and only if they belong to the same $\gl_2(A)$-orbit. Since homothety classes of $A$-lattices of rank $2$ are in bijection with $\gl_2(A)\backslash \Omega$, and $j$ is $\gl_2(A)$-invariant, we have that $j$ descends to the rigid analytic isomorphism
\[
\gl_2(A)\backslash \Omega\simeq \A^1_{\C_\infty}.
\]
For $z_1,z_2\in\Omega$, let $\sigma\in\gl_2(F)$ be such that there exists a cyclic isogeny  between the Drinfeld modules corresponding to $z_1$ and $z_2$, i.e., $z_1=\sigma z_2$. The degree of such an isogeny is represented by the monic generator $N$ of the ideal $\det(a\sigma)A$, where $a\in F^\times$ is chosen so that the entries of $a\sigma$ are in $A$ with no common factors.
In the level notation we identify $N$ with the ideal $(N)$.
For $\sigma_1,\sigma_2\in\gl_2(F)$, the image of the rigid analytic map
\begin{equation}\label{e:modular-corr}
    \varrho\colon\Omega\rightarrow \A^2_{\C_\infty},\qquad z\mapsto (j(\sigma_1 z),j(\sigma_2 z))
\end{equation}
is a {\em pure modular curve} $Y_0'(N)$, which is indeed the image of $Y_0(N)$ in $\A^2_{\C_\infty}$, where $N$ is the degree of $\sigma_2\sigma_1^{-1}$. This image can be described as the zero locus of the irreducible, symmetric, {\em modular} polynomial $\Phi_N(X,Y)\in A[X,Y]$ whose roots are the $j$-invariants of Drinfeld modules of $\rank$-$2$ linked by a cyclic isogeny of degree $N$ (see \cite[Theorem 3.2]{bae}). 
\\Note that in general a pure modular curve is birational to $Y_0(N)$ but not smooth.

\subsubsection{CM Drinfeld modules}

 Let $\cald$ be a Drinfeld $A$-module of $\rank$-$2$ over $\C_\infty$. Its endomorphism ring $\End(\cald)$ is the centralizer of $\varphi(A)$ in $\C_\infty\{\uptau\}$. Let $K$ be a quadratic extension of $F$ such that there are no embeddings of $K$ into $F_\infty$, that is $K$ is a CM (i.e., {\em imaginary} quadratic) extension of $F$. Equivalently, one can ask that there is a unique place $\widetilde{\infty}$ of $K$ above the place $\infty$ of $F$. 
 We say that $\cald$ has {\em CM by $K$} whenever $\End(\cald)$ is an order $\calo_c$ of conductor $c$ in $K$, i.e., $\calo_c \simeq A+c\calo_K$.

If $q$ is odd, then $K=F(\sqrt{D})$ for $D\in A$ square-free. 
On the other hand, the case of even characteristic presents some differences, which we recall, following \cite[Section 2.2.2]{angles} and the references therein. If $q=2^s$, for some positive integer $s$, there is exactly one inseparable quadratic extension, namely $F_{2^s}(\sqrt{T})$. 
\\
Let now $K/F$ be such a separable quadratic extension in even characteristic. Then there exists $\xi\in F^{\text{alg}}$ such that $K=F(\xi)$ and which satisfies the Artin--Schreier equation
\[
\xi^2+\xi=f
\]
for $f\in F$. For $f=\frac{B}{C}$, with $B,C\in A$ we say that $f$ is {\em normalized} if:
\begin{enumerate}
\item $C$ monic, $B$ and $C$ coprime, and $\deg B\ge \deg C$;
\item $C=\prod_{i=1}^r P_i^{e_i}$ with $P_1,\dots,P_r\in A$ distinct monic irreducibles and
$e_i\equiv 1\pmod 2$ for all $i$;
\item either
\begin{enumerate}
\item[(i)] $\deg B>\deg C$ and $\deg B-\deg C\equiv 1\pmod 2$ (then $\infty$ ramifies in $K/F$), or
\item[(ii)] $\deg B=\deg C$ and $X^2+X+\text{sgn}(B)$ is irreducible in $\F_q[X]$, where $\text{sgn}(B)$ is the leading coefficient of $B$ (then $\infty$ is inert in $K/F$).
\end{enumerate}
\end{enumerate}

Moreover, it is well known that $K=K'$ if and only if $\xi'=\xi+g$ for some $g\in F$; then $(\xi')^2+\xi'=f+(g^2+g)$, so $f'-f=g^2+g$.

Set $G:=\prod_{i=1}^r P_i^{n_i}$, where $n_i$ are the smallest integers such that $n_i>e_i/2$. Then the maximal order of $K$ is given by
$\calo_K \;=\; A\oplus AG\xi$,
and its discriminant  is $D_K:=G^2$. Thus the unique order $\calo_c$ in $K$ of conductor $c$ is given by $\calo_c=A\oplus AcG\xi$, and its discriminant is $D_c:=c^2G^2=c^2D_K$. Unlike the odd characteristic case, the discriminant does not uniquely determine the quadratic field $K$. Nonetheless, for $K$ fixed, an order is uniquely determined by its discriminant. Let $\text{rad}(G):=\prod_{i=1}^rP_i$. With this fixed choice of $\xi$, define
\[
S_c(\xi):=
\left\{
\frac{b+cG\xi}{a}\ \middle|\ 
\begin{array}{l}
a,b\in A,\ a\ \text{monic},\ \exists\,\gamma\in A\ \text{with } a\gamma=b^2+bcG+c^2\text{rad}(G)B,\\[2pt]
|b|_\infty<|a|_\infty\le |\gamma|_\infty,\ (a,b,c)=1
\end{array}
\right\}.
\]
By \cite[Lemma~2.3(1)]{angles}, we have $S_c(\xi)\subset\calf$.
\\

Note that, as $\C_\infty$ has infinite degree over $F_\infty$, the action of $\gl_2(F_\infty)$ on $\Omega$ is not transitive. We say that $z\in\Omega$ is {\em locally quadratic} if $F_\infty(z)$ has degree $2$ over $F_\infty$. Thus we see that the Drinfeld module associated to $z$ has CM if and only if $[F(z):F]=2$. By the CM theory for Drinfeld modules (as, for instance, in \cite[Lemma~2.1]{angles}), we  have that $j(z)$, for $z$ with $K=F(z)$ quadratic and $\End(\cald_z)=\calo_c$, is integral over $A$ and $K(j(z))$ is the ring class field of $\calo_c$ \cite[Section~4.2]{wei}.
For $c=1$, this is the Hilbert class field of $K$, i.e.,
the maximal unramified abelian extension of $K$ in
$\C_\infty$ in which $\widetilde{\infty}$ splits completely. Moreover we also have
 \begin{equation}\label{e:cmorbit}
 \Gal\big(K(j(z))/K\big)\simeq \pic(\calo_c)\ ,
 \end{equation}
 see \cite[Lemma~2.1]{angles}. For $q$ odd, the class number satisfies the following effective
Siegel estimate. Namely, as discussed in \cite[Section~3.1]{breuerAO}, there exist two effectively computable positive constants $B_\varepsilon$ and $C_\varepsilon$ such that
 \begin{equation}\label{siegelclassnumber}
     B_\varepsilon\big|D_c\big|_\infty^{1/2-\varepsilon}\le \#\pic(\calo_c)\le C_\varepsilon\big|D_c\big|_\infty^{1/2+\varepsilon}
 \end{equation}
for every $\varepsilon>0$. For an explicit upper bound, see \cite[Proposition 2.11]{angles}.

\begin{lem}\label{l:heightfundamentaldomain}
    Let $K$ be a  separable imaginary quadratic extension of $F$ and let $z$ be a quadratic point in the quadratic fundamental domain $\calf_K$, and let $\calo_c=\End(\cald_z)$, of discriminant $D_c$. Then:
    \begin{enumerate}
    \item if $q$ is odd or $q$ even with $\infty$ inert in $K$, one has $H(z)\le |D_c|_\infty^{1/2}$;
    \item if $q$ is even with $\infty$ ramified in $K$, one has $H(z)\le |D_c|_\infty^{1/2}|\xi|_\infty$.
\end{enumerate}
\end{lem}

\begin{proof}
We may use reduced CM representatives because $h$ is unchanged
between $\gl_2(A)$-equivalent points $z,z'\in\calf_K$.
Indeed, $|z|_\infty=|z'|_\infty$, as recalled above.
If $|z|_i>1$, the imaginary part transformation formula forces
$z'=uz+b$ with $u\in\F_q^\times$ and $b\in A$.
This preserves finite pole orders, so $h(z)=h(z')$.
If $|z|_i=1$, the transformation belongs to $\gl_2(\F_q)$,
which preserves projective height.  For $q$ odd, this was shown in \cite[Lemma 5.3]{ran}. 
Let us therefore assume $q=2^s$ for a positive integer $s$. If $\alpha$ is quadratic over $F$ with conjugate $\bar\alpha$ and primitive minimal polynomial
$aX^2+bX+c\in A[X]$, then, by \cite[Lemma 3.7]{ran} for the logarithmic height
\begin{equation}\label{eq:quadheight}
2h(\alpha)=\log_q|a|_\infty+\log_q^+|\alpha|_\infty+\log_q^+|\bar\alpha|_\infty.
\end{equation}
Since $K/F$ is imaginary, there is a unique place $\widetilde\infty\mid\infty$, hence  $|\bar\alpha|_\infty=|\alpha|_\infty$. In particular, for $\alpha\in\calf_K\subset\calf$ we have
$|\alpha|_\infty\ge 1$ and \eqref{eq:quadheight} becomes
\begin{equation}\label{eq:quadheight_simpl}
2h(\alpha)=\log_q\bigl(|a|_\infty\,|\alpha\bar\alpha|_\infty\bigr)=\log_q|c|_\infty
\end{equation}
using $\alpha\bar\alpha=c/a$.

Write $K=F(\xi)$ with $\xi^2+\xi=B/C$ in its normalized form, 
and let $G$ be the associated
polynomial, so that $D_c=c^2G^2$.
Write $A^+$ for the set of monic polynomials in $A$.
Let $z\in S_c(\xi)$, so $z=(b+cG\xi)/a$ with $a\in A^+$, $b\in A$ and $|b|_\infty<|a|_\infty$. Let $\bar\xi=\xi+1$ and $\bar z=(b+cG\bar\xi)/a$; then $\bar z$ is the Galois conjugate of $z$ and
$|z|_\infty=|\bar z|_\infty\ge 1$. We have that $z$ satisfies
\[
(aX-b)^2+(aX-b)\,cG+c^2\text{rad}(G)B=0,
\]
whose left-hand side has leading coefficient $a^2$ and constant term $a\gamma=b^2+bcG+c^2\text{rad}(G)B$ with $\gamma\in A$, as in the definition of $S_c(\xi)$.
After dividing this polynomial by the gcd of its coefficients,
\eqref{eq:quadheight_simpl} gives
$2h(z)\le\log_q|a\gamma|_\infty$. By ultrametricity and $|b|_\infty<|a|_\infty$ we have $|b|_\infty\le |cG|_\infty\cdot|\xi|_\infty$, as
$|z|_\infty\ge 1$ forces $|a|_\infty\le |b+cG\xi|_\infty=|cG\xi|_\infty$. Hence
\[
|a\gamma|_\infty
\le \max\{|b|_\infty^2,\ |b|_\infty|cG|_\infty,\ |c^2\text{rad}(G)B|_\infty\}
\le |cG|_\infty^2\,|\xi|_\infty^2,
\]
using $|c^2\text{rad}(G)B|_\infty=|cG|_\infty^2\,|B/C|_\infty$ and $|\xi|_\infty^2=\max\{1,|B/C|_\infty\}$.
Therefore 
$$H(z)=q^{h(z)}\le |cG|_\infty\,|\xi|_\infty=|D_c|_\infty^{1/2}\,|\xi|_\infty.$$
If $\infty$ is inert in $K$ then $|B/C|_\infty=1$, so $|\xi|_\infty=1$ and $H(z)\le |D_c|_\infty^{1/2}$.
\end{proof}
\begin{comment}
{ 
\begin{rmk}
Assume $q=2^s$ and write $K=F(\xi)$ with $\xi^2+\xi=f\in F$.
For any $u\in F$ the element $\xi':=\xi+u$ generates the same $K$ and satisfies
\[
(\xi')^2+\xi'=f+(u^2+u).
\]
In particular one can make $|\xi'|_\infty$ arbitrarily large by choosing $u$ with a large pole at $\infty$.
\\
As in \cite[Section~2.2.2]{angles} one fixes a \emph{reduced} Artin--Schreier equation, i.e., $f$ is chosen so that at every place $v$ either $\ord_v(f)\ge 0$ or $\ord_v(f)=-m_v$ with $m_v$ odd. After such a reduction, the size of $|\xi|_\infty$ is no longer arbitrarily: it is determined by the reduced $f$, in particular by the pole order of $f$ at $\infty$.

The quadratic fundamental domain  $S_c(\xi)$ depends on the chosen reduced generator $\xi$, as its elements $z\in S_c(\xi)$ are written as
\[
z=\frac{b+cG\xi}{a}
\]
for $a,b\in A$ with $a$ monic and $|b|_\infty<|a|_\infty$,
with $G$ built from the reduced denominator of $f$.
Replacing $\xi$ by $\xi'=\xi+u$ changes the numerator to $b'=b+cGu$, so for $|u|_\infty$ large the condition
$|b'|_\infty<|a|_\infty$ does not hold in general. In conclusion, a point reduced with respect to $\xi$ need not be reduced with respect to $\xi'$.  one cannot freely replace the chosen reduced generator $\xi$ by another
generator $\xi'=\xi+u$ (with $|u|_\infty$ large) in the estimates for the points $z\in S_c(\xi)$. Indeed,
the definition of $S_c(\xi)$ depends on $\xi$, and after changing $\xi$ the same point $z$ does not need to lie in $S_c(\xi')$.
\end{rmk}
}
\end{comment}

\subsubsection{Taguchi height}

Following~\cite{tag} and~\cite{wei}, we recall the Drinfeld module analogue of the Faltings height, called {\em Taguchi} height, which thus measures the arithmetic complexity of points in the moduli space of Drinfeld modules.
\\

We begin by recalling the original construction of the Taguchi height, as in \cite[Sections 2 and 5]{tag}. 
\\Let $S$ be an integral normal scheme of finite type over $A$ with function
field $L$, a finite extension of $F$, and let $\cald$ be a Drinfeld $A$-module over $L$.  A {\em model} $\ncal$ of $\cald$ over $S$ is an $A$-module scheme $(N,\varphi)$ over $S$ with a morphism  
\[
f\colon \ncal\times_S\Spec L\rightarrow\cald
\]
which is an isomorphism of Drinfeld
modules over $L$. A model $\ncal$ over $S$ of $\cald$ is {\em minimal} if, for any other model $\calm=(M,\varphi',f')$,
there exists a unique morphism $\calm\rightarrow \ncal$ inducing an isomorphism on the generic fiber, which is required to be compatible with $f$ and $f'$. Rephrasing \cite[Proposition 2.2]{tag}, if $S$ is a factorial Noetherian scheme then the minimal model exists and is unique up to isomorphism.

%For $\widetilde{\infty}|\infty$, we embed $L\hookrightarrow\C_\infty$ via $\widetilde{\infty}$.

A {\em metrized line bundle}  on $\Spec \calo_L$ is a pair $(\call,\|\cdot\|)$ consisting of:
\begin{itemize}
\item a projective $\calo_L$-module $\call$ of $\rank$-$1$;
\item for every place $\widetilde{\infty}\mid\infty$, a norm
\[
\|\cdot\|_{\widetilde{\infty}}\colon \call \otimes_{\calo_L} L_{\widetilde{\infty}} \rightarrow \R.
\]
\end{itemize}
The {\em degree} $\deg(\call,\|\cdot\|)$ of a metrized line bundle on $\calo_L$ is then defined as
\[
\deg(\call,\|\cdot\|)
:= 
\log_q \#(\call/\ell \calo_L)
\;-\;
\sum_{\widetilde{\infty}\mid\infty} n_{\widetilde{\infty}} \log_q \|\ell\|_{\widetilde{\infty}},
\]
for some $\ell\in \call-\{0\}$, where $n_{\widetilde{\infty}}$ is the local degree at $\widetilde{\infty}$. By the product formula 
it is independent of the choice of $\ell$.

Let now $\cald$ be a Drinfeld $A$-module of $\rank$-$2$ over $L$ and let $\ncal = (N,\varphi,f)$
be its minimal model over $\calo_L$. Let $e$ denote its unit section, and set $$\omega_{\ncal/\calo_L}:= e^*\bigl(\Omega^1_{\ncal/\calo_L}\bigr).$$

For a place $\widetilde{\infty}$ of $L$ such that $\widetilde{\infty} | \infty$, via the embedding $\widetilde{\infty}\colon L\hookrightarrow\C_\infty$ we denote by $\Lambda_{\widetilde{\infty}}$ the $A$-lattice in $\C_\infty$ associated to $\cald$ and by $D_A(\Lambda_{z_{\widetilde{\infty}}})$ the lattice discriminant of $\Lambda_{\widetilde{\infty}}$ (as an $A$-lattice in $\C_\infty$, see \cite[pp.1064-1065]{wei} for the detailed construction).
Let  $\cald_{\widetilde{\infty}}$ be the Drinfeld module over $L_{\widetilde{\infty}}$ obtained
by extension of scalars $\calo_L\to L_{\widetilde{\infty}}$, with corresponding lattice $\Lambda_{\widetilde{\infty}}$. For $X$  the coordinate function of $\G_{a,L_{\widetilde{\infty}}}$, then $\text{d}X$ is a generator of
$\omega_{\ncal/L_{\widetilde{\infty}}}$, on which one defines the following metric
\[
\|\text{d}X\|_{\widetilde{\infty}} := D_A(\Lambda_{\widetilde{\infty}}).
\]

Let $\cald$ be a Drinfeld $A$-module of rank $2$ over $L$ and
$\ncal$  its minimal model over $\calo_L$.
The \emph{Taguchi height} of $\cald$ over $L$ is
\begin{equation}\label{e:tagdef}
h_{\mathrm{Tag}}(\cald)
\;:=\;
\frac{1}{[L:F]}
\deg\bigl(\omega_{\ncal/\calo_L},\|\cdot\|\bigr).
\end{equation}

 Following~\cite[Section 4]{wei}, we give an equivalent description of the Taguchi height in terms of local contributions. For each place $w$ of $L$ such that $w\nmid\infty$, denote by $\F_w$ the residue field of $\calo_L$  at $w$. 
 Set
\[
\ord_w(\varphi_a)\;:=\;\min_{\substack{i\ge1\\g_{i,a}\ne0}}\Bigl\{\ord_w(g_{i,a})/(q^i-1)\Bigr\},
\qquad
\ord_w(\cald)\;:=\;\min_{a\in A-\{0\}}\ord_w(\varphi_a)
\]
{where $g_{i,a}$ is the coefficient of $\tau^i$ in $\varphi_a$ (and an empty minimum is defined as $+\infty$).}

The  {\em local Taguchi height} at $w$ of $\cald$ is defined as
\[
h_{\text{Tag},w}(\cald/L):=-[\F_w:\F_q]\cdot \lfloor{\ord_w(\cald)}\rfloor.
\]
On the other hand, for a place $\widetilde{\infty}$ of $L$ such that $\widetilde{\infty} | \infty$, we then set
\[
h_{\text{Tag},\widetilde{\infty}}(\cald/L):=-[L_{\widetilde{\infty}}:F_\infty]\cdot \log_q D_A(\Lambda_{\widetilde{\infty}}).
\]
As in \cite{wei}, the  Taguchi height of $\cald$ over $L$ is then
\[
h_{\text{Tag}}(\cald/L):=\dfrac{1}{[L:F]}\left( \sum_{w\nmid \infty}h_{\text{Tag},w}(\cald/L)+\sum_{\widetilde{\infty}|\infty}h_{\text{Tag},\widetilde{\infty}}(\cald/L) \right).
\]

A Drinfeld module $\cald$  over $L$ has {\em stable reduction} at a place $v$ if it is isomorphic to a Drinfeld module  defined over the valuation ring  $v$ whose reduction modulo the maximal ideal of the valuation ring is a Drinfeld module of $\rank>0$ over the residue field at $v$. Note that if $\cald$ has stable reduction at $w$, then $\ord_w(\cald)$ is an integer (see \cite[Remark~4.2.(1)]{wei}).
\\As every Drinfeld module has everywhere potentially stable reduction (possibly after a finite extension, see \cite[Lemme~2.10]{daviddenis}), we define the {\em stable} Taguchi height as
\[
\widetilde{h}_{\text{Tag}}(\cald):=\dfrac{-\ln q}{[L:F]}\left(\sum_{w\nmid \infty}[\F_w:\F_q]\ord_w(\cald)+\sum_{\widetilde{\infty}|\infty}[L_{\widetilde{\infty}}:F_\infty]\log_q D_A(\Lambda_{\widetilde{\infty}}) \right).
\]
The stable height uses natural logarithms, whereas
$h_{\mathrm{Tag}}$ uses logarithms to base $q$.
With these conventions, if $\cald$ has stable reduction
at every finite place of $L$, then
$\widetilde h_{\mathrm{Tag}}(\cald)
=(\ln q)\,h_{\mathrm{Tag}}(\cald/L)$.
We remark that the stable Taguchi height is invariant under $F^{\text{alg}}$-isomorphisms of Drinfeld modules.

\subsubsection{Upper bound for the height of a CM Drinfeld module}
For the following CM height formula and Lemma~\ref{colmezbound}, we shall assume that $q$ is odd.

We now consider the case of a CM extension $K$ of $F$, and let $\cald$ have CM by $\calo_c$.
Consider the following partial zeta function
\[
\zeta_{\calo_c}(s)=\sum_{\gotn}\#\big(\calo_c/\gotn\big)^{-s}
\]
where $\gotn$ is an invertible ideal of $\calo_c$. Let
\begin{equation}\label{e:weidiff}
    \mathfrak d(\calo_c):=c^2\cdot\prod_{\substack{\text{prime}\ p\in A \\\text{ramified in}\ K}} p\;.
\end{equation}
In \cite[Corollary 4.5]{wei} Wei proved that, for the stable Taguchi height of a CM Drinfeld $A$-module of $\rank$-$2$, the following  Colmez-type formula
\begin{equation}\label{colmez}
    \widetilde{h}_{\text{Tag}}(\cald)=-\dfrac{1}{4}\ln|\mathfrak d(\calo_c)|_\infty-\dfrac{1}{2}\dfrac{\zeta'_{\calo_c}(0)}{\zeta_{\calo_c}(0)}
\end{equation}
holds. If $q$ is odd, then $\mathfrak d(\calo_c/A)=D_K\cdot c^2$, i.e., is the discriminant of $\calo_c$. 
\begin{comment}
On the other hand, for $q$ even we have what follows. Let $K=F(\xi)$ with an Artin--Schreier equation
$\xi^2+\xi=f=\frac{B}{C}$ such that $B,C\in A=\F_q[T]$ with  $(B,C)=1$, where $f$ is in reduced form and
\[
C=\prod_{i=1}^r P_i^{e_i}\]
for $P_i$ distinct and irreducible and $e_i$ odd.

We recall that, for a finite prime $P\in A$, one has that $P$ ramifies in $K$ if and only if $v_P(f)<0$, that is, $P$ divides $C$.
This shows that the finite ramified primes are exactly $\{P_1,\dots,P_r\}$ and therefore
\[
\prod_{\substack{p\in A\\ \mathrm{ramified }}}p=\prod_{i=1}^r P_i
\]
Recall we define $G:=\prod_{i=1}^r P_i^{n_i}$ with $n_i=(e_i+1)/2$, so that $G^2=\prod_{i=1}^r P_i^{e_i+1}$.
which has the same prime support $\{P_1,\dots,P_r\}$ as the squarefree product above.
for the order $\calo_c$ one has
\[
\mathfrak d(\calo_c)=c^2\prod_{i=1}^r P_i
\;\;\;\;\text{and}\;\;\;\;
D_c=c^2G^2=c^2\prod_{i=1}^r P_i^{e_i+1}
\]
so $\mathfrak d(\calo_c)$ divides $D_c$ and hence $|\mathfrak d(\calo_c)|_\infty\le |D_c|_\infty$.
\\
\end{comment}

The next Lemma is the Drinfeld module version of \cite[Corollary~3.3]{tsimao}.

\begin{lem}\label{colmezbound}
Assume that $q$ is odd. Let $\cald$ be a Drinfeld $A$-module over $F^{\text{alg}}$ with CM by $\calo_c$. Then, for every $\varepsilon>0$,
    \[
    \widetilde{h}_{\text{Tag}}(\cald)\le C_\varepsilon |D_c|_\infty^{\varepsilon}
    \]
    for an effectively computable positive constant $C_\varepsilon$.
\end{lem}
\begin{proof}
    We proceed by bounding both terms on the right hand side of the Colmez-type formula \eqref{colmez} by $O_\varepsilon(|D_c|_\infty^\varepsilon)$. Indeed, for the first term, such a bound is trivial. 
Consider
\[
\zeta_{\calo_c}^*(s):=(1-q^{1-s})\zeta_{\calo_c}(s),
\]
which removes the pole at $s=1$.
The functional equation for the maximal order zeta function (see \cite[Theorem~5.9, pp.~53--54]{rosen}),
together with the factors at $\infty$ and at primes dividing
$c$, gives
\[
\frac{(\zeta_{\calo_c}^*)'(1)}{\zeta_{\calo_c}^*(1)}
+\frac{\zeta'_{\calo_c}(0)}{\zeta_{\calo_c}(0)}
=O_q\bigl(1+\ln|D_c|_\infty\bigr).
\]
The conductor factors contribute $O_q(\ln|c|_\infty)$
to these logarithmic derivatives.
It therefore suffices to bound
\[
\left|
\frac{(\zeta_{\calo_c}^*)'(1)}{\zeta_{\calo_c}^*(1)}
\right|.
\]
    By the class number formula and Siegel's bound \eqref{siegelclassnumber}
    \[
   \zeta_{\calo_c}^*(1)
   \ge C_\varepsilon |D_c|_\infty^{-\varepsilon}.
   \]
   On the other hand, Cauchy's integral formula on a sufficiently
small circle centered at $1$, together with the convexity
estimate and the conductor factors, gives
\[
\left|(\zeta_{\calo_c}^*)'(1)\right|
\le C_\varepsilon |D_c|_\infty^\varepsilon.
\]
Applying the last two estimates with $\varepsilon/2$,
and absorbing the logarithmic error, we conclude.
    %On the other hand, by Cauchy's integral theorem we have that $\zeta_{\calo_c}'(1)$ is an average of $\zeta_{\calo_c}(s)$ over a complex ball of center $1$ and radius $\varepsilon>0$ arbitrarily small. By the Riemann hypothesis over function fields (see, for instance, \cite[Appendix]{rosen}) and the convexity estimate we have
    %\[
    %\zeta_{\calo_c}'(1)\le C_\varepsilon |D_c|^{\varepsilon}_\infty.
    %\]
    %As $\varepsilon$ is arbitrary, we conclude.
\end{proof}

The following result adapts  \cite[Propositions~1.1 and~2.1]{silv} to the
Drinfeld module setting.
\\
For every place $\widetilde{\infty}$ of $L$ above $\infty$, we write
$\cald_{\widetilde{\infty}}:=
\cald\otimes_{L}\C_\infty$.
By rigid-analytic uniformisation via the Drinfeld exponential
(see, for instance, \cite[(5.2.6), p.~290]{pap}), there is a rank-$2$
$A$-lattice $\Lambda_{\widetilde{\infty}}\subset\C_\infty$ and an exact
sequence
\[
0\rightarrow \Lambda_{\widetilde{\infty}}
\rightarrow \C_\infty
\xrightarrow{\exp_{\Lambda_{\widetilde{\infty}}}}
\cald_{\widetilde{\infty}}(\C_\infty)
\rightarrow 0.
\]
Thus
\[
\cald_{\widetilde{\infty}}(\C_\infty)
\simeq
\C_\infty/\Lambda_{\widetilde{\infty}}\ .
\]
Pick $z_{\widetilde{\infty}}\in\Omega$ such that
$\Lambda_{\widetilde{\infty}}$ is homothetic to
$\Lambda_{z_{\widetilde{\infty}}}
=A+z_{\widetilde{\infty}}A$.
Let us set $n_{\widetilde{\infty}}:=[L_{\widetilde\infty}:F_\infty]$.

\begin{prop}\label{p:taguchi-weil}
Let $L/F$ be a finite extension and let $\cald$ be a Drinfeld
$A$-module of rank $2$ over $L$. Let $\gotd$ denote its minimal
discriminant ideal. Then
\begin{equation}\label{e:silv}
(q^2-1)[L:F]\,h_{\mathrm{Tag}}(\cald)
=
\log_q N_{L/F}(\gotd)
-
\sum_{\widetilde{\infty}\mid\infty}
n_{\widetilde{\infty}}
\left(
\log_q\!\left|\Delta(z_{\widetilde{\infty}})\right|_{\widetilde{\infty}}
+
(q^2-1)\,\log_q D_A\!\left(\Lambda_{\widetilde{\infty}}\right)
\right),
\end{equation}
where $N_{L/F}(\gotd)=\#(\calo_L/\gotd)$.

Moreover, if $\cald$ has stable reduction at every finite place of $L$, then
\begin{equation}\label{e:silv2}
\left|
h\big(j(\cald)\big)
-
(q^2-1)\,h_{\mathrm{Tag}}(\cald)
\right|
\le
\frac{q^2-1}{2}\log_q\!\left(1+h\big(j(\cald)\big)\right)
+O_q(1).
\end{equation}

\end{prop}

\begin{proof}
We first prove \eqref{e:silv}.  Recall that $X$ denotes the coordinate of $\G_{a,L}$ and that $\varphi_T=\iota(T)+g\uptau +\Delta\uptau^2$. 
Let $\ncal$ be the minimal model of $\cald$ over $\calo_L$, and write
\[
\omega:=\omega_{\ncal/\calo_L}.
\]
We compute the degree of $\omega^{\otimes (q^2-1)}$, rather than that of
$\omega$. 

Consider the rational section
\[
\beta:=\Delta\cdot(\mathrm dX)^{\otimes (q^2-1)}
\]
of $\omega^{\otimes (q^2-1)}$. It is independent of the choice of coordinate.
Indeed, if $X'=uX$, then
\[
g'=u^{1-q}g,\qquad
\Delta'=u^{1-q^2}\Delta,
\qquad
\mathrm dX'=u\,\mathrm dX,
\]
and hence
\[
\Delta'(\mathrm dX')^{\otimes (q^2-1)}
=
u^{1-q^2}\Delta\cdot u^{q^2-1}(\mathrm dX)^{\otimes (q^2-1)}
=
\Delta(\mathrm dX)^{\otimes (q^2-1)}.
\]

Let $w\nmid\infty$ be a finite place of $L$. Choose a minimal local
coordinate $X_w$ at $w$, and put
\[
\alpha_w:=\mathrm dX_w .
\]
Then $\alpha_w$ is a generator of $\omega$ over $\calo_{L,w}$, and by the
coordinate-invariance of $\beta$ we have
\[
\beta=\Delta_w\cdot \alpha_w^{\otimes (q^2-1)}.
\]
Therefore
\[
\frac{\omega^{\otimes (q^2-1)}\otimes_{\calo_L}\calo_{L,w}}
     {\calo_{L,w}\beta}
\simeq
\frac{\calo_{L,w}\alpha_w^{\otimes (q^2-1)}}
     {\calo_{L,w}\Delta_w\alpha_w^{\otimes (q^2-1)}}
\simeq
\frac{\calo_{L,w}}{\calo_{L,w}\Delta_w}.
\]
Thus the finite contribution of $w$ to
$\deg(\omega^{\otimes (q^2-1)})$ is
\[
\log_q\#\left(\calo_{L,w}/\calo_{L,w}\Delta_w\right).
\]
Summing over all finite places gives
\[
\sum_{w\nmid\infty}
\log_q\#\left(\calo_{L,w}/\calo_{L,w}\Delta_w\right)
=
\log_q N_{L/F}(\gotd).
\]

It remains to compute the infinite contribution. Fix
$\widetilde{\infty}\mid\infty$ and write
$z=z_{\widetilde{\infty}}$ and
$\Lambda=\Lambda_{\widetilde{\infty}}$. Via the Drinfeld exponential, the
differential $\mathrm dX$ corresponds to the differential on
$\C_\infty/\Lambda$ whose norm is
\[
\|\mathrm dX\|_{\widetilde{\infty}}=D_A(\Lambda).
\]
Write $\Lambda=a\Lambda_z$ with $a\in\C_\infty^\times$.
Then $D_A(\Lambda)=|a|_\infty D_A(\Lambda_z)$ and
$\Delta=a^{1-q^2}\Delta(z)$, so
\[
\|\beta\|_{\widetilde{\infty}}
=
|\Delta|_{\widetilde{\infty}}D_A(\Lambda)^{q^2-1}
=
|\Delta(z)|_{\widetilde{\infty}}D_A(\Lambda_z)^{q^2-1}.
\]
By the definition of the degree of a metrized line bundle,
the infinite contribution to $\deg(\omega^{\otimes (q^2-1)})$ is therefore
\[
-
\sum_{\widetilde{\infty}\mid\infty}
n_{\widetilde{\infty}}
\log_q
\left(
\left|\Delta(z_{\widetilde{\infty}})\right|_{\widetilde{\infty}}
D_A(\Lambda_{z_{\widetilde{\infty}}})^{(q^2-1)}
\right),
\]
which is exactly the infinite term in \eqref{e:silv}. Since
$\deg(\omega^{\otimes (q^2-1)})=(q^2-1)\deg(\omega)$, and $h_{\mathrm{Tag}}(\cald)=
\frac{1}{[L:F]}\deg(\omega)$,
formula \eqref{e:silv} follows.

We now prove \eqref{e:silv2}. For each
$\widetilde{\infty}\mid\infty$, choose
$z_{\widetilde{\infty}}$ in the fundamental domain $\calf$ for the
$\gl_2(A)$-action. On $\calf$, the usual $Q$-expansions at the cusp give
\[
-\log_q\left|\Delta(z)\right|_{\widetilde{\infty}}
=
\log_q^+\left|j(z)\right|_{\widetilde{\infty}}
+O_q(1).
\]
Let us show it. By \cite[Theorem 6.1]{gekelerj} we have the following expansion of $\Delta$. Consider the $a$-th inverse cyclotomic polynomial $f_a(X)=\ccal_a(X^{-1})X^{|a|_\infty}\in A[X]$, where we recall that $\ccal$ denotes the Carlitz $A$-module. Then $\Delta$ has an expansion in $Q$:
\[
\Delta(z)
=
-\widetilde{\pi}^{q^2-1}
Q^{q-1}
\prod_{\substack{a\in A\\ a\ \mathrm{monic}}}
f_a(Q)^{(q^2-1)(q-1)}
\]
and
\[
j(z)=-Q^{-(q-1)}+r(Q),
\]
while on the bounded part of $\calf$ both sides are $O_q(1)$.

At a finite place $w\nmid\infty$, stable reduction implies that the finite
Taguchi contribution agrees with the finite contribution of the graded
height. In a stable integral model, if the reduction has rank
$2$, then $\Delta_w$ is a unit and $j$ has no pole at $w$. If the reduction
has rank $1$, then the coefficient of $\uptau$ is a unit and $\ord_w(j)=-\ord_w(\Delta_w)$.
Thus
\[
\log_q N_{L/F}(\gotd)
=
\sum_{w\nmid\infty}
[\F_w:\F_q]\max\{-\ord_w(j(\cald)),0\}.
\]
Combining this finite equality with the preceding infinite estimate gives
\[
\log_q N_{L/F}(\gotd)
-
\sum_{\widetilde{\infty}\mid\infty}
n_{\widetilde{\infty}}
\log_q\left|\Delta(z_{\widetilde{\infty}})\right|_{\widetilde{\infty}}
=
[L:F]\,h\big(j(\cald)\big)
+
O_q([L:F]).
\]
Substituting this into \eqref{e:silv}, we obtain
\[
(q^2-1)[L:F]\,h_{\mathrm{Tag}}(\cald)
=
[L:F]\,h\big(j(\cald)\big)
-
(q^2-1)\sum_{\widetilde{\infty}\mid\infty}
n_{\widetilde{\infty}}
\log_qD_A(\Lambda_{z_{\widetilde{\infty}}})
+
O_q([L:F]).
\]

It remains to estimate the last sum. Let us first recall the normalization of
the lattice volume. In the notation of \cite[Definition~2.1]{bpr}, if
$\Lambda$ has successive minimum basis $(\omega_1,\dots,\omega_r)$, then
\[
D(\Lambda):=\prod_{i=1}^r|\omega_i|_\infty .
\]
Thus, for $z\in\calf$ and $\Lambda_z=A+zA$, \cite[Lemma~2.4(6)]{bpr}
gives
\[
D(\Lambda_z)=|z|_i.
\]
On the other hand, the normalization entering the infinite Taguchi metric is $A$-volume $D_A$ (see \cite[Section~4, before Definition~4.1]{wei}). Hence, in rank $2$, for the free reduced lattice $\Lambda_z=A+zA$, we have
\[
D_A(\Lambda_z)=D(\Lambda_z)^{1/2}=|z|_i^{1/2}.
\]

Recall also that $Q(z)=\exp_{\ccal}(\widetilde\pi z)^{-1}$. Near the cusp, we write $f(z)\sim g(z)$ to mean that $f(z)/g(z)$ stays between two positive constants depending only on $q$ as $|z|_i\to\infty$.
By \cite[Lemma~5.5]{gekelerj}, for $z\in\calf$ with $|z|_i>1$ one has
\[
\log_q |Q(z)|_{\widetilde{\infty}}^{-1}\sim |z|_i .
\]
Therefore
\[
\log_qD_A(\Lambda_z)
=
\frac12\log_q|z|_i
=
\frac12\log_q\log_q |Q(z)|_{\widetilde{\infty}}^{-1}
+O_q(1).
\]
Using the $Q$-expansion
$j(z)=-Q(z)^{-(q-1)}+r(Q(z))$,
we obtain%, uniformly in the cusp,
\[
\log_qD_A(\Lambda_z)
=\frac12
\log_q\!\left(1+\log_q^+|j(z)|_{\widetilde{\infty}}\right)
+O_q(1).
\]
On the bounded part of $\calf$ the same estimate is absorbed into $O_q(1)$.
Since $j(z)=-Q^{-(q-1)}+r(Q)$,
we get, uniformly for $z\in\calf$,
\[
0\le
\log_qD_A(\Lambda_z)
\le
\frac12
\log_q\!\left(1+\log_q^+|j(z)|_{\widetilde{\infty}}\right)
+O_q(1).
\]
Therefore, using the mean inequality,
\[
\begin{aligned}
0
&\le
\sum_{\widetilde{\infty}\mid\infty}
n_{\widetilde{\infty}}
\log_qD_A(\Lambda_{z_{\widetilde{\infty}}})  \\
&\le
\frac12
\sum_{\widetilde{\infty}\mid\infty}
n_{\widetilde{\infty}}
\log_q\!\left(
1+\log_q^+|j(\cald)|_{\widetilde{\infty}}
\right)
+
O_q([L:F])                                           \\
&\le
\frac{[L:F]}{2}
\log_q\!\left(
1+
\frac{1}{[L:F]}
\sum_{\widetilde{\infty}\mid\infty}
n_{\widetilde{\infty}}
\log_q^+|j(\cald)|_{\widetilde{\infty}}
\right)
+
O_q([L:F])   \\
&\le
\frac{[L:F]}{2}
\log_q\!\left(1+h\big(j(\cald)\big)\right)
+
O_q([L:F]).
\end{aligned}
\]
Dividing by $[L:F]$ we conclude.
\end{proof}

\section{The Pila--Zannier strategy}

\subsection{Functional transcendence}

\subsubsection{Ax--Lindemann--Weierstrass}

Recall that, over a field $L$ of characteristic $p$, a \emph{$p$-polynomial} in one variable is a polynomial
\[
f(X)=\sum_{j=0}^m a_j X^{p^j}\in L[X]
\]
which is characterized by the {polynomial} identity {$f(X+Y)=f(X)+f(Y)$ in $L[X,Y]$}.
For the $n$-variable case, a morphism {of group schemes} $\G_{a,L}^n\rightarrow\G_{a,L}$ is given by a polynomial
\[
F(X_1,\dots,X_n)=\sum_{i=1}^n f_i(X_i)
\]
where each $f_i$ is a $p$-polynomial.

   %Let us recall that  $F\in L[X_1,\dots,X_n]$ is a a {\em $p$-polynomial} if $F(X_1,\dots,X_n)=\sum f_i(X_i)$ where $f_i(X_i)=\sum a_{i,j}X_i^{p^j}\in L[X_i]$, for $a_{i,j}\in L$.
    
	\begin{thm}\label{alw}
		Let $V$ be a $\C_\infty$-algebraic subvariety of $\cald\times\cald'$ over $\C_\infty$, and let $W$ be a maximal irreducible $\C_\infty$-algebraic subvariety contained in $\exp^{-1}_{\Phi}(V)$. Assume $\rk\cald=\rk\cald'$. Then $\exp_{\Phi}(W)$ is the translate of a sub-$\Phi$-module of $\cald\times\cald'$.
	\end{thm}
The rest of this section is devoted to the proof of this theorem.
\\

Since $\cald\times\cald'\simeq\A^2$ as varieties, we have $\dim W\in\{0,1,2\}$.
If $\dim W=0$ the conclusion is immediate.
If $\dim W=2$, then $W=\A^2_{\C_\infty}$, hence $\exp_\Phi^{-1}(V)=\A^2_{\C_\infty}$ and thus
$V=\cald\times\cald'$; again the conclusion holds. We may therefore assume $\dim W=1$,
so $W$ is an irreducible algebraic curve. If $W$ is vertical or horizontal, i.e., one of the two coordinates $W\rightarrow \A^1$ is not dominant, then the claim of the theorem holds. We therefore assume that $W$ is neither vertical nor horizontal.

By maximality of $W$, we have that $V$ is a curve.
If necessary, replace $V$ by the irreducible component
containing $\exp_\Phi(W)$; such a component exists because
$W^{\mathrm{an}}$ is irreducible by
\cite[Theorem~2.3.1]{conrad}.
Both coordinate projections of $V$ are nonconstant.
Let $F(X,Y)$ and $G(X,Y)$ be irreducible defining
polynomials of $W$ and $V$, respectively.

Since $\C_\infty$ is perfect, the partial derivatives
$F_X$ and $F_Y$ cannot both vanish.
Interchanging the two coordinates, together with the
modules and lattices, if necessary, we may assume
$F_Y\ne0$. Thus the first coordinate projection of $W$
is generically separable.

The same holds for $V$. Indeed, choose a local analytic
graph $y=f(x)$ on $W$. If $G_Y=0$, differentiating
$G\bigl(\exp_\Lambda(x),\exp_{\Lambda'}(f(x))\bigr)=0$
and using $\exp_\Lambda'=\exp_{\Lambda'}'=1$ gives
\[
G_X\bigl(\exp_\Lambda(x),\exp_{\Lambda'}(f(x))\bigr)=0.
\]
This analytic image is infinite, since its first coordinate
has derivative $1$, and hence is Zariski dense in $V$.
Thus $G$ divides $G_X$, which forces $G_X=0$ by degree.
But $G_X=G_Y=0$ would make $G$ a $p$-th power over
$\C_\infty$, contradicting its irreducibility.
Therefore $G_Y\ne0$.

Set 
\[W_x=\{y: F(x,y)=0\}\;\;\;\text{and}\;\;\;V_u=\{v : G(u,v)=0\}.
\]

\begin{lem}\label{l:3.2}
    There is an additive subgroup $\Lambda_1\subset\Lambda$ of finite index $[\Lambda:\Lambda_1]=k_1$, such that 
    for any $\lambda\in\Lambda_1$ and $x\in\A^1_{\C_\infty}$ there is a $\mu=\mu(\lambda)\in\Lambda'$ such that 
    \[
    W_{x+\lambda}=W_x+\mu\;.
    \]
\end{lem}
\begin{proof}
\begin{comment}
Choose some $x\in\A^1_{\C_\infty}$  and let $u=\exp_\Lambda(x)$ such that $W_x$ and $V_u$ are locally given by smooth branches $y_{x,i}$ and $v_{u,i}$. For any $\lambda\in\Lambda$ define  $e(\lambda)=\exp_{\Lambda'}\left(W_{x+\lambda}\right)$. As $\exp_\Lambda$ is $\Lambda$-periodic, one has $e(\lambda)\subset V_u$. 
\end{comment}
{Choose $x\in\C_\infty$ and put $u=\exp_\Lambda(x)$,
so that $F(x+\lambda,Y)$ has degree $\deg_Y F$ and
only simple roots for every $\lambda\in\Lambda$,
and $G(u,Y)$ has degree $\deg_Y G$ and only simple roots.
Such a choice exists: separability excludes only finitely
many values for each projection, including the zeros of
the leading coefficients in $Y$.
Their translates by $\Lambda$ and their exponential
preimages exclude only countably many $x$, since $\Lambda$
is countable and each exponential fiber is a lattice coset.
As $\C_\infty$ is uncountable, the required local analytic branches exist at all these fibers.}
    
   Assume that $e(\lambda_1)\cap e(\lambda_2)\neq\emptyset$. We claim that $$F(X+\lambda,Y)=F(X, Y+\mu)$$ 
   for 
   $\lambda:=\lambda_2-\lambda_1$ and some $\mu\in\Lambda'$.
   %Indeed, the assumption means that  there are   some branches $y_{x+\lambda_1,i}$ and $y_{x+\lambda_2,j}$ and   some $\mu\in\Lambda'$ such that  $y_{x+\lambda_1,i}=y_{x+\lambda_2,j}+\mu$ at $x$ and therefore in some neighborhood of $x$, as the images of these branches under $\exp_{\Lambda'}$  coincide and $\exp_{\Lambda'}$ is locally a homeomorphism. This shows that the two polynomials $F(X+\lambda,Y)$ and $F(X, Y+\mu)$  share a common local branch $y_{x+\lambda_1,i}$. Since $F$ is irreducible, it follows that they coincide up to a multiplicative constant. This constant is necessarily equal to $1$: the coefficient of the leading term (in the sense of, say, lexicographic order) remains the same after such shifts.
   { Indeed, choose corresponding branches
$y_{x+\lambda_1,i}$ and $y_{x+\lambda_2,j}$, pulled back
by translation to a common variable near $x$, whose
exponential images agree at $x$.
These images lie on the same local graph of $V$ over
its first coordinate, so they agree near $x$.
A local analytic inverse of $\exp_{\Lambda'}$ then gives
\[
y_{x+\lambda_1,i}
=
y_{x+\lambda_2,j}+\mu
\]
near $x$, for some $\mu\in\Lambda'$.
Consequently the curves defined by
$F(X+\lambda_1,Y)$ and $F(X+\lambda_2,Y-\mu)$
share a local analytic graph.
Their irreducible defining polynomials are proportional.
Their leading coefficients in lexicographic order are
unchanged by translation, so
\[
F(X+\lambda_1,Y)=F(X+\lambda_2,Y-\mu).
\]
This is the claimed identity after translating $X$ and $Y$.}

   We say that two polynomials $F_1(X,Y)$ and $F_2(X,Y)$ are $y$-equivalent  if $$F_1(X,Y)=F_2(X,Y+\mu)$$ for some $\mu\in\Lambda'$.
The previous argument shows that, for any $\lambda,\lambda'\in \Lambda$, the subsets $e(\lambda)$ and $e(\lambda')$ either coincide or are disjoint. Moreover, in the first case $F(x+\lambda,y)$ and $F(x+\lambda',y)$ are $y$-equivalent. 
   %It follows that the family $\{F(X+\lambda,Y) : \lambda\in\Lambda\}$ meets only
   Therefore the set of functions $F(x+\lambda, y)$ for $\lambda\in\Lambda$, is a union of  $k_1$ $y$-equivalence classes, $k_1\le\deg_Y(G)$, since the sets $e(\lambda)\subset V_u$ are pairwise disjoint unless equal. The lattice $\Lambda$ acts on this finite set by translation on the $X$ variable.
   Denoting by $\Lambda_1:=\Stab([F])$ the stabilizer of the class of $F$, the result follows. 
\end{proof} 

Consider two algebraic curves $\Gamma_{1}$ and $\Gamma_2$ in the (affine) space of bi-variate polynomials of bi-degree $\le (\deg_XF,\deg_Y  F)$:
\[
\Gamma_1=\{F(X+t,Y)\},\quad \Gamma_2=\{F(X, Y+s)\} 
\]
for $s,t\in\C_\infty$.
By Lemma~\ref{l:3.2}, $\Gamma_1\cap\Gamma_2$ is infinite, as $F(X+\lambda,Y)=F(X,Y+\mu)$ for all $\lambda\in\Lambda_1$ and some suitable $\mu$ depending on $\lambda$. Thus $\Gamma_1=\Gamma_2$.

Define 
$$L=\{(x,y)\in \A^2_{\C_\infty} :  F(X+x,Y+y)=F(X,Y)\}.$$ 
This is an algebraic set and an $\F_p$-vector space. The above means that the projection of $L$ onto the first coordinate is surjective, so $\dim L\ge 1$. Clearly, 
$W=W+L$, so, by irreducibility of $W$, $W$ is a coset of $L$ and $\dim L=1$.

Thus, up to a shift such that $0\in W$, we have:
\begin{enumerate}
    \item $W$ is an algebraic set, and an $\F_p$-vector space,
    \item There are two sublattices $\Lambda_1\subset\Lambda$, $\Lambda_1'\subset\Lambda'$ of finite index 
    such that $W_{\lambda}\cap \Lambda'\neq\emptyset$ for all $\lambda\in\Lambda_1$ and similarly for the second coordinate.
\end{enumerate}
The first property implies that  $W$ is the zero locus of an irreducible $p$-polynomial $P(X)-Q(Y)$, 
\begin{equation*}
    W=\big\{(x,y)\in \A^2_{\C_\infty}\;:\; P(x)=Q(y)\big\}
\end{equation*} where, by definition, $P$ and $Q$ are $p$-polynomials.
{ Indeed, after translation, continue to denote by $F$ an
irreducible polynomial defining $W$, so that $F(0,0)=0$.
For $(x,y)\in W$, translation invariance gives
$F(X+x,Y+y)=F(X,Y)$.
The coefficients in $X,Y$ of
$F(X+x,Y+y)-F(X,Y)-F(x,y)$ vanish on $W$ and have total degree in $x,y$ strictly less
than $\deg F$. Each is therefore divisible by $F(x,y)$
and must be zero. Hence $F$ is additive, and
\[
F(X,Y)=F(X,0)+F(0,Y)=P(X)-Q(Y),
\]
where $P(X)=F(X,0)$ and $Q(Y)=-F(0,Y)$ are $p$-polynomials.

We now obtain the period correspondence in both coordinates,
without any further separability assumption.
Consider the additive groups $H:=W\cap(\C_\infty\times\Lambda')$ and $K:=W\cap(\Lambda\times\Lambda')$.
Since $P$ is nonconstant and $\C_\infty$ is algebraically
closed, the second projection $H\rightarrow\Lambda'$
is surjective.
The homomorphism $\exp_\Phi|_H$ has kernel $K$ and image
contained in the finite set
$V\cap(\A^1_{\C_\infty}\times\{0\})$.
Hence $H/K$ is finite, and so is its quotient
$\Lambda'/\pr_2(K)$.
Interchanging coordinates in this group argument shows
that $\pr_1(K)$ has finite index in $\Lambda$.

Thus we may take
$\Lambda_1=\pr_1(K)$ and $\Lambda'_1=\pr_2(K)$
which are finite-index additive subgroups.
For every $(\lambda,\mu)\in K$, additivity gives
$W+(\lambda,\mu)=W$, and therefore
\[
W_{x+\lambda}=W_x+\mu
\]
for every $x$, with the analogous equality in the
second coordinate.
}

\begin{lem}\label{l:3.3}
    Let the ranks of $\Lambda, \Lambda'$ be $r, r'$ respectively.  Then 
    $$\dfrac r {r'}=\frac{\deg P}{\deg Q}.$$
\end{lem}
\begin{proof}

    Denote  $\alpha=\frac{\deg Q}{\deg P}$. 
If $P(x)-Q(y)=0$ then
\[
\deg_X(P)\,\log_q|x|_\infty = \deg_Y(Q)\,\log_q|y|_\infty + O(1),
\]
hence
\begin{equation}\label{e:size}
\log_q|x|_\infty = \alpha\,\log_q|y|_\infty + O(1).
\end{equation}
Denote $b(d)=\#\{x\in\Lambda, \log_q|x|_\infty\le d\}$ and similarly $b'(d)=\#\{y\in\Lambda', \log_q|y|_\infty\le d\}$. Also, denote by 
$$b_1(d):=\#\{x\in\Lambda_1, \log_q|x|_\infty\le d\}$$
and
$$b_1^*(d):=\#\{y\in\Lambda'|\,    (\log_q|y|_\infty\le d)\wedge(\exists x\in\Lambda_1\,:\,  P(x)-Q(y)=0)\}.
$$ 
By the definition of $\Lambda_1$,  for every $x\in\Lambda_1$ there exists a $y\in\Lambda'$ such that $P(x)-Q(y)=0$.
Therefore
\begin{equation*}
 q^{r'd}O(1)= b'(d)\ge b_1^*(d)\ge \frac{1}{\deg P} b_1(\alpha d+O(1))
 \ge \frac{1}{\deg_X(P)[\Lambda:\Lambda_1]} b(\alpha d+O(1))=q^{r\alpha d}O(1).
\end{equation*}
Similarly, one obtains $q^{r\alpha d}\ge q^{r'd}O(1)$. Taking $d\to\infty$, we conclude $r'=\alpha r$.
\end{proof}

From now on, we assume that $\rk \Lambda=\rk\Lambda'=r$ and denote $n=\deg P=\deg Q=p^m$.

\begin{prop}\label{prop:deg=1}
    If $\rk\Lambda=\rk\Lambda'=r$ then $\deg P=\deg Q=1$ and  $W$ is a line.   
\end{prop}
{
\begin{proof}
We take $P(X)-Q(Y)$ to be an irreducible defining equation of $W$.
By Lemma~3.3, we have that $\deg P=\deg Q=n$.
Note that the additive subgroup $\Lambda_1\cap T^{-1}\Lambda_1$ has finite
index in $\Lambda$: it is the kernel of the homomorphism
\[
\Lambda\rightarrow(\Lambda/\Lambda_1)^2,\qquad
x\mapsto(x+\Lambda_1,Tx+\Lambda_1).
\]
For each $x$ in this subgroup, choose $y,y'\in\Lambda'$ such that
$(x,y)\in W$ and $(Tx,y')\in W$. By additivity,
\[
\begin{aligned}
Q(y'-Ty)
&=P(Tx)-Q(Ty)=\bigl(P(Tx)-T^nP(x)\bigr)
  -\bigl(Q(Ty)-T^nQ(y)\bigr).
\end{aligned}
\]
Both bracketed polynomials have degree at most $n-1$.
Moreover, since $\deg P=\deg Q=n$, the equality $P(x)=Q(y)$
gives $|y|_\infty\ll\max\{1,|x|_\infty\}$.
Thus for $d\ge0$ and $|x|_\infty\le q^d$,
\[
|Q(y'-Ty)|_\infty\ll q^{(n-1)d}
\qquad
\text{and}
\qquad
|y'-Ty|_\infty\ll q^{(n-1)d/n}.
\]
The second estimate follows because the leading term of $Q$ dominates outside a fixed ball. 
% All implicit constants are independent of $d$ and of the chosen lifts.

The lattice count used in Lemma~\ref{l:3.3} gives
\[
\#\{x\in\Lambda_1\cap T^{-1}\Lambda_1:
          |x|_\infty\le q^d\}
=q^{rd+O(1)}.
\]
Indeed, intersecting a lattice ball with a finite index additive subgroup decreases its cardinality by at most that index.
On the other hand, since $y'-Ty\in\Lambda'$, the number of
possible values of $y'-Ty$ is at most
\[
q^{r(n-1)d/n+O(1)}.
\]
Thus some value occurs for at least $q^{rd/n-O(1)}$ of the
chosen triples $(x,y,y')$.

Fix one triple $(x_0,y_0,y'_0)$ with this value.
For every other triple with the same value,
\[
y'-y'_0=T(y-y_0).
\]
Subtracting the corresponding points of $W$ and using its
additivity, we obtain
\[
(x-x_0,y-y_0)\in W
\qquad
\text{and}
\qquad
T(x-x_0,y-y_0)
=\bigl(T(x-x_0),y'-y'_0\bigr)\in W.
\]
Distinct values of $x$ give distinct points of $W\cap T^{-1}W$.
As $q^{rd/n-O(1)}$ tends to infinity, this intersection is
infinite. Since $W$ and $T^{-1}W$ are irreducible algebraic
curves, they coincide; hence $TW=W$.

Choose any nonzero $(x,y)\in W$. The points
$(T^kx,T^ky)$ for  $k\ge0$ are distinct, since $|T|_\infty=q>1$, and belong to both $W$
and the line given by $yX-xY=0$. Thus $W$ meets this line in infinitely many points.
Since $W$ is an irreducible algebraic curve, it follows that $W$ is this line.
Since $P(X)-Q(Y)$ is irreducible, it follows that $\deg P=\deg Q=1$.
\end{proof}
}

\begin{proof}[Proof of Theorem~\ref{alw}]
Write $W=(x_0,y_0)+W_0$ with $W_0$ a one dimensional subspace. {By the two-sided period correspondence established
before Lemma~\ref{l:3.3}, the
$A$-lattice $W_0\cap(\Lambda\times\Lambda')$ has projections
of finite index in $\Lambda$ and $\Lambda'$.
By lattice uniformisation, these projections induce 
isogenies from the associated Drinfeld module to $\cald$
and $\cald'$. The resulting map into $\cald\times\cald'$
has image  $\exp_\Phi(W_0)$, by surjectivity of the exponential, and is finite since its first coordinate
is an isogeny. Hence $\exp_\Phi(W_0)$ is a closed connected
algebraic subgroup, and \eqref{drinfeldexp} gives its
$A$-invariance. Thus,}
$\exp_\Phi(W_0)$ is a sub-$\Phi$-module of $\cald\times\cald'$ and
\[
\exp_\Phi(W)=\exp_\Phi(x_0,y_0)+\exp_\Phi(W_0)
\]
is its translate.
\end{proof}

On the other hand, for different ranks $r\neq r'$ we have the following counterexample.

 \begin{eg}\label{eg:different-ranks}
  Let $\Lambda=\F_q[T^{1/q}]$ and $\Lambda'=A$ be of rank $q$ and $1$, respectively. The map  $X\mapsto X^q$ is a bijection between $\Lambda$ and $\Lambda'$. Thus, the algebraic curve $W=\{(x,y) : y=x^q\}$ is an $\F_q$-subspace intersecting $\Lambda\times\Lambda'$ {  in} a set projecting onto both $\Lambda$ and $\Lambda'$.

Moreover, let $y=x^q$, $x\in\C_\infty$. Then
  \begin{equation*}
   u^q=    \exp_{\Lambda}(x)^q=\left(x\prod_{0\neq\lambda\in\Lambda}\left(1-\frac{x}{\lambda}\right)\right)^q=x^q\prod_{0\neq\lambda\in\Lambda}\left(1-\frac{x^q}{\lambda^q}\right)=y\prod_{0\neq\lambda'\in\Lambda'}\left(1-\frac{y}{\lambda'}\right)=\exp_{\Lambda'}(y)=v.
  \end{equation*}
  In other words, $\exp_{\Phi}W$ is contained in an algebraic curve  $V$ defined by $v=u^q$, which  is a connected algebraic subgroup of $\G_{a,\C_\infty}^2$. However, $V$ is not a sub-$\Phi$-module of $\cald\times\cald'$, otherwise it would be the graph of a morphism of Drinfeld $A$-modules of different ranks, while there is no nonzero morphism between Drinfeld $A$-modules of different rank.
\end{eg}

\subsubsection{Hyperbolic Ax--Lindemann}

We use the definitions of weakly special and special
subvarieties given in case~{\rm(H)} of the introduction.

{Here an algebraic subvariety of $\Omega^2$ means an intersection
$\overline Z^{\mathrm{an}}\cap\Omega^2$, with its reduced
structure, where $\overline Z$ is a closed algebraic subvariety
of $\p^1_{\C_\infty}\times\p^1_{\C_\infty}$.
For an irreducible algebraic curve $\overline Z$, this
intersection is irreducible whenever it is nonempty.
Indeed, on the normalization of $\overline Z$, each nonconstant
coordinate projection is finite, while any constant coordinate
belongs to $\Omega$. Thus restricting to $\Omega^2$ removes
only classical points.
The inverse image of the intersection is therefore connected
and smooth, hence irreducible. Its finite surjective image
$\overline Z^{\mathrm{an}}\cap\Omega^2$ is therefore irreducible.}

\begin{thm}\label{p:hyperbolicAL}
    Let $C$ be a $\C_\infty$-algebraic subvariety of $\A^2_{\C_\infty}$ and let $Z$ be a maximal irreducible $\C_\infty$-algebraic subvariety contained in $\boldsymbol{j}^{-1}(C)$. Then $\boldsymbol{j}(Z)$ is a weakly special variety.
\end{thm}

The rest of the section is devoted to the proof of this theorem, modeled on the proof of Theorem~\ref{alw}.
\\

If $\dim Z=0$, its image is a point and the conclusion
is immediate. If $C=\A^2_{\C_\infty}$, maximality gives
$Z=\Omega^2$, whose image is $\A^2_{\C_\infty}$.
We may therefore assume that $C$ and $Z$ are curves.

Let $C$ be defined as the zero locus of an irreducible polynomial $G(X,Y)\in\C_\infty[X,Y]$. Let also $H(x,y)$ be the polynomial whose zero locus defines  the curve $Z\subset \boldsymbol{j}^{-1}(C)$.
    
If one of the coordinate projections $Z\rightarrow \A^1_{\C_\infty}$ is constant, say $x\equiv x_0$ on $Z$, then
$G(j(x_0),j(y))=0$ for all $(x_0,y)\in Z$, so $G\big(j(x_0),Y\big)$ is identically zero as a polynomial in $Y$.
Thus $C$ is the vertical line $\{j(x_0)\}\times \A^1_{\C_\infty}$ and therefore $\boldsymbol{j}(Z)=\{j(x_0)\}\times \A^1_{\C_\infty}$ is weakly special.
The horizontal case $\A^1_{\C_\infty}\times\{y\}$ is analogous. Henceforth we shall assume that both projections are dominant.
For $x\in\Omega$ and $u\in \A^1_{\C_\infty}$ we set
$$Z_x:=\{y\in\Omega: H(x,y)=0\}\;\;\;\text{and}\;\;\;C_u:=\{v\in \A^1_{\C_\infty}: G(u,v)=0\}.$$

\begin{lem}\label{l:jbranch}
There exists a  subgroup $\Gamma_1\le \pgl_2(A)$ with finite index $[\pgl_2(A):\Gamma_1]=k_1$  and a map $\mu\colon \Gamma_1\rightarrow \pgl_2(A)$ such that for every $\gamma\in\Gamma_1$ and every $x\in\Omega$ one has
\[
Z_{\gamma x}=\mu(\gamma)\,Z_x\;.
\]
\end{lem}

\begin{proof}
%Choose some $x\in\Omega$ such that $j(x)$ is not a critical value of $j$, and set $u:=j(x)\in \A^1_{\C_\infty}$,
%so that $Z_x$ is locally given by smooth branches $y_{x,i}$ of $Z$ above $x$.
{
We exclude the finitely many values $u$ for which $u$ is an elliptic $j$-value, or some $(u,v)\in C$ is singular or has elliptic $j$-value $v$.
Let $K_0\subset\C_\infty$ be the field generated over $F_\infty$ by the coefficients of $H$ and by one preimage of $j$ of each excluded value.
Since $K_0/F_\infty$ is finitely generated,
$K_0\cap\overline{\F}_q$ is finite by
\cite[Lemma~9.26.11]{stacks}.
We may therefore choose $x\in\overline{\F}_q$ such that
$[K_0(x):K_0]>\deg_X(H)$.
In particular,we have $x\in\Omega$.
Every preimage of $j$ of an excluded value belongs to $K_0$, as each $j$-fiber is a $\pgl_2(A)$-orbit, so that $u:=j(x)$ is not excluded.

For every $\gamma\in\pgl_2(A)$, we have
$K_0(\gamma x)=K_0(x)$.
Therefore $H(\gamma x,Y)$ has degree $\deg_Y(H)>0$.
It has no root $y\in F_\infty$: such a root would make
$\gamma x$ a root of the nonzero polynomial
$H(X,y)\in K_0[X]$, of degree at most $\deg_X(H)$.
Here $H(X,y)$ is nonzero because $H$ is irreducible and
$Z$ is not horizontal.
Therefore $Z_{\gamma x}$ is nonempty for every $\gamma$.
}

Let $d_X:=\deg_X(H)$ and $d_Y:=\deg_Y(H)$. For $\gamma=\begin{bmatrix} 
    a & b
    \\
    c & d
\end{bmatrix}\in\pgl_2(A)$, we define a ``slash operator'' as follows
\[
(H|_X\gamma):=(cX+d)^{d_X}H(\gamma X,Y),\;\;(H|_Y\gamma):=(cY+d)^{d_Y}H(X,\gamma Y)\,\in\C_\infty[X,Y].
\]

Note that for $X\in\Omega$ one has $cX+d\neq 0$ (since $d/c\in F_\infty$ when $c\neq 0$), %and for $c=0$then $d\neq 0$), 
and similarly  for $cY+d\neq 0$; hence these factors do not change
zero sets on $\Omega$. Moreover, an irreducible polynomial remains irreducible under this action.  

For any $\gamma\in\pgl_2(A)$ define the finite subset 
\[
e(\gamma):=j\big(Z_{\gamma x}\big)\ \subset\ \A^1_{\C_\infty}.
\]
Since $j(\gamma x)=j(x)=u$, one has $e(\gamma)\subset C_u$ for all $\gamma$. In particular, the set
$C_u$ is finite of cardinality $\le \deg_Y(G)$, hence there are only finitely many possibilities for $e(\gamma)$.

Assume that $e(\gamma_1)\cap e(\gamma_2)\neq\emptyset$.
This means there exist $y_1\in Z_{\gamma_1 x}$ and
$y_2\in Z_{\gamma_2 x}$ such that $j(y_1)=j(y_2)$. By definition of the quotient map $j$, this implies
$y_1=\delta\,y_2$ for some $\delta\in\pgl_2(A)$. 
\begin{comment}
Working on sufficiently small neighborhoods where $j$
is a local homeomorphism and the branches are single-valued, we obtain an equality of  branches
near $x$ of the form
\[
y_{\gamma_1 x,i}=\delta\,y_{\gamma_2 x,j}.
\]
Equivalently, the two polynomials $(H|_X\gamma_1)$ and $((H|_X\gamma_2)|_Y\delta^{-1})$ share a common
analytic branch. Since $H$ is irreducible, it follows that they coincide up to a constant $c\in\C_\infty^\times$.
\end{comment}
{By the choice of $u$, the point $(u,j(y_1))$ is smooth
on $C$, and neither coordinate is an elliptic $j$-value.
Hence $\boldsymbol j$ is a local analytic isomorphism at
$(x,y_1)$, and $\boldsymbol j^{-1}(C)$ has a unique
smooth curve germ there.
The curves defined by
\[
(H|_X\gamma_1)
\qquad\text{and}\qquad
((H|_X\gamma_2)|_Y\delta^{-1})
\]
both pass through $(x,y_1)$ and, by the invariance of $j$,
their germs are contained in this same smooth curve germ.
Since both have dimension one, their reduced germs coincide\footnote{Equality of germs means that the reduced analytic
curves agree on a sufficiently small admissible neighborhood
of $(x,y_1)$.}.
Thus the two irreducible polynomials have a common
one dimensional analytic germ and are proportional.}

We say that two polynomials $F_1(X,Y)$ and $F_2(X,Y)$ are $y$-equivalent if and only if
\[
F_1(X,Y)=c \cdot (F_2|_Y\eta)
\]
for some $\eta\in\pgl_2(A)$ and $c\in\C_\infty^\times$.
%With this terminology, the previous argument shows that for any $\gamma,\gamma'\in\pgl_2(A)$, the subsets
%$e(\gamma)$ and $e(\gamma')$ either coincide or are disjoint; moreover, in the first case
%$(F|_X\gamma)$ and $(F|_X\gamma')$ are $y$-equivalent.
With this terminology, the previous paragraph shows that whenever
$e(\gamma_1)\cap e(\gamma_2)\neq\emptyset$, then $(H|_X\gamma_1)$ and $(H|_X\gamma_2)$ are $y$-equivalent.
In particular $e(\gamma_1)=e(\gamma_2)$, since $j(\eta y)=j(y)$ for $\eta\in\pgl_2(A)$.
Thus the subsets $e(\gamma)\subset C_u$ are pairwise disjoint unless equal.
\\
Therefore the family $\{(H|_X\gamma):\gamma\in\pgl_2(A)\}$ meets only finitely many $y$-equivalence
classes, in fact at most $\deg_Y(G)$. 
%\#C_u\le \deg_Y(G)$. 
The group $\pgl_2(A)$ acts on the right on this finite set by the slash operators
in the $X$ variable. Denoting by $\Gamma_1:=\Stab([H])$ the stabilizer of the $y$-equivalence class of $H$, we have $[\pgl_2(A):\Gamma_1]\le \deg_Y(G)$.
Finally, for each $\gamma\in\Gamma_1$ choose $\nu(\gamma)\in\pgl_2(A)$ and $c(\gamma)\in\C_\infty^\times$
such that
\[
(H|_X\gamma)=c(\gamma)\cdot (H|_Y\nu(\gamma)).
\]
Evaluating at an arbitrary {  $x\in\Omega$} and taking zero sets in the variable $Y$, the scalar $c(\gamma)$
and the nonvanishing factors $(c x+d)^{d_X}$ and $(cY+d)^{d_Y}$ are irrelevant; %\Mitya{Should $z$ be $x$? yes, corrected} 
we get
\[
Z_{\gamma x}=\nu(\gamma)^{-1}Z_x.
\]
By setting $\mu(\gamma):=\nu(\gamma)^{-1}$ we conclude.
\end{proof}
Let $V$ be the finite dimensional
$\C_\infty$-vector space of bi-variate polynomials of bi-degree $\le(d_X,d_Y)$. Since we only care about zero sets, we work in the projective space
$\p(V):=\text{Proj}(\text{Sym}\big(V^\vee)\big)$, i.e., we identify polynomials up to nonzero scalars, and we denote by $[H]$ the projective class corresponding to a polynomial $H$. Consider the two algebraic subsets
\[
\bigl\{\,[(H|_X \gamma)] : \gamma\in\pgl_2(\C_\infty)\,\bigr\}
,\;
\bigl\{\,[(H|_Y \eta)] : \eta\in\pgl_2(\C_\infty)\,\bigr\}
\subset \p(V)
\]
and  denote by $\Gamma_X$ and $\Gamma_Y$ their Zariski closures.
By Lemma~\ref{l:jbranch}, for every $\gamma\in\Gamma_1$ there exist $\mu(\gamma)\in\pgl_2(A)$ and
$c(\gamma)\in\C_\infty^\times$ such that $(H|_X\gamma)=c(\gamma)\,(H|_Y\mu(\gamma)^{-1})$. Equivalently, $[(H|_X\gamma)]\in\Gamma_Y$, hence
\[
\{\,[(H|_X\gamma)] : \gamma\in\Gamma_1\,\}\subset \Gamma_Y.
\]

We recall the following standard result.

\begin{lem}\label{l:dense}
Let $\Gamma\le \pgl_2(A)$ be of finite index. Then $\Gamma$ is Zariski dense in $\pgl_{2/\C_\infty}$. 
\end{lem}
\begin{proof}
As the index $[\pgl_2(A):\Gamma]=n$ is finite, we choose representatives $\gamma_1,\dots,\gamma_n\in\pgl_2(A)$ such that $\pgl_2(A)=\bigsqcup_i^n\gamma_i \Gamma$. Taking the Zariski closure, as $\pgl_2(A)$ is Zariski dense in $\pgl_{2,\C_\infty}$, we have $$\pgl_{2,\C_\infty}=\overline{\pgl_2(A)}\subseteq \bigcup_{i=1}^n\gamma_i\overline{\Gamma}.$$
Since $\pgl_{2,\C_\infty}$ is irreducible, we conclude.
%every connected algebraic group over an algebr closed field is irreducible
\end{proof}

Lemma~\ref{l:dense} implies that the subset $\{[(H|_X\gamma)] : \gamma\in\Gamma_1\}$ is Zariski dense in
$\Gamma_X$, i.e., it is dense in the image of the orbit map. It follows that $\Gamma_X\subseteq \Gamma_Y$. 
To get the reverse inclusion, one applies the same argument with the two coordinates interchanged. Hence we conclude that $\Gamma_X=\Gamma_Y$.

For a projective class $[H]\in\p(V)$, define
\[
S:=\Bigl\{(\gamma,\eta)\in\pgl_2(\C_\infty)^2:\ \exists\,c(\gamma,\eta)\in\C_\infty^\times\ \text{such that}\ 
\bigl((H|_X \gamma)|_Y \eta\bigr)=c(\gamma,\eta)\,H\Bigr\}.
\]

Let $\overline Z\subset
\p^1_{\C_\infty}\times\p^1_{\C_\infty}$
be the Zariski closure of $Z$.

\begin{lem}\label{l:SZ=Z}
The set $S$ is a Zariski closed subgroup of
$\pgl_2(\C_\infty)^2$, and
$S\cdot\overline Z=\overline Z$.
\end{lem}
\begin{proof}
Consider the  morphism $\alpha\colon\pgl_{2,\C_\infty}^2 \rightarrow \p(V)$ sending $(g,h)\mapsto [\,H|_X g|_Y h\,]$.
Since the existence of $\lambda\in\C_\infty^\times$ such that $G=\lambda H$  is equivalent to $[G]=[H]$ in $\p(V)$, we have $S=\alpha^{-1}([H])$ is Zariski closed.% in $\pgl_{2,\C_\infty}^2$.

{Bihomogenizing $H$ to bidegree $(d_X,d_Y)$ identifies
$[H]$ with a projective equation for $\overline Z$.
The slash operators are the pullbacks under the
coordinatewise action on
$\p^1_{\C_\infty}\times\p^1_{\C_\infty}$.
Thus the defining equality for $S$ is equivalent to
preserving $\overline Z$.
In particular, $S$ is a subgroup and
$S\cdot\overline Z=\overline Z$.}
 \end{proof}
%On the other hand, if $\bigl((H|_X g)|_Y h\bigr)=c\,H$ with $c\in\C_\infty^\times$, then on $\Omega^2$ the clearing denominators
%factors of the slash operator do not vanish. Hence $H(gx,hy)=0$ holds if and only if $\bigl((H|_X g)|_Y h\bigr)(x,y)=0$,
%and this is equivalent to $H(x,y)=0$ as $c$ is invertible. Thus $(\gamma,\eta).Z=Z$.

%By Lemma \ref{l:jbranch}, we see 
%$(F|_X\gamma)=c(\gamma)(F|_Y\mu(\gamma))$ apply $|_Y\mu(\gamma)^{-1}$ to both sides
%that  $(\gamma,\mu(\gamma))\in S$
%for every $\gamma\in\Gamma_1$. This means that the projection onto the first factor $\pr_1\colon S\to \pgl_{2}(\C_\infty)$  contains $\Gamma_1$. This shows that $\dim S=3$.
%As both projections restricted to $S$ are dominant, by Goursat's lemma we have that
{
By Lemma~\ref{l:jbranch}, applied in both coordinates,
both projections of $S$ contain finite-index subgroups
of $\pgl_2(A)$, and hence are dominant by
Lemma~\ref{l:dense}.
Their images are constructible subgroups by Chevalley's
theorem, so both projections are surjective on
$\C_\infty$-points.

Their kernels are finite on $\C_\infty$-points.
Indeed, an element $(1,\eta)$ in the first kernel
permutes every finite fiber of
$\overline Z\rightarrow\p^1_{\C_\infty}$.
Choose finitely many such fibers whose second coordinates
include at least three distinct points.
The action on their finite union is faithful, since
a M\"obius transformation fixing three distinct points
is the identity. Thus the first kernel is finite,
and the second is finite by the same argument.
By surjectivity, these kernels identify with finite
normal subgroups of the opposite factors
$\pgl_2(\C_\infty)$, which are simple infinite groups.
Hence both kernels are trivial.
In particular, $\dim S=3$.
Goursat's lemma  now gives}

\begin{equation}
    S=\big\{\big(\gamma,\psi(\gamma)\big):\gamma\in \pgl_2(\C_\infty)\big\}
\end{equation}
for some $\psi\in \Aut(\pgl_2(\C_\infty))$.

\begin{lem}\label{l:aut}
  The automorphism group of $\pgl_{2}(\C_\infty)$ is naturally isomorphic to $\pgl_2(\C_\infty) \rtimes \Aut(\C_\infty)$.
\end{lem}
\begin{proof}
Since $\C_\infty$ is algebraically closed,
$\PSL_2(\C_\infty)=\pgl_2(\C_\infty)$.
The assertion follows from \cite[Satz~1, case $n=1$]{svdw}.
\end{proof}
Let $\sigma$ be the field component of $\psi$ under Lemma~\ref{l:aut}.
Since $S$ is algebraic, $\sigma$ is definable, i.e., its graph is definable using field operations. Therefore, as mentioned {\em en passant} in \cite[p.644]{moosa}, since $\C_\infty$ is algebraically closed, any field automorphism arising from $S$ is an integer power of the Frobenius $x\mapsto x^p$, which we still denote in this section, by slight abuse of notation, by $\uptau$.
After exchanging the two coordinates if necessary, we may assume that the Frobenius exponent is non-negative.
Therefore, by Lemma \ref{l:aut} 
we have that 
\begin{equation}\label{e:S}
    S=\big\{\big(\gamma,\text{Inn}(\delta)\circ\uptau^m(\gamma)\big):\gamma\in \pgl_2(\C_\infty)\big\}
\end{equation}
for $m$ a non-negative integer, $\delta\in\pgl_2(\C_\infty)$ and $\text{Inn}(\delta)$ the inner automorphism induced by $\delta$.

Let us write
\[
\pgl_2(A)^{(m)}:=\uptau^m\big(\pgl_2(A)\big)=\pgl_2(\F_q[T^{p^m}])
\]
for $m\ge 0$.

\begin{lem}\label{l:commensurator-frobenius}
Let $m\ge 0$, let $\delta\in\pgl_2(\C_\infty)$, and suppose that
\[
K:=\pgl_2(A)\cap \delta\,\pgl_2(A)^{(m)}\,\delta^{-1}
\]
has finite index in $\pgl_2(A)$. Then $\delta\in\pgl_2(F)$.
\end{lem}

\begin{proof}
We first recall that every point of $\p^1(F)$ is fixed by a nontrivial unipotent subgroup of $\pgl_2(A)$. Indeed, for $x=\infty$ this is clear from the subgroup
$U(A):=\left\{{\small{\begin{bmatrix}1&f\\0&1\end{bmatrix}}}:f\in A\right\}
\subseteq\pgl_2(A)$.
If $x=a/b\in F$, with $a,b\in A$ coprime, choose $u,v\in A$ such that $av-bu=1$, and set
$g=\small\begin{bmatrix}a&u\\ b&v\end{bmatrix}\in \gl_2(A)$.
Then $g\in\pgl_2(A)$, $g\infty=x$, and the conjugate $gU(A)g^{-1}$ is an infinite unipotent subgroup of $\pgl_2(A)$ fixing $x$.

Now let $x\in\p^1(F)$. Since $K$ has finite index in $\pgl_2(A)$, the intersection $gU(A)g^{-1}\cap K$ has finite index in $gU(A)g^{-1}$, and is therefore nontrivial. Fix a nontrivial unipotent element $u_x\in K$ fixing $x$.

Because $u_x\in K\subseteq \delta\pgl_2(A)^{(m)}\delta^{-1}$, there exists $\gamma_x\in\pgl_2(A)^{(m)}$ such that
$u_x=\delta\gamma_x\delta^{-1}$. Hence $\gamma_x=\delta^{-1}u_x\delta$ is a nontrivial unipotent element of $\pgl_2(A)^{(m)}$ which fixes $\delta^{-1}x$. 
%But a nontrivial unipotent element
%of $\pgl_2(F)$ has a unique fixed point, which is in $\p^1(F)$. 
{Since $\gamma_x\in\pgl_2(A)^{(m)}$, it has a representative
in $\gl_2(A)$ with determinant in $\F_q^\times$.
Its repeated eigenvalue belongs to $F$: in odd characteristic
it is half the trace, and in characteristic $2$ it is the
square root of the determinant, which lies in $\F_q$.
Thus its unique eigenline is defined over $F$, so its unique
fixed point lies in $\p^1(F)$.}
Therefore
\[
\delta^{-1}x\in \p^1(F).
\]

Applying this to $x=0,1,\infty$, we find that
\[
\delta^{-1}(0),\quad \delta^{-1}(1),\quad \delta^{-1}(\infty)
\]
all lie in $\p^1(F)$. A M\"obius transformation is uniquely determined by the
images of $0,1,\infty$, and therefore $\delta^{-1}$, hence also $\delta$,
is defined over $F$. Thus $\delta\in\pgl_2(F)$.
\end{proof}

\begin{prop}\label{p:nofrob}
There is no Frobenius twist in $S$, namely,
\[
S=\left\{\bigl(\gamma,\delta\gamma\delta^{-1}\bigr):\gamma\in\pgl_2(\C_\infty)\right\}
\]
for some $\delta\in\pgl_2(F)$.
\end{prop}

\begin{proof}
Applying Lemma~\ref{l:jbranch} to $Z$, and then applying the same argument with the two coordinates interchanged, we obtain finite-index subgroups
$\Gamma_1,\Gamma_2\le\pgl_2(A)$ and maps
\[
\mu\colon\Gamma_1\rightarrow\pgl_2(A)
\qquad 
\text{and}
\qquad
\nu\colon\Gamma_2\rightarrow \pgl_2(A)
\]
such that
\[
(\gamma,\mu(\gamma))\in S
\quad\text{for }\gamma\in\Gamma_1
\qquad
\text{and}
\qquad
(\nu(\eta),\eta)\in S
\quad\text{for }\eta\in\Gamma_2.
\]
By \eqref{e:S}
\[
S=\{(\gamma,\psi(\gamma)):\gamma\in\pgl_2(\C_\infty)\}
\]
for $\psi=\operatorname{Inn}(\delta)\circ\uptau^m$ and for some $m\ge 0$ and some $\delta\in\pgl_2(\C_\infty)$.
From $(\nu(\eta),\eta)\in S$, with $\nu(\eta)\in\pgl_2(A)$, we get
\[
\eta=\psi(\nu(\eta))\in \psi\big(\pgl_2(A)\big)\cap\pgl_2(A)
\]
for every $\eta\in\Gamma_2$. Hence
\[
K:=\pgl_2(A)\cap\psi\big(\pgl_2(A)\big)
   =\pgl_2(A)\cap \delta\,\pgl_2(A)^{(m)}\,\delta^{-1}
\]
has finite index in $\pgl_2(A)$. By Lemma~\ref{l:commensurator-frobenius},
$\delta\in\pgl_2(F)$.

It remains to show that $m=0$. Suppose, for contradiction, that $m>0$.
Let
\[
U(d):=
\left\{
u_f=\begin{bmatrix}1&f\\0&1\end{bmatrix}\in\pgl_2(A):
\deg_T f\le d
\right\}.
\]
Since $K$ has finite index in $\pgl_2(A)$, the subgroup
\[
B:=\{f\in A: u_f\in K\}
\]
has finite index in the additive group $A$. Therefore
\begin{equation}\label{eq:lowerU-new}
\#(U(d)\cap K)=\#\{f\in B:\deg_T f\le d\}\ge q^{d+O(1)}.
\end{equation}

Now take $u_f\in U(d)\cap K$. Since $u_f\in\psi\big(\pgl_2(A)\big)$, there exists
$g\in\pgl_2(A)$ such that
\[
u_f=\psi(g)=\delta\,\uptau^m(g)\,\delta^{-1}.
\]
Thus
$\delta^{-1}u_f\delta\in \pgl_2(A)^{(m)}
=\pgl_2(\F_q[T^{p^m}])$.

Choose a representative
\[
\tilde\delta=\begin{bmatrix}a&b\\ c&d\end{bmatrix}\in\gl_2(F)
\]
of $\delta$, and let $D=ad-bc$. A direct computation gives
\[
\tilde\delta^{-1}u_f\tilde\delta
=
\begin{bmatrix}
1+\frac{cd}{D}f & \frac{d^2}{D}f\\[4pt]
-\frac{c^2}{D}f & 1-\frac{cd}{D}f
\end{bmatrix}.
\]
At least one entry has the form $\alpha+\beta f$, with
$\alpha,\beta\in F$ and $\beta\neq 0$.

Since the class of $\tilde\delta^{-1}u_f\tilde\delta$ lies in
$\pgl_2(\F_q[T^{p^m}])$, there exists
$\lambda=\lambda(f)\in\C_\infty^\times$ such that
\[
\lambda\,\tilde\delta^{-1}u_f\tilde\delta
\in \gl_2\big(\F_q[T^{p^m}]\big).
\]
Taking determinants, and using
$\det(\tilde\delta^{-1}u_f\tilde\delta)=1$, we obtain
$\lambda^2\in \F_q^\times$.
Hence $\lambda$ belongs to a fixed finite set.

For the chosen entry $\alpha+\beta f$, set
\[
h_f:=\lambda(f)(\alpha+\beta f)\in \F_q[T^{p^m}].
\]
Since $\alpha,\beta\in F$ are fixed and $\deg_T f\le d$, we have
\[
\deg_T h_f\le d+O(1).
\]
For fixed $\lambda$ and fixed $h_f$, the equality $h_f=\lambda(\alpha+\beta f)$
determines $f$ uniquely. Therefore
\[
\#(U(d)\cap K)
\le
O(1)\cdot
\#\{h\in\F_q[T^{p^m}]:\deg_T h\le d+O(1)\}
=
q^{d/p^m+O(1)}.
\]
This contradicts \eqref{eq:lowerU-new} as $d\to\infty$, because $m>0$
implies $p^{-m}<1$. Hence $m=0$, and so
$\psi(\gamma)=\delta\gamma\delta^{-1}$.
\end{proof}

\begin{eg}
    The subset $L=\{(\gamma, \uptau(\gamma)) : \gamma\in\pgl_2(\C_\infty)\}$ is not allowed. Indeed, in this case the subgroup given by $ \uptau\left(\pgl_2(\F_q[T])\right)= \pgl_2(\F_q[T^p])$  has infinite index: the group $\pgl_2(\F_q[T])$ is not a  union of finitely many cosets of $\pgl_2(\F_q[T^p])$ (in fact, already $\F_q[T]$ is not a union of finitely many cosets of $\F_q[T^p]$).
\end{eg}

We recall the following standard geometric lemma. 

\begin{lem}\label{l:diagonal-invariant-curve}
Let $k$ be an algebraically closed field, and let $Y\subset \p^1_k\times \p^1_k$
be an irreducible closed curve with dominant projections. If $Y$ is invariant
under the diagonal action of $\pgl_2(k)$, then $Y$ is the diagonal.
\end{lem}

\begin{proof}
Let $\Delta=\{(z,z):z\in\p^1_k\}$
and let $\Delta^c:=(\p^1_k\times\p^1_k)-\Delta$.
The diagonal action of $\pgl_2(k)$ has two orbits, which are $\Delta$ and $\Delta^c$. Indeed, $\Delta^c$ is the set of ordered pairs of distinct points, and
$\pgl_2(k)$ acts transitively on such pairs.
If $Y\cap \Delta^c\neq\emptyset$, then invariance gives $\Delta^c\subseteq Y$. Since
$Y$ is closed and $\Delta^c$ is Zariski dense in $\p^1_k\times\p^1_k$, this would imply
$Y=\p^1_k\times\p^1_k$,
contradicting that $Y$ is a curve. Hence $Y\subseteq\Delta$. Since the
projections of $Y$ are dominant, $Y$ is not a point. Therefore $Y=\Delta$.
\end{proof}

\begin{proof}[Proof of Theorem~\ref{p:hyperbolicAL}]
By Proposition~\ref{p:nofrob}, we have
$S=
\left\{
\bigl(\gamma,\delta\gamma\delta^{-1}\bigr):
\gamma\in\pgl_2(\C_\infty)
\right\}$
for some $\delta\in\pgl_2(F)$.

Let
$\overline Z\subset \p^1_{\C_\infty}\times \p^1_{\C_\infty}$
be the Zariski closure of $Z$, and set
$Y:=(\id\times\delta^{-1})(\overline Z)$.
Then $Y$ is an irreducible closed curve with dominant projections.

We show that $Y$ is invariant under the diagonal action of $\pgl_2(\C_\infty)$. Indeed, by the definition of $S$, the curve
$\overline Z$ is $S$-invariant. Hence for every
$\gamma\in\pgl_2(\C_\infty)$,
$\bigl(\gamma,\delta\gamma\delta^{-1}\bigr)\overline Z=\overline Z$.
Conjugating by $\id\times\delta^{-1}$, we obtain
$(\gamma,\gamma)Y=Y$.
Thus $Y$ is invariant under the diagonal action.

By Lemma~\ref{l:diagonal-invariant-curve}, we get $Y=\Delta$.
Therefore
$\overline Z=(\id\times\delta)(\Delta)$.
Since $\delta\in\pgl_2(F)\subset\pgl_2(F_\infty)$, it preserves $\Omega$.
Thus, on $\Omega^2$,
\[
Z=\{(z,\delta z):z\in\Omega\}.
\]

By \eqref{e:modular-corr}, since $\delta\in\pgl_2(F)$, there exists
$N\in A-\{0\}$ such that
\[
\bigl(j(z),j(\delta z)\bigr)\in Y_0'(N)
\qquad\text{for all }z\in\Omega.
\]
Therefore $\boldsymbol j(Z)\subseteq Y_0'(N)$.
Since $Z=\{(z,\delta z):z\in\Omega\}$, then 
$\boldsymbol j(Z)=Y_0'(N)$.
Hence $\boldsymbol j(Z)$ is weakly special.
\end{proof}

\subsection{Rigid analytic Pila--Wilkie}

\subsubsection{Overconvergent version}

Let us begin by refreshing some terminology from  \cite{rigidpilawilkie}. We refer to \cite{kato} for an extensive and detailed treatment of any rigid-analytic notion we make use of here.
\\

We use an auxiliary formal parameter $t$, specialized to
$1/T$ in the arithmetic statements and applications;
thus $\F_q((t))$ is identified with $F_\infty$.

We denote by $\F_q[\![t]\!]\llangle \mathbf{x} \rrangle$, where $\mathbf{x}=(x_1,\dots,x_n)$, the {\em restricted power series $\F_q[\![t]\!]$-ring}, namely the $t$-adic completion of $\F_q[\![t]\!][\mathbf{x}]$. A {\em topologically of finite type} $\F_q[\![t]\!]$-algebra $A$ is an $\F_q[\![t]\!]$-algebra isomorphic to an $\F_q[\![t]\!]$-algebra of the form $\F_q[\![t]\!]\llangle \mathbf{x} \rrangle/\gota$, where $\gota$ is an ideal. A topologically of finite type $\F_q[\![t]\!]$-algebra is {\em admissible} if it is $t$-torsion free. We denote by $\mathbf{Ad}_{\F_q[\![t]\!]}$ the category of admissible $\F_q[\![t]\!]$-algebras and $\F_q[\![t]\!]$-algebra morphisms.

An {\em affinoid $\F_q(\!(t)\!)$-algebra} $\acal$ is an $\F_q(\!(t)\!)$-algebra isomorphic to an $\F_q(\!(t)\!)$-algebra of the form $A[1/t]$, where $A$ is  a topologically of finite type $\F_q[\![t]\!]$-algebra. The category of affinoid $\F_q(\!(t)\!)$-algebras is denoted by $\mathbf{Af}_{\F_q(\!(t)\!)}$. A {\em formal model} of  $\acal$ is a pair $(A,f)$ consisting of a topologically of finite type $\F_q[\![t]\!]$-algebra $A$ and of an isomorphism $f\colon A[1/t]\rightarrow \acal$. Such a formal model is admissible if so is $A$.
\\

Denote by $F_\alpha$ a non-archimedean local field containing $\F_q(\!(t)\!)$, and let $V_\alpha$ be its valuation ring with uniformizer $\varpi$. Let $\alpha\colon \F_q[\![t]\!]\rightarrow V_\alpha$ be an adic morphism such that $\alpha(t)=u\varpi^{r_\alpha}$, with $u\in V_\alpha^\times$ and $r_\alpha\in \Z_{>0}$. The fiber $(\Sp\acal)_\alpha$ is an $F_\alpha$-affinoid space.
We immediately move to the overconvergent setting (see \cite[4.3]{kato} for a general characterization of overconvergent subsets). For $\delta\in\Q_{>0}$, we denote by $\Sp\acal^{\natural \delta}$ the intersection of $\Sp\acal$ with the polydisk of radius $|t^{\delta}|$. For the rigid-analytic constructions, see \cite[Section 2.10]{rigidpilawilkie}.

We set
\[
(\Sp\acal)(\F_q(t),H):=\big\{f\in\Sp\acal\cap\F_q(t )^n : \max_{1\le i\le n}\{\log_qH(f_i)\}\le  \log_qH\big\}\ .
\]
Here $H(a/b)=q^{\max\{\deg_t a,\deg_t b\}}$ for coprime
$a,b\in\F_q[t]$ with $b$ monic.
An irreducible closed subset of $\Sp\F_q(\!(t)\!)\llangle\mathbf{x}\rrangle$ is said to be {\em algebraic} if it is an irreducible component of a closed set defined by an ideal of $\F_q(\!(t)\!)\llangle\mathbf{x}\rrangle$ obtained by extension from an ideal of $\F_q(\!(t)\!)[\mathbf{x}]$. Lastly, we define $(\Sp\acal)^\text{alg}$ as the union of every positive dimensional algebraic set in $\Sp\acal$, and $(\Sp\acal)^{\text{tran}}$ as its complement.

We can now state the overconvergent form of the counting theorem (see \cite[Theorem 1.3.1]{rigidpilawilkie} for the proof).  Let $\acal$ be an affinoid $\F_q(\!(t)\!)$-algebra. Then for every $\varepsilon >0$ and for any positive integer $r_\alpha$ there exists a positive constant $C(\acal,\varepsilon,\delta,r_\alpha)$ such that
\[
		\#(\Sp \acal^{\natural\delta})^{\text{tran}}(\F_q(t),H)\le C(\acal,\varepsilon,\delta,r_\alpha)\cdot H^\varepsilon\;.
\]

	\subsubsection{Refinement in blocks}\label{s:bk}
    Following \cite[Theorem 5.1.3]{rigidpilawilkie}, we now introduce the counting theorem in an overconvergent version which is uniform in families and that, in the archimedean setting, corresponds to the block version of the counting theorem as in \cite[Theorem 3.6]{pilaAO}.
\\

Let $\boldsymbol{\delta}=(\delta_1,\dots,\delta_n)\in\Q_{>0}^n$ and $\nu=(\nu_1,\dots,\nu_n)\in\N^n$; let also $\boldsymbol{\delta}\cdot\nu=\delta_1\nu_1+\dots+\delta_n\nu_n$. We define the affinoid algebra with a convergence condition as
\begin{equation}\label{e:polydisk}
\acal \llangle\mathbf{x};\boldsymbol{\delta} \rrangle:=\left\{\sum_\nu a_\nu\mathbf{x}^\nu\in\acal[\![\mathbf{x}]\!] : |t|^{-\boldsymbol{\delta}\cdot\nu}|a_\nu|\to 0\;\text{as}\;|\nu|\to\infty\right\},
\end{equation}
where $\acal$ is an affinoid $\F_q((t))$-algebra and
$\mathbf{x}=(x_1,\dots,x_n)$. Put
\[
\calb=\acal\llangle\mathbf{x};\boldsymbol{\delta}\rrangle/\gota
\]
for an ideal $\gota\subseteq\acal\llangle\mathbf{x};\boldsymbol{\delta}\rrangle$, and choose a tuple
\[
\mathbf{f}=(f_1,\dots,f_n)\in\calb^n
\]
of elements which are algebraic over $\acal[\mathbf{x}]$ and topologically generate $\calb$ over $\acal$. Consider the morphism
\[
\varphi\colon\calx=\Sp\calb\rightarrow\caly=\Sp\acal
\]
giving a family in the polydisk of polyradius $\vert t^{-\boldsymbol{\delta}}\vert=(\vert t\vert^{-\delta_1},\dots,\vert t\vert^{-\delta_n})$.

Let  $\alpha\colon\Sp K\rightarrow \caly$ be a classical point: by \cite[Lemma~2.2.2]{rigidpilawilkie}, any such a point of a rigid space of finite type over $\text{Spf}(\F_q[\![t]\!])^\text{rig}$ is isomorphic to $\text{Spf}(V)^\text{rig}$, where $V$ is a complete discrete valuation ring with finite residue field. For the general notion of classical point, see \cite[8.2.(c)]{kato}.

Consider $\calx_\alpha=\Sp(\calb\widehat{\otimes}_\acal K)$. Let $F_{\mathrm{loc}}\subseteq K$ be a non-archimedean local field
containing $\F_q((t))$, with $K/F_{\mathrm{loc}}$ finite, and let
$V_{F_{\mathrm{loc}}}$ be its valuation ring. Consider the classical point $\beta\colon\Sp F_{\mathrm{loc}}\rightarrow\Sp\F_q(\!(t)\!)$, so that we have the following commutative diagram
\[
\begin{tikzcd}
    \calx_\alpha \arrow{r}\arrow{d}& \Sp \calb \arrow{d}{\varphi}
    \\
    \Sp K \arrow{d}\arrow{r}{\alpha}& \Sp\acal\arrow{d}
    \\
    \Sp F_{\mathrm{loc}}\arrow{r}{\beta}&\Sp\F_q(\!(t)\!)
\end{tikzcd}
\]
For such a diagram, we say that $\alpha$ lies over $\beta$. We underline that, although our goal is to count $F_{\mathrm{loc}}$-points in affinoid spaces, a finite extension $K$ is necessary whenever we consider the base change $\F_q(\!(t)\!)\rightarrow \F_q(\!(t^{1/m})\!)$ for some positive integer $m$. Fix a finite extension $K'$ of $F_{\mathrm{loc}}$ such that $K\subseteq K'$ and consider the classical point $\alpha'\colon\Sp K'\rightarrow\calx_\alpha$, with corresponding algebra map $\alpha'^{\#}\colon\calb_\alpha\rightarrow K'$ such that $\alpha'^{\#}(x_i)\in F_{\mathrm{loc}}$ for every $i\in\{1,\dots,n\}$. Then we denote by $H(\alpha';\boldsymbol{f})$ the height of $\alpha'$, which is defined as 
\[
H(\alpha';\boldsymbol{f}):=\max_{1\le i\le n}\{H(\alpha'^\#(f_i))\}.
\]
Here $H$ is the rational-function height on $\F_q(t)$ defined above.
We can thus consider
\[
\calx_\alpha(V_{F_{\mathrm{loc}}},H;\boldsymbol{f}):=\big\{\alpha'\colon\Sp K'\rightarrow \calx_\alpha^\natural \;|\; H(\alpha';\boldsymbol{f})\le H,\; \alpha'^\#(f_i)\in\F_q(t),\;\alpha'^\#(x_i)\in F_{\mathrm{loc}}\;\text{for all}\; i\big\}.
\]

For a finite global extension $L/F$ embedded in $F_{\mathrm{loc}}$,
write $\calx_\alpha(L,H;\boldsymbol f)$ for the analogous set with
both coordinate tuples in $L$, using absolute Weil height in the
same formula for $H(\alpha';\boldsymbol f)$.

In order to state the main counting theorem, we still need a few more definitions.
\\An irreducible analytic subspace $Z\hookrightarrow \Sp K\llangle \mathbf{x};\boldsymbol{\delta}\rrangle$ is {\em algebraic} if there exists an algebraic subvariety $\calz\hookrightarrow \Spec K[\mathbf{x}]$ such that $\dim Z=\dim\calz$ and such that every classical point of $Z$ lies in $\calz$.  Naturally, an analytic subspace $Z$ is algebraic if every irreducible component is algebraic. For $\varphi$ and $\alpha$ as above, the family $\varphi$ has {\em algebraic fibers} if every fiber $\calx_\alpha$ is algebraic. Moreover, by {\em maximal fiber dimension} of $\varphi$ we simply mean
$\max_\alpha\dim\calx_\alpha$
where $\alpha$ varies over the classical points of $\caly$.

For $\acal$ and $\calb$ as above, let $A$ be an admissible formal model of $\acal$, and set
\[
B:=A\llangle\mathbf{x};\boldsymbol{\delta}\rrangle/
\bigl(\gota\cap A\llangle\mathbf{x};\boldsymbol{\delta}\rrangle\bigr),
\]
which is an admissible formal model of $\calb$. Let also $r_\beta$ be the positive integer such that  $t=u\varpi^{r_\beta}$, for $\varpi$ the uniformizer of $V_{F_{\mathrm{loc}}}$ and $u$ a unit. Naturally, we consider
\[
A\llangle \boldsymbol{x};\boldsymbol{\delta} \rrangle=\left\{
\sum_\nu a_\nu\mathbf{x}^\nu
\in\acal\llangle\mathbf{x};\boldsymbol{\delta}\rrangle:
|a_\nu|\le|t|^{\boldsymbol{\delta}\cdot\nu}
\right\},
\]
which again is independent of the norm $|\cdot|$.
Let $T$ denote a finite directed rooted tree with the orientation away from the root. If $v\rightarrow u$ is a directed edge of $T$, then $u$ is dubbed {\em child} of $v$, and {  a} vertex with no children is called a {\em leaf}. By considering the rooted tree $T$ as a category with a minimal object, we obtain functors of admissible and affinoid algebras
\[
A_*\colon T\rightarrow \mathbf{Ad}_{\F_q[\![t]\!]}
\qquad
\text{and}
\qquad
A[1/t]_*\colon T\rightarrow \mathbf{Af}_{\F_q(\!(t)\!)}\;
\]
called {\em diagrams}. For a diagram $A_*$, an ideal $\gota_*$ is simply a collection of ideals $\gota_v$ of $A_v$ for every vertex $v$ such that, for any edge $v\rightarrow u$, one has $\gota_vA_u\subset\gota_u$. In particular, the diagram $\F_q[\![t^{1/N_*}]\!]$ modeled on $T$ consists of $\F_q[\![t^{1/N_v}]\!]$ for any vertex $v$, and a positive integer $N_v$ such that, $N_u$ is a multiple of $N_v$ for any edge $v\to u$, and $\F_q[\![t^{1/N_v}]\!]\hookrightarrow\F_q[\![t^{1/N_u}]\!]$ is uniquely determined by  $t^{1/N_v}\to (t^{1/N_u})^{N_u/N_v}$.
Lastly, a {\em transformation diagram} (modeled on $T$) consists of a chain of morphisms
\[
\F_q[\![t]\!]\rightarrow A_*\rightarrow A_*\llangle\mathbf{x}_*;\boldsymbol{\delta}_* \rrangle\rightarrow B_*=A_*\llangle\mathbf{x}_*;\boldsymbol{\delta}_* \rrangle/\gota_*
\]
of diagrams of admissible $\F_q[\![t]\!]$-algebras such that:
\begin{enumerate}
    \item[(a)] for any vertex $v$, $\mathbf{x}_v$ is a tuple of $n$ variables and $\boldsymbol{\delta}_v\in\Q_{>0}^n$;
    \item[(b)] for any edge $v\rightarrow u$, we have $\boldsymbol{\delta}_v\ge \boldsymbol{\delta}_u$;
    \item[(c)] for any edge $v\rightarrow u$, we have $\gota_vA_u\llangle\mathbf{x}_u;\boldsymbol{\delta}_u\rrangle\subset \gota_u$.
\end{enumerate}
Analogously, by considering a diagram $A[1/t]_*\colon T\rightarrow\mathbf{Af}_{F_q(\!(t)\!)}$, one obtains the corresponding transformation diagram of affinoid $\F_q(\!(t)\!)$-algebras.

   \begin{thm}\label{t:blockcounting}
       For every $\varepsilon>0$ there exists a transformation diagram
       \[
       \F_q[\![t^{1/N_*}]\!]\rightarrow A_*\rightarrow A_*\llangle\mathbf{x}_*,\boldsymbol{\delta}_*\rrangle\rightarrow B_*=A_*\llangle\mathbf{x}_*,\boldsymbol{\delta}_*\rrangle/\gota_*
       \]
       of admissible $\F_q[\![t]\!]$-algebras modeled on a rooted tree $T$ with root $v_0$, and for every positive integer $\sigma$, there exists a constant $C(B,\varepsilon,\sigma)$ such that:
       \begin{enumerate}
           \item the map $A_{v_0}\rightarrow B_{v_0}$ coincides with $A\rightarrow B$;
           \item for any leaf $w$, the family $\varphi_w\colon\calx_w=\Sp\calb_w\rightarrow\caly_w=\Sp\acal_w$ has equidimensional algebraic fibers. Moreover, the family is algebraic over $\caly_w$;
           \item for any classical point $\alpha\colon\Sp K\rightarrow\caly$ above $\beta\colon\Sp F_{\mathrm{loc}}\rightarrow \Sp\F_q(\!(t)\!)$, with $r_\beta\le\sigma$, we have
          \[\calx_\alpha(V_{F_{\mathrm{loc}}},H;\mathbf{f})\ \subset \bigcup_{(w_j,\alpha_j)}(\calx_{w_j})_{\alpha_j}(V_{F_{\mathrm{loc}}},H;\mathbf{f})\ \subset \ \calx_\alpha
           \]
           where $(w_j,\alpha_j)$ runs over a set of at most $C(B,\varepsilon,\sigma)\cdot H^\varepsilon$ pairs, where $w_j$ is a leaf of $T$ and $\alpha_j$ is a classical point of $\caly_{w_j}$ lying over $\alpha$.
       \end{enumerate}
The same conclusion holds for $\calx_\alpha(L,H;\boldsymbol f)$,
where both coordinate tuples lie in a finite extension $L/F$
embedded in $F_{\mathrm{loc}}$, and heights are absolute Weil
heights. For fixed $d$, the diagram and the bound
$C(B,\varepsilon,\sigma,d)H^\varepsilon$ are uniform for
$[L:F]\le d$ and $r_\beta\le\sigma$.
\end{thm}

\begin{proof}
This amounts to \cite[Theorem~5.1.3 and Remark~5.1.4]{rigidpilawilkie}, with the following two (minor) additions: the second inclusion in (iii) and the extension of the counting bound uniformly to finite extensions $L/F$ of bounded degree.

For the second inclusion in (iii), extend the $\alpha_j$'s to a common finite extension $K'/K$ and restrict to unit polydisks, denoted by $\natural$. The coordinate changes in
\cite[Definition~3.2.1 and proof of Proposition~5.1.5]
{rigidpilawilkie} are polynomial automorphisms over $\F_p$
preserving these polydisks. After identifying coordinates,
condition (c) enlarges the defining ideals, giving surjections
of the specialized coordinate rings. Composing along
$v_0\rightarrow w_j$ yields a closed immersion
$\bigl((\calx_{w_j})_{\alpha_j}\bigr)^\natural
\hookrightarrow
\calx_\alpha^\natural\times_{\Sp K}\Sp K'$.
The polynomial coordinate changes preserve algebraicity
and dimension. These block images give the second inclusion.

For the $L/F$-version, let $\Delta\ne0$ be an interpolation
determinant of size $\mu$ and monomial degree at most $D_0$,
evaluated at points of $L^n$ with coordinate heights at most $H$.
The product formula gives, for the chosen place $w\mid\infty$,
\[
|\Delta|_w
\ge H^{-nD_0\mu[L:F]/[L_w:F_\infty]}
\ge |t|^{D_0\mu\log_{\#\F_w}(H^{nd})},
\]
where $\F_w$ is the residue field of $L_w$ and $|t|=q^{-1}$.
Use this in place of
\cite[Proposition~4.2.3]{rigidpilawilkie}.
The resulting bound is $H^\varepsilon$, uniformly since
$[L_w:F_\infty]\le d$.
The coordinate changes preserve $L$, so the induction in
\cite[Proposition~5.1.5]{rigidpilawilkie} applies with these
modified interpolation bounds.
\end{proof}

Note that in part (3) of the above theorem $\mathbf{f}$ on the right hand side still denotes, by slight abuse of notation, the image of $\mathbf{f}$ in $(B_{w_j})_{\alpha_j}$.

\begin{rmk}
    Theorem \ref{t:blockcounting} is not effective.  A path to an effective version could involve extending the results of \cite{gbfoliation} to the function field setting.
\end{rmk}

\subsection{Andr\'e--Oort--Manin--Mumford}
	\subsubsection{Manin--Mumford for  two Drinfeld modules}
   
We start with a  specialization argument.    {  
Suppose that an irreducible closed curve $V$ contains infinitely many torsion points but
is not a torsion subvariety. Translating by one of these points,
we may assume that $0\in V$ and that $V$ is not a
sub-$\Phi$-module. Let $F(X,Y)$ be an irreducible defining polynomial of $V$.
Then either $F(u+v)\ne0$ for some $u,v\in V$, or
$F(\Phi_a(u))\ne0$ for some $u\in V$ and $a\in A$.

Consider a normal finitely generated $F$-subalgebra
$R\subset\C_\infty$ containing the coefficients of both modules
and of $F(X,Y)$, the coordinates of the chosen points, the
inverses of the two leading coefficients, and the inverse of
the nonzero polynomial value above. After localizing $R$,
the curve spreads out over $S=\Spec R$ with geometrically
integral curve fibers. Every geometric fiber contains the origin and retains the
chosen failure of closure, so it is not a submodule for the
specialized action. Since a translate of a subgroup containing
the origin equals that subgroup, no geometric fiber is even
a translate of a submodule.

Choose a closed point $s\in S$. Its residue field $L$ is finite
over $F$. Since every nonzero element of $A$ is invertible in
$R$, all torsion schemes are finite \'etale over $S$.
Over the strict henselization at $s$, they split into disjoint
sections. Choosing a geometric generic point therefore gives
compatible specialization maps for all torsion orders, which
are injective on the full torsion subgroup. Since the curve
is closed, its torsion points specialize into $V_s$.

Thus $V_s$ is a geometrically integral curve containing
infinitely many geometric torsion points, but is not a torsion
subvariety. A counterexample over $\C_\infty$ would therefore
give one over a finite extension of $F$, proving the reduction.
}

We shall thus reduce to consider two Drinfeld $A$-modules $\cald$ and $\cald'$ defined over $L$. For their product $\cald\times \cald'$ and a $\C_\infty$ curve $V$ in it defined by the zero locus of a polynomial $F(X,Y)\in \C_\infty[X,Y]$, it is enough to take $R$ with also the coefficients of $\varphi'_T$ and of $F(X,Y)$, so that our curve spreads out to a closed subscheme of $\A^2_R$. 

A {\em torsion subvariety} of the Anderson $A$-module $\cald\times\cald'$ is defined to be a translate of an irreducible sub-$\Phi$-module of $\cald\times\cald'$ by a torsion point.

	 \begin{thm}
	 	Let $\cald$ and $\cald'$ be two Drinfeld $A$-modules over $\C_\infty$ of the same rank $r$. Let $V$ be an irreducible closed $\C_\infty$-subvariety of  $\cald\times\cald'(\C_\infty)$.  Then $V\cap (\cald\times\cald')_{\text{tor}}(\C_\infty)$ is Zariski dense in $V$ if and only if $V$ is a torsion subvariety. 
	 \end{thm}

\begin{proof}
The reverse implication is clear: torsion points are Zariski dense in every
sub-$\Phi$-module, and translating by a torsion point preserves torsion
points.

We prove the converse. By the specialization argument above, after spreading
out the coefficients of $\cald$, $\cald'$, and of the equations defining
$V$, and after specializing at a closed point of the corresponding parameter
space, we may assume that $\cald$, $\cald'$, and $V$ are defined over a
finite extension $L/F$. %This specialization preserves rank and generic
%characteristic, and the torsion remains finite \'etale.

We first reduce to the case where $V$ is a curve. If $\dim V=0$, then
Zariski density of torsion points in $V$ means that $V$ is a torsion point.
If $\dim V=2$, then $V=\cald\times\cald'$, which is itself a
sub-$\Phi$-module. Thus we may assume that $V$ is an irreducible curve.

We now show that $V$ must be a translate of a sub-$\Phi$-module. Suppose,
towards a contradiction, that $V$ is not such a translate.

\proofstep{The torsion logarithm polydisk}
Let $\Lambda$ and $\Lambda'$ be the period lattices of $\cald$ and
$\cald'$, and fix $A$-bases
\[
(\xi_1,\dots,\xi_r)
\qquad\text{and}\qquad
(\xi'_1,\dots,\xi'_r)
\]
of $\Lambda$ and $\Lambda'$, respectively. Enlarging $L$ by a finite
extension if necessary, we may assume that all these periods lie in a finite
extension $\widetilde L_{\widetilde\infty}/L_\infty$.
Indeed, for $m\gg 1$, the points
\[
\exp_\Lambda(\xi_i/T^m)
\]
are $T^m$-torsion points of $\cald$, hence algebraic over $L$; once they
lie in the disk of convergence of $\log_\Lambda$, we have
\[
\xi_i=T^m\log_\Lambda\bigl(\exp_\Lambda(\xi_i/T^m)\bigr),
\]
and similarly for the $\xi'_j$.

By Lemma~\ref{l:torsionball}, every logarithmic preimage of a torsion point of
$\cald\times\cald'$ has a representative of the form
\[
\left(
\sum_{i=1}^r a_i\xi_i,\,
\sum_{j=1}^r b_j\xi'_j
\right)
\qquad
\text{with}
\qquad
a_i,b_j\in F
\quad\text{and}\quad
|a_i|_\infty, |b_j|_\infty\le q^{-1}.
\]
Thus the torsion logarithms are parametrized by $F$-rational points in the
bounded affinoid polydisk
\[
\mathbb D
:=
\Sp \widetilde L_{\widetilde\infty}
\llangle
TX_1,\dots,TX_{2r}
\rrangle .
\]

Define the $\widetilde L_{\widetilde\infty}$-linear map
\[
\ell\colon \A^{2r}_{\C_\infty}\rightarrow \G_{a,\C_\infty}^2
\qquad
(a_1,\dots,a_r,b_1,\dots,b_r)
\mapsto
\left(
\sum_{i=1}^r a_i\xi_i,\,
\sum_{j=1}^r b_j\xi'_j
\right).
\]
Let $G_V(X,Y)\in L[X,Y]$ be an irreducible polynomial defining the curve
$V$. We set
\[
\calx
:=
\left\{
\mathbf v\in\mathbb D:
G_V\bigl(\exp_\Phi(\ell(\mathbf v))\bigr)=0
\right\}.
\]
This is an affinoid subspace of $\mathbb D$. Write
\[
\calx=\Sp\calb,
\]
and consider the tuple
\[
\boldsymbol f
=
(TX_1,\dots,TX_{2r})\in\calb^{2r}.
\]
Taking $F_{\mathrm{loc}}=F_\infty$, let
\[
\alpha\colon \Sp \widetilde L_{\widetilde\infty}\rightarrow \Sp
\widetilde L_{\widetilde\infty}
\]
be the unique classical point of the base. With this notation, the torsion
logarithms above give points of
\[
\calx_\alpha(V_{F_{\mathrm{loc}}},H;\boldsymbol f)
\]
for suitable $H$.

\proofstep{The block-counting upper bound}
We claim that, for every $\varepsilon>0$,
\begin{equation}\label{e:mm-upper}
\#\calx_\alpha(V_{F_{\mathrm{loc}}},H;\boldsymbol f)\ll_\varepsilon H^\varepsilon.
\end{equation}

Apply Theorem~\ref{t:blockcounting} to
$\calx\rightarrow \Sp \widetilde L_{\widetilde\infty}$
with the tuple $\boldsymbol f$. For every $\varepsilon>0$, the theorem gives
a transformation diagram and a finite set
$\mathcal J=\mathcal J(\alpha,H)$
of pairs $(w_j,\alpha_j)$, with
\[
\#\mathcal J\ll_\varepsilon H^\varepsilon,
\]
such that
\[
\calx_\alpha(V_{F_{\mathrm{loc}}},H;\boldsymbol f)
\subseteq
\bigcup_{j\in\mathcal J}
(\calx_{w_j})_{\alpha_j}(V_{F_{\mathrm{loc}}},H;\boldsymbol f)
\subseteq
\calx_\alpha.
\]
Here the fibers
$(\calx_{w_j})_{\alpha_j}$
are algebraic, and their maps into $\calx_\alpha$ are closed immersions.

After base change to $\C_\infty$, let $\mathscr B$ be a reduced
irreducible positive-dimensional component of a block image in $\calx_\alpha$. We analyze its image under $\ell$.

If $\ell|_{\mathscr B}$ is nonconstant, then
\[
W_0:=\overline{\ell(\mathscr B)}^{\mathrm{Zar}}
\subset \G_{a,\C_\infty}^2
\]
is positive dimensional. Since $\mathscr B\subseteq\calx$, we have
\[
\exp_\Phi\bigl(\ell(\mathscr B)\bigr)\subseteq V.
\]
{
Let $\overline{\mathscr B}$ be the Zariski closure of
$\mathscr B$ in $\A^{2r}_{\C_\infty}$.
By algebraicity, $\mathscr B$ has the same dimension as
$\overline{\mathscr B}$ and is an analytic component of
$\overline{\mathscr B}^{\mathrm{an}}\cap\mathbb D$.
It therefore contains a nonempty admissible open of
$\overline{\mathscr B}^{\mathrm{an}}$.
Since $\exp_\Phi$ is entire, the  identity theorem on this
irreducible analytification gives
\[
G_V\circ\exp_\Phi\circ\ell=0
\quad\text{on }\overline{\mathscr B}^{\mathrm{an}}.
\]
By Chevalley's theorem, $\ell(\overline{\mathscr B})$
contains a nonempty Zariski open subset of its closure $W_0$.
Applying the identity theorem again, now to
$G_V\circ\exp_\Phi$ on $W_0^{\mathrm{an}}$, yields
\[
W_0^{\mathrm{an}}\subseteq\exp_\Phi^{-1}(V).
\]
}
Let $W$ be a maximal irreducible algebraic subvariety of $\exp_\Phi^{-1}(V)$
containing $W_0$. Then $W$ is positive dimensional. By
Theorem~\ref{alw}, the image
$\exp_\Phi(W)$
is a positive dimensional translate of a sub-$\Phi$-module contained in
$V$. Since $V$ is an irreducible curve, this forces $V$ itself to be a
translate of a sub-$\Phi$-module, contrary to our assumption.

Hence, for every positive dimensional algebraic block appearing in the
block-counting cover, the map $\ell$ must be constant on that block.

Now suppose that $\ell(\mathbf v)=\ell(\mathbf v')$ for two counted points
$\mathbf v,\mathbf v'\in F^{2r}$.
Writing $\mathbf v=(a_1,\dots,a_r,b_1,\dots,b_r)$,
and
$\mathbf v'=(a_1',\dots,a_r',b_1',\dots,b_r')$,
we get $\sum_i(a_i-a_i')\xi_i=0$ and
$\sum_j(b_j-b_j')\xi_j'=0$.
Since the $\xi_i$'s and $\xi_j'$'s are $F$-linearly independent, all
coefficients vanish. Hence $\mathbf v=\mathbf v'$. Therefore any block on
which $\ell$ is constant contains at most one counted $F$-point.

On the other hand, zero dimensional algebraic fibers in the fixed transformation diagram contain only $O(1)$ counted points, uniformly in
$\alpha$ and $H$. The fixed algebraic families have bounded degrees, so each
fiber gives only $O(1)$ distinct irreducible Zariski closures
of blocks. Blocks with the same Zariski closure have the
same constant value under $\ell$, and hence together contain
at most one counted $F$-point. Therefore every algebraic fiber appearing in the
block counting cover contributes $O(1)$ counted points. Since there are
$\ll_\varepsilon H^\varepsilon$ such fibers, we obtain
\[
\#\calx_\alpha(V_{F_{\mathrm{loc}}},H;\boldsymbol f)\ll_\varepsilon H^\varepsilon,
\]
which proves \eqref{e:mm-upper}.

\proofstep{The Galois lower bound}
We now derive the contradiction. Suppose that $V$ contains a Zariski dense set
of torsion points. Since $V$ is a curve, and the set of torsion points of
bounded order is finite, $V$ contains torsion points of arbitrarily large
order.

Let
\[
x\in V(\C_\infty)\cap(\cald\times\cald')_{\mathrm{tor}}
\]
be a torsion point of order $a\in A-\{0\}$. For each
\[
\sigma\in\Hom_L(L(x),\C_\infty),
\]
the conjugate $x^\sigma$ is again a torsion point of order $a$, and it lies
on $V$ because $V$ is defined over $L$.

Choose for each $x^\sigma$ the representative of a logarithm given by
Lemma~\ref{l:torsionball}:
\[
z^\sigma
=
\ell(\mathbf v^\sigma),
\qquad
\mathbf v^\sigma\in F^{2r},
\qquad
|v_i^\sigma|_\infty\le q^{-1}.
\]
Then $\mathbf v^\sigma\in\calx_\alpha$. Moreover, by
Lemma~\ref{l:heightorder},
\[
H(\mathbf v^\sigma)\le |\ord(x^\sigma)|_\infty=|a|_\infty.
\]
Thus, after replacing $H$ by a constant multiple of $|a|_\infty$, all
conjugates $x^\sigma$ give counted points in
$\calx_\alpha(V_{F_{\mathrm{loc}}},H;\boldsymbol f)$.
These points are distinct: if
$\mathbf v^\sigma=\mathbf v^\tau$,
then $z^\sigma=z^\tau$,
and hence $x^\sigma=\exp_\Phi(z^\sigma)=\exp_\Phi(z^\tau)=x^\tau$.

Therefore
\[
\#\calx_\alpha(V_{F_{\mathrm{loc}}},H;\boldsymbol f)
\ge [L(x):L].
\]
By Lemma~\ref{masserboundtorsion}, there are constants $C'>0$ and
$\eta>0$, independent of $x$, such that
\[
[L(x):L]\ge C'|a|_\infty^\eta.
\]
Combining this with the upper bound \eqref{e:mm-upper}, with
$H\ll |a|_\infty$, gives
\[
|a|_\infty^\eta
\ll
|a|_\infty^\varepsilon
\]
for every $\varepsilon>0$. Choosing $\varepsilon<\eta$, this is impossible
for $|a|_\infty$ sufficiently large. Hence $V$ cannot contain torsion
points of arbitrarily large order, contradicting the assumed Zariski density.

We have proved that if torsion points are Zariski dense in $V$, then $V$ is
a translate of a sub-$\Phi$-module.

\proofstep{The translate is torsion}
It remains to show that the translating point may be chosen torsion. Write
\[
V=\mathscr F+a
\]
for a sub-$\Phi$-module $\mathscr F$ and a point $a$. Since torsion points
are Zariski dense in $V$, choose a torsion point
\[
d\in V.
\]
Then $d=b+a$
for some $b\in\mathscr F$, hence
$a=d-b$.
Since $\mathscr F$ is a subgroup, translation by $a$ is the same as
translation by $d$:
\[
V=\mathscr F+a=\mathscr F+d.
\]
Thus $V$ is a translate of a sub-$\Phi$-module by a torsion point, i.e. a
torsion subvariety.
\end{proof}

\subsubsection{Andr\'e--Oort for a product of Drinfeld modular curves}

As in the characteristic zero case, the level structure plays no role in the Andr\'e--Oort statement that we are going to prove. In fact, raising the level from $M$ to $M'$ corresponds to taking a finite morphism of Drinfeld modular curves $Y(M')\rightarrow Y(M)$. Hence the preimage  of a  modular variety in $Y(M)^2$ consists of a finite union of modular varieties in $Y(M')^2$.

 \begin{thm}\label{t:ao}
    Assume that $q$ is  odd. Let $C\subset \A^2_{\C_\infty}$ be an irreducible algebraic curve. Then $C(\C_\infty)$ contains an infinite subset of CM points if and only if $C$ is a special subvariety.
 \end{thm}

\begin{proof}

The reverse implication is obvious.
We prove the converse by contradiction. Suppose that $C$ is not special, but
that $C(\C_\infty)$ contains a Zariski dense set of CM points.

\proofstep{Preliminary reductions and choice of a CM point}
Since CM points are algebraic over $F$, the Zariski density of these points
implies that $C$ is defined over $F^{\mathrm{alg}}$. Fix a finite extension
$M/F$ over which $C$ is defined.

We may assume that $C$ is neither vertical nor horizontal. Indeed, if $C=\{a\}\times\A^1_{\C_\infty}$ is
vertical, then the existence of a CM point on $C$ forces $a$ to be CM, so $C$ is
special. The horizontal case is analogous.

Fix a CM point
\[
x=(x_1,x_2)\in C(\C_\infty)
\]
whose two CM discriminants are sufficiently large. Write
\[
x_i=j(z_i')
\]
with $z_i'\in\Omega$ quadratic. After acting by an element of
$\pgl_2(A)^2$, choose representatives
\[
z=(z_1,z_2)\in \boldsymbol j^{-1}(x)\cap
\bigl(\calf_{K_1}\times\calf_{K_2}\bigr).
\]
For $i=1,2$, let $\calo_i\subset K_i$ be the CM order of the Drinfeld
module corresponding to $z_i$, and let $D_i$ be its discriminant. By CM theory \eqref{e:cmorbit},
\[
[K_i(j(z_i)):K_i]=\#\pic(\calo_i).
\]
Moreover, since $z_i\in\calf_{K_i}$, Lemma~\ref{l:heightfundamentaldomain}(1)
gives
\[
H(z_i)^2\le |D_i|_\infty.
\]
By Siegel's lower bound \eqref{siegelclassnumber},
\[
\#\pic(\calo_i)\gg_\varepsilon |D_i|_\infty^{1/2-\varepsilon}.
\]

Set
\[
j_i:=j(z_i),
\qquad
E:=M K_1K_2,
\qquad
L:=E(j_1,j_2),
\]
and define
\[
\Sigma:=\Hom_E(L,\C_\infty).
\]
Since $C$ is defined over $M\subset E$, every conjugate $x^\sigma=(j_1^\sigma,j_2^\sigma)$
for $\sigma\in\Sigma$, again lies on $C$.

\proofstep{Galois orbit lower bound}
We first record the lower bound for $\#\Sigma$. For $i=1,2$,
\[
[L:E]\ge [E(j_i):E]
      \ge \frac{[K_i(j_i):K_i]}{[E:K_i]}.
\]
Since $M/F$ is fixed and $[K_1K_2:K_i]\le 2$, the degrees
$[E:K_i]$ are bounded independently of the CM point. Therefore
\[
\#\Sigma=[L:E]\gg \max_i \#\pic(\calo_i).
\]
Setting $D:=\max\{|D_1|_\infty,|D_2|_\infty\}$,
we obtain
\begin{equation}\label{e:galois-lower}
\#\Sigma\gg_\varepsilon D^{1/2-\varepsilon}.
\end{equation}

\proofstep{Conjugates in the fundamental domain}
For every $\sigma\in\Sigma$, fix representatives
$z^\sigma=(z_1^\sigma,z_2^\sigma)\in
\calf_{K_1}\times\calf_{K_2}$ such that
\[
\boldsymbol j(z^\sigma)=x^\sigma.
\]
These points are again quadratic with CM orders $\calo_1,\calo_2$. In
particular,
\[
H(z_i^\sigma)^2\le |D_i|_\infty
\qquad
\text{for}
\quad
i=1,2.
\]

\proofstep{Truncating the cusp}
We now show that many of these representatives are not too deep in the cusp.
For fixed $i$, the restrictions of the embeddings $\sigma\in\Sigma$ to
$E(j_i)$ run through all embeddings of $E(j_i)$ over $E$, with equal
multiplicity. Therefore, by the definition of the Weil height,
\[
\frac{1}{\#\Sigma}
\sum_{\sigma\in\Sigma}
\log_q^+|j_i^\sigma|_\infty
\ll h(j_i),
\]
where the implicit constant depends only on $M/F$.

By the preceding average estimate, after increasing the implicit constant, at least
three quarters of the embeddings $\sigma\in\Sigma$ satisfy
\[
\log_q^+|j_i^\sigma|_\infty\ll h(j_i).
\]
Intersecting the good subsets for $i=1,2$, we obtain a subset
$\Sigma'\subseteq\Sigma$
with
\[
\#\Sigma'\ge \frac12\#\Sigma
\]
such that, for every $\sigma\in\Sigma'$ and $i=1,2$,
\[
\log_q^+|j_i^\sigma|_\infty\ll h(j_i).
\]

By \cite[Lemma~5.5]{gekelerj} and \eqref{jexpansion},
\[
|w|_i\ll_q 1+\log_q^+|j(w)|_\infty
\qquad (w\in\calf).
\]
Hence
\[
|z_i^\sigma|_i\ll 1+h(j_i)
\qquad
\text{for}
\qquad
\sigma\in\Sigma',\;\; \text{and}\;\; i=1,2.
\]
For the height comparison, work over an extension with stable reduction at every finite place, using $\widetilde h_{\mathrm{Tag}}=(\ln q)h_{\mathrm{Tag}}$.
Combining Lemma~\ref{colmezbound} with Proposition~\ref{p:taguchi-weil}, and
shrinking $\varepsilon$ if necessary, we obtain
\[
1+h(j_i)\ll_\varepsilon |D_i|_\infty^\varepsilon.
\]
Choose $C_\varepsilon\ge1$, independently of the CM point,
so that, for every $\sigma\in\Sigma'$,
\[
z_i^\sigma\in
\calr_i:=\big\{w\in\calf_{K_i}: |w|_i\le C_\varepsilon|D_i|_\infty^\varepsilon\big\}.
\]
\proofstep{Covering and pigeonhole principle}
We now apply Lemma~\ref{l:cover}. Write
\[
K_i=F\big(\sqrt{\Delta_i}\big)
\]
with $\Delta_i\in A$ squarefree, and set
\[
R_i:=C_\varepsilon|D_i|_\infty^\varepsilon,
\qquad
d_i:=\lceil \log_q R_i\rceil,
\qquad
c_i:=\left\lceil\frac{\deg_T(\Delta_i)}2\right\rceil.
\]
Let $\rho_i\in\{0,1\}$ be as in Lemma~\ref{l:cover}. Then Lemma~\ref{l:cover}
gives finite sets of centers
\[
C_{i,d_i}:=
\Bigl\{
P_2+\frac{P_1}{T^{c_i}}\sqrt{\Delta_i}:
P_1,P_2\in A,\ \deg P_j\le d_i
\Bigr\}
\]
such that
\[
\calr_i\subseteq
\bigcup_{P\in C_{i,d_i}} U_P^{\rho_i},
\]
where $U_P^{\rho_i}=
\{w\in\Omega: |w-P|_\infty\le q^{-1}\}$.
Moreover, 
$\#C_{i,d_i}=q^{2d_i+2}\le
q^4C_\varepsilon^2|D_i|_\infty^{2\varepsilon}$.

Therefore
\[
\calr_1\times\calr_2
\subseteq
\bigcup_{P\in C_{1,d_1}}
\bigcup_{P'\in C_{2,d_2}}
U_P^{\rho_1}\times U_{P'}^{\rho_2}.
\]
By the pigeonhole principle, there exist
$P\in C_{1,d_1}$ and $P'\in C_{2,d_2}$
such that
\[
\Sigma'':=
\left\{
\sigma\in\Sigma':
(z_1^\sigma,z_2^\sigma)\in
U_P^{\rho_1}\times U_{P'}^{\rho_2}
\right\}
\]
satisfies
\begin{equation}\label{e:Sigma2-lower}
\#\Sigma''
\ge
\frac{\#\Sigma'}{\#C_{1,d_1}\cdot \#C_{2,d_2}}.
\end{equation}
Consequently, \eqref{e:galois-lower} gives
\[
\#\Sigma''\gg_\varepsilon D^{1/2-5\varepsilon}.
\]
Thus the product of admissible opens
$U_P^{\rho_1}\times U_{P'}^{\rho_2}$
contains at least
$\frac{\#\Sigma'}{\#C_{1,d_1}\cdot \#C_{2,d_2}}$
quadratic lifts of Galois conjugates of $x=\boldsymbol j(z)$.

\proofstep{Local uniformisation near the chosen product}
For $i=1,2$, set
\[
r_i:=q^{-1}<1.
\]
We first observe that the centers $P$ and $P'$ belong to the fundamental
domain. Indeed, choose $\sigma\in\Sigma''$. For every $c\in F_\infty$, we have $|z_1^\sigma-c|_\infty
\ge|z_1^\sigma|_i
\ge 1>r_1\ge|z_1^\sigma-P|_\infty$. It follows from the ultrametric inequality that  $|P-c|_\infty=|z_1^\sigma-c|_\infty$.
Taking the infimum over $c\in F_\infty$, we obtain $|P|_i=|z_1^\sigma|_i$. Moreover, since $|z_1^\sigma-P|_\infty<|z_1^\sigma|_\infty$, we have $|P|_\infty=|z_1^\sigma|_\infty$. Hence $|P|_\infty=|P|_i\ge1$, so $P\in\calf$. The same argument gives $P'\in\calf$. 

Let us set
\[
Q(w):=\exp_{\ccal}(\widetilde\pi w)^{-1}.
\]
Writing $r_J(U)$ for the power series denoted by $r(Q)$ in
\eqref{jexpansion}, we have, in a neighbourhood of the cusp, \[
j(w)=-Q(w)^{-(q-1)}+r_J(Q(w)).
\]
Therefore the power series
\[
\widehat J(U):=
-\frac{U^{q-1}}{1-U^{q-1}r_J(U)}
\]
is holomorphic on a sufficiently small disk about $0$ and satisfies
\[
\widehat J(Q(w))=\frac1{j(w)}
\]
whenever $w$ lies sufficiently deep in the cusp. Fix $r^+\in \Q$ such that
$\max\{r_1,r_2\}<r^+<1$.
Fix also rational radii
$0<s<s^+$ such that $\widehat J$ is holomorphic on the closed disk
$|U|_\infty\le s^+$
and
\begin{equation}\label{e:cusp-radii}
s^+\cdot|\widetilde\pi|_\infty r^+<1.
\end{equation}
Pick $c_0>1$ sufficiently large so that the above
$Q$-expansion represents $j(w)$ whenever
$|w|_i\ge c_0$,
and so that
\[
|Q(w)|_\infty\le s
\qquad
\text{whenever}
\qquad
|w|_i\ge c_0.
\]
The second condition is possible by
Lemma~\ref{l:carlitz-cusp}.

\begin{comment}
By Lemma~\ref{l:carlitz-cusp}  we can assume that
\begin{itemize}
\item the above $Q$-expansion represents $j(w)$ whenever
$|w|_i\ge c_0$;
\item one has
$|Q(w)|_\infty\le s$
for
$|w|_i\ge c_0$.
\end{itemize}
\end{comment}
We shall call a center $P$ {\em cuspidal} if $|P|_i\ge c_0$ and {\em bounded} otherwise. We first construct the family over the cuspidal part. For
$|u|_\infty\le r^+$ and $|b|_\infty\le s^+$
we set
\[
\mathcal Q(u,b):=
\frac{b}{1+b\exp_{\ccal}(\widetilde\pi u)}.
\]
Since $\exp_{\ccal}(\widetilde\pi u)=\widetilde\pi\,\exp_A(u)$,
and since every non-zero element of $A$ has absolute value at least $1$,
the product expansion of $\exp_A$ gives
\[
|\exp_{\ccal}(\widetilde\pi u)|_\infty
=
|\widetilde\pi|_\infty |u|_\infty
\qquad
\text{for}
\quad
|u|_\infty<1.
\]
It follows from \eqref{e:cusp-radii} that
$|b\exp_{\ccal}(\widetilde\pi u)|_\infty
\le
s^+|\widetilde\pi|_\infty r^+
<1$. Hence $1+b\exp_{\ccal}(\widetilde\pi u)$ is a unit on the above closed bi-disk and
\[
|\mathcal Q(u,b)|_\infty=|b|_\infty\le s^+.
\]
Therefore
\[
\Phi^{\mathrm{cusp}}(u,b)
:=
\widehat J\bigl(\mathcal Q(u,b)\bigr)
\]
is holomorphic on this fixed bounded bi-disk.

Suppose now that $P$ is cuspidal and specialize $b=Q(P)$. By the
additivity of the Carlitz exponential we obtain
\[
\mathcal Q(u,Q(P))
=\frac{Q(P)}{1+Q(P)\exp_{\ccal}(\widetilde\pi u)}
=Q(P+u).
\]
Moreover, since $|u|_\infty\le r^+<1\le |P|_i$, we have
\[
|P+u|_i=|P|_i\ge c_0.
\]
Thus the $Q$-expansion is valid throughout the translated disk and
\begin{equation}\label{e:cusp-specialization}
\Phi^{\mathrm{cusp}}(u,Q(P))=
\frac1{j(P+u)}.
\end{equation}
We next treat the bounded part. Set
\[
\calf_{\mathrm{bd}}
:=
\{p\in\calf:|p|_i\le c_0\}.
\]
Note that the set
$\calf_{\mathrm{bd}}$ is contained in a fixed finite union of affinoid domains. Indeed, fix $\varrho\in\Q$ such that
$r^+<\varrho<1$. We may choose finitely many pairs of affinoid domains
$V_\nu\subset V_\nu^+\subset\Omega$ for $1\le\nu\le m$, such that
\[
\calf_{\mathrm{bd}}
\subseteq
\bigcup_{\nu=1}^m V_\nu,
\]
each $V_\nu$ is  compact in $V_\nu^+$, and $|p|_i\ge\varrho$ for all 
$p\in V_\nu^+$. If $p\in V_\nu^+$ and $|u|_\infty\le r^+$ then, for every $c\in F_\infty$, we have $|p-c|_\infty\ge|p|_i
\ge\varrho>r^+\ge|u|_\infty$.  Therefore
$|p+u-c|_\infty=|p-c|_\infty$,
and in particular $p+u\in\Omega$. We may consequently define
\[
\Phi_\nu^{\mathrm{bd}}(u,p):=j(p+u)
\]
which is holomorphic on the fixed bounded affinoid
\[
\{|u|_\infty\le r^+\}\times V_\nu^+.
\]
Let $G(X,Y)\in\C_\infty[X,Y]$
be the irreducible polynomial defining $C$, and set $d_X:=\deg_XG$ and $d_Y:=\deg_YG$. Consider the polynomials
\[
G^{\vee\vee}(U,V):=
U^{d_X}V^{d_Y}
G\left(\frac1U,\frac1V\right)
\]
and
\[
G^{\vee 0}(U,Y)
:=
U^{d_X}
G\left(\frac1U,Y\right)
\qquad
\text{and}
\qquad
G^{0\vee}(X,V)
:=
V^{d_Y}
G\left(X,\frac1V\right).
\]
We now define the required rigid analytic family. There are four cases.

If both $P$ and $P'$ are cuspidal, let $\caly$ be the bi-disk defined by
\[
|b|_\infty\le s
\qquad
\text{and}
\qquad
|b'|_\infty\le s
\]
and define $\calx\rightarrow\caly$ by
\[
G^{\vee\vee}\left(
\Phi^{\mathrm{cusp}}(x,b),
\Phi^{\mathrm{cusp}}(y,b')
\right)=0.
\]
If $P$ is cuspidal and $P'$ is bounded, choose $\nu'$ such that
$P'\in V_{\nu'}$. Let
\[
\caly=
\{|b|_\infty\le s\}\times V_{\nu'}
\]
and define $\calx\rightarrow\caly$ by
\[
G^{\vee 0}\left(
\Phi^{\mathrm{cusp}}(x,b),
\Phi_{\nu'}^{\mathrm{bd}}(y,p')
\right)=0.
\]
If $P$ is bounded and $P'$ is cuspidal, choose $\nu$ such that
$P\in V_\nu$. Let
\[
\caly
=
V_\nu\times\{|b'|_\infty\le s\}
\]
and define $\calx\rightarrow\caly$ by
\[
G^{0\vee}\left(
\Phi_\nu^{\mathrm{bd}}(x,p),
\Phi^{\mathrm{cusp}}(y,b')
\right)=0.
\]
Finally, if both $P$ and $P'$ are bounded, choose $\nu,\nu'$ such that $P\in V_\nu$ and $P'\in V_{\nu'}$.
Let
\[
\caly=V_\nu\times V_{\nu'}
\]
and define $\calx\rightarrow\caly$ by
\[
G\left(
\Phi_\nu^{\mathrm{bd}}(x,p),
\Phi_{\nu'}^{\mathrm{bd}}(y,p')
\right)=0.
\]
In every case, the fiber variables satisfy
\[
|x|_\infty\le r_1
\qquad
\text{and}
\qquad
|y|_\infty\le r_2.
\]
The defining functions extend to the larger domains defined by $|x|_\infty,|y|_\infty\le r^+$ and
$|b|_\infty,|b'|_\infty\le s^+$,
and to the larger affinoids $V_\nu^+$ and $V_{\nu'}^+$. For Theorem~\ref{t:blockcounting}, use the presentation
coordinates $(Tx,Ty)$, retaining the counting tuple
$\boldsymbol f=(x,y)$.
The original disks form the unit bidisk in these
presentation coordinates, and the defining functions
extend to radius $qr^+>1$.
We apply the theorem to this extension and restrict
the resulting block images to the original disks.
Throughout, $x,y$ remain the unscaled coordinates. Since only finitely many families occur, all counting constants may be chosen uniformly. 
Let $\alpha\in\caly$ be the parameter obtained by specializing
\[
b=Q(P)
\quad\text{or}\quad
p=P,
\]
according as $P$ is cuspidal or bounded, and similarly $b'=Q(P')$ or $p'=P'$
in the second coordinate. In the cuspidal coordinates, the specialized functions in
\eqref{e:cusp-specialization} are non-zero. Hence clearing denominators
does not introduce any additional points in the specialized fiber. We
therefore have
\[
\calx_\alpha
=
\left\{
(x,y):
G\bigl(j(P+x),j(P'+y)\bigr)=0,\ 
|x|_\infty\le r_1,\ 
|y|_\infty\le r_2
\right\}.
\]

For every $\sigma\in\Sigma''$, set
\[
x^{(\sigma)}:=z_1^\sigma-P,
\qquad
y^{(\sigma)}:=z_2^\sigma-P'.
\]
Then $|x^{(\sigma)}|_\infty\le r_1$,
and $|y^{(\sigma)}|_\infty\le r_2$ and
\[
\bigl(x^{(\sigma)},y^{(\sigma)}\bigr)\in\calx_\alpha.
\]

\begin{comment}[A better way?]
What we lose is that the heights of $x$ and $z$ are equivalent only if $H(P)$ is small, but this is enough for us. I guess there should be a slightly more cumbersome way to do it in a way that they are always compatible? Or not?

Formally, we look for a parameterization of the total space of family $\{B(r,p)\}$, $|p|_i>1$, by an analytic map $F$ with  domain $U\subset \C_\infty)^2_{x,\lambda}$ of definition, such that
 $F(x,\lambda)$ has  height equivalent to the height of $x$ independently of  $\lambda$. This implies that  $F(x)$ doesn't depend on $\lambda$. As $F(x,\lambda)$ eventually covers the whole $\{|z|_i>1\}$, this means that $F(x)$ should be essentially $F(x)=1/x$.
 It is reasonable to take $\lambda=Q(p)$.
Then the domain $U$ should be like $|x-\lambda|<r|\lambda|^2$. 
Not clear....
    
\end{comment}

{ 

}
\proofstep{The counted points and their heights}
Write
\[
\calx=\Sp\calb.
\]
Following the notation of Theorem~\ref{t:blockcounting}, we take
\[
\boldsymbol f=(f_1,f_2):=(x,y)\in\calb^2.
\]

We now choose the local field over which we count. Fix a finite extension
$F_0/F_\infty$ containing $\widetilde\pi$, the coefficient fields of the
finitely many charts, and the coefficients of the fixed families and chosen $Q$-expansions, independently of the CM point and the centers. Set
\[
F_\alpha:=F_0K_{1,\infty}K_{2,\infty}.
\]
For a cuspidal center $P\in K_i$, the Carlitz exponential converges in the complete field $F_0K_{i,\infty}$, so $Q(P)$ belongs to this field. A bounded center already belongs to $K_{i,\infty}$.
Thus the classical point $\alpha\in\caly$ is defined over $F_\alpha$.
Moreover,
\[
[F_\alpha:F_0]\le 4,
\]
Taking $F_{\mathrm{loc}}=F_\alpha$, the ramification parameter
$r_\beta$ in Theorem~\ref{t:blockcounting} is therefore uniformly bounded.

For each $\sigma\in\Sigma''$, let
\[
\alpha_\sigma'\in\calx_\alpha
\]
be the classical point with coordinates
\[
\alpha_\sigma'{}^\#(x)=x^{(\sigma)}
=
z_1^\sigma-P
\qquad
\text{and}
\qquad
\alpha_\sigma'{}^\#(y)=y^{(\sigma)}
=
z_2^\sigma-P'.
\]
By construction,
\[
\bigl(x^{(\sigma)},y^{(\sigma)}\bigr)
\in(K_1K_2)^2\subset F_\alpha^2,
\qquad [K_1K_2:F]\le4.
\]
Thus these points are counted by the set
\[
\calx_\alpha(K_1K_2,H;\boldsymbol f)
\]
once $H$ is chosen large enough.

We next record the required height bound. By the standard height inequalities,
\[
H\bigl(x^{(\sigma)}\bigr)
=
H(z_1^\sigma-P)
\ll
H(z_1^\sigma)\,H(P)
\qquad
\text{and}
\qquad
H\bigl(y^{(\sigma)}\bigr)
=
H(z_2^\sigma-P')
\ll
H(z_2^\sigma)\,H(P')\ .
\]

The construction of the sets $C_{i,d_i}$ gives polynomial bounds for
$H(P)$ and $H(P')$ in terms of the discriminants, while
Lemma~\ref{l:heightfundamentaldomain} gives
\[
H(z_i^\sigma)^2\le |D_i|_\infty
\qquad 
\text{for}
\quad
i=1,2.
\]
Therefore, if
$D:=\max\{|D_1|_\infty,|D_2|_\infty\}$,
then there exists a constant $A>0$, independent of the CM point, such that
\[
H(\alpha_\sigma';\boldsymbol f)\le D^A
\qquad
\text{for every }\sigma\in\Sigma''.
\]
Fix $H:=D^A$.
Thus the points $\alpha_\sigma'$, with $\sigma\in\Sigma''$, all belong to
$\calx_\alpha(K_1K_2,H;\boldsymbol f)$.

\proofstep{Application of block counting}
We now apply Theorem~\ref{t:blockcounting}  to the family
\[
\varphi\colon\calx\rightarrow\caly
\]
with the tuple $\boldsymbol f=(x,y)$. For every $\eta>0$, the theorem gives
a transformation diagram and a finite set
$\mathcal J=\mathcal J(\alpha,H)$
of pairs $(w_j,\alpha_j)$, with
\[
\#\mathcal J\ll_\eta H^\eta,
\]
where the implicit constant is independent of the chosen centers $P,P'$, such
that
\[
\calx_\alpha(K_1K_2,H;\boldsymbol f)
\subseteq
\bigcup_{j\in\mathcal J}
(\calx_{w_j})_{\alpha_j}
(K_1K_2,H;\boldsymbol f)
\subseteq
\calx_\alpha.
\]
Here each $w_j$ is a leaf of the transformation diagram, each
$\alpha_j$ is a classical point lying over $\alpha$, and the fibers
\[
(\calx_{w_j})_{\alpha_j}
\]
are algebraic. Moreover, by Theorem~\ref{t:blockcounting}, their maps into
$\calx_\alpha$ are closed immersions.

The points $\alpha_\sigma'$, for $\sigma\in\Sigma''$, are distinct.
Indeed, since $P$ and $P'$ are fixed, the pair
$\bigl(x^{(\sigma)},y^{(\sigma)}\bigr)$ determines
\[
\bigl(z_1^\sigma,z_2^\sigma\bigr)
=\bigl(P+x^{(\sigma)},P'+y^{(\sigma)}\bigr)
\]
and hence determines $(j_1^\sigma,j_2^\sigma)$, which determines the embedding $\sigma$ of $L=E(j_1,j_2)$ over $E$.

Suppose that every algebraic fiber
$(\calx_{w_j})_{\alpha_j}$ for $j\in\mathcal J$,
meeting the set of points
\[
\big\{\alpha_\sigma':\sigma\in\Sigma''\big\}
\]
were zero dimensional. Since the transformation diagram is fixed, the algebraic families
$(\calx_w\rightarrow\caly_w)$ have uniformly bounded degree. Hence there is a
constant $\kappa$, depending only on the diagram, such that every zero dimensional
fiber $(\calx_{w_j})_{\alpha_j}$ has at most $\kappa$ geometric points. Therefore
each such fiber contains at most $O(1)$ of our counted points. Therefore
\[
\#\Sigma''\ll_\eta H^\eta.
\]
Choose $0<\varepsilon<1/10$ in the preceding estimates, and
then $\eta>0$ such that $A\eta<1/2-5\varepsilon$.
Since $H=D^A$, the bound $\#\Sigma''\ll_\eta H^\eta$
contradicts $\#\Sigma''\gg_\varepsilon D^{1/2-5\varepsilon}$
for sufficiently large $D$. Hence one of the algebraic fibers produced by
Theorem~\ref{t:blockcounting} is positive dimensional.
\begin{comment}
Suppose that every algebraic fiber
$(\calx_{w_j})_{\alpha_j}$
meeting the set of points $\{\alpha_\sigma':\sigma\in\Sigma''\}$ were
zero dimensional. Since the transformation diagram is fixed, such
zero dimensional fibers contain only $O(1)$ counted points, uniformly in
$\alpha$ and $H$. Therefore
\[
\#\Sigma''\ll_\eta H^\eta.
\]
For $\eta>0$ sufficiently small, this contradicts the lower bound
\eqref{e:Sigma2-lower}, after shrinking the previous $\varepsilon>0$ if
necessary. Hence one of the algebraic fibers produced by
Theorem~\ref{t:blockcounting} is positive dimensional.
\end{comment}

\proofstep{The algebraic block and the Ax--Lindemann conclusion}
Let
\[
\mathscr B\subseteq\calx_\alpha
\]
be, after base change to $\C_\infty$, a reduced irreducible positive dimensional component of one of the block images in $\calx_\alpha$. Consider the algebraic automorphism
\[
\theta\colon\A^2_{\C_\infty}\rightarrow\A^2_{\C_\infty},
\qquad
\theta(x,y):=(P+x,P'+y).
\] By the definition of the fiber $\calx_\alpha$, the image $\theta(\mathscr B)$
is contained in the set cut out by
\[
G(j(z),j(z'))=0.
\]
%Thus it gives a positive dimensional irreducible algebraic subvariety contained
%in $\boldsymbol j^{-1}(C)$.
%Let $Z$ be a maximal irreducible algebraic subvariety of $\boldsymbol j^{-1}(C)$ containing it.
{Since $\theta$ is an algebraic automorphism,
$\dim\theta(\mathscr B)=\dim\mathscr B>0$.
As $C$ is a curve and $\boldsymbol j$ has discrete fibers,
this dimension is $1$.
Let $\overline Z$ be the Zariski closure of
$\theta(\mathscr B)$ in
$\p^1_{\C_\infty}\times\p^1_{\C_\infty}$, and set $Z:=\overline Z^{\mathrm{an}}\cap\Omega^2$.
By algebraicity of the block, $\dim\overline Z=1$.
The observation preceding Theorem~\ref{p:hyperbolicAL}
shows that $Z$ is irreducible.
Moreover, $\theta(\mathscr B)$ contains a nonempty
admissible open of $Z$, by
\cite[Corollaries~2.2.7 and~2.2.9]{conrad}.
The analytic function $G\circ\boldsymbol j$ vanishes on
this open, hence on all of $Z$ by
\cite[Lemma~2.2.3]{conrad}.
Thus $Z\subseteq\boldsymbol j^{-1}(C)$.
Since this preimage is one dimensional, $Z$ is maximal
among its irreducible algebraic subvarieties.}

By Theorem~\ref{p:hyperbolicAL}, the image
$\boldsymbol j(Z)$ is weakly special. Since
$\boldsymbol j(Z)\subseteq C$ and $C$ is an irreducible curve, it follows that $C$ is weakly special.
Finally, $C$ contains a CM point, so this weakly special curve is special.
This contradicts our assumption that $C$ is not special.
\end{proof}

\begin{rmk}
The use of the Taguchi height in the preceding proof can be replaced by an
equidistribution input. More precisely, the combination of Lemma~\ref{colmezbound} and Proposition~\ref{p:taguchi-weil} is used only to
show that, after discarding a controlled proportion of the Galois conjugates,
the corresponding lifts lie in a truncated part of the fundamental domain, so
that the rigid Pila--Wilkie theorem may be applied on finitely many affinoid
balls. An alternative way to obtain this would be to invoke the Duke-type equidistribution theorem for CM Drinfeld modules proved by the third
author in~\cite{fms}. 
\end{rmk}

\begin{rmk}\label{r:odd-restriction}
The restriction $q$ odd in Theorem~\ref{t:ao} comes mainly from 
Lemma~\ref{l:heightfundamentaldomain}. In the proof we use the estimate
$H(z)^2\le |D_c|_\infty$
for quadratic points $z\in\calf_K$. This is available when $q$ is odd
(and also, in characteristic $2$, when $\infty$ is inert in $K$).
However, in characteristic $2$ with $\infty$ ramified in $K$,
Lemma~\ref{l:heightfundamentaldomain} gives only
$H(z)\le |D_c|_\infty^{1/2}|\xi|_\infty$,
where $K=F(\xi)$ is given by a reduced Artin--Schreier equation. The extra
factor $|\xi|_\infty$ is not controlled uniformly by the discriminant in the
form needed in the Pila--Zannier argument. Thus the present method seems to prove the
Andr\'e--Oort theorem only under the assumption $q$ odd.
\end{rmk}

\end{document}